\documentclass[final,leqno]{siamltex}
\usepackage{booktabs} 
\usepackage[noadjust]{cite}
\usepackage{ulem}
\usepackage{amsmath}
\allowdisplaybreaks[4]
\usepackage{cancel}
\usepackage{multirow}
\usepackage{graphicx}
\usepackage{amssymb}
\usepackage{color}
\usepackage{mathtools}
\usepackage{tikz}
\usepackage{hyperref}
\hypersetup{
  bookmarks=true,
  bookmarksdepth=2,
  bookmarksnumbered=true
}
\makeatletter
\@ifundefined{UseSocket}{}{%
  \def\refstepcounter#1{%
    \UseSocket{refstepcounter}{%
      \stepcounter{#1}%
      \edef\reserved@a{#1}%
      \ifx\reserved@a\ltx@star@counter\else
        \let\@currentcounter\reserved@a
      \fi
      \protected@edef\@currentlabel{%
        \csname p@#1\expandafter\endcsname
        \csname the#1\endcsname}%
      \UseSocket{refstepcounter/target}{#1}%
      \@ifundefined{UseTaggingSocket}{}{%
        \UseTaggingSocket{recordtarget}}%
    }%
  }%
}
\makeatother
\usepackage{float}
\usepackage[toc,page]{appendix}
\usepackage{mathrsfs}
 \DeclarePairedDelimiter{\norm}{\lVert}{\rVert}
\usepackage{latexsym, bm}
\numberwithin{equation}{section}
\newtheorem{remark}{Remark}[section]

\newcommand{\vertiii}[1]{{\left\vert\kern-0.23ex\left\vert\kern-0.23ex\left\vert #1
		\right\vert\kern-0.23ex\right\vert\kern-0.23ex\right\vert}}
\renewcommand\abstractname{Abstract}

\renewcommand\figurename{Figure}
\renewcommand\tablename{Table}

\normalsize

\renewcommand\refname{References}

\def\vT{\vartheta}

\def\no{{\nonumber}}

\def\dx{\delta^{\mathrm c}}
\def\dxt{\delta_\tau^{\mathrm c}}
\def\dpred{\delta^{\mathrm p}}
\def\dtpred{\delta_\tau^{\mathrm p}}

\newcommand{\mNQ}{\mathbf{N}_{h}}
\newcommand{\mNu}{\mathcal{N}_{h}}
\newcommand{\mNQz}{\mathbf{N}}
\newcommand{\mNuz}{\mathcal{N}}
\newcommand{\muQ}{\mathbf{H}_h}

 \newcommand{\eaa}{\dtpred\be_{\mathbf{Q}}^{n,*}}
\newcommand{\eba}{\dtpred e_{u}^{n,*}}

\newcommand{\mLQ}{\mathcal{L}_{h}}
\newcommand{\mLu}{\mathcal{D}_{h}}

\newcommand{\id}[2]{\left\langle #1,#2\right\rangle_h}
\newcommand{\ia}[1]{ \norm{#1 }_{h}^2}

\newcommand{\ilQ}[1]{\norm{#1 }_{\mathcal{L}_h}^2}
\newcommand{\ilu}[1]{\norm{#1 }_{\mathcal{D}_h}^2}

\newcommand{\iaQ}[1]{\norm{#1 }^2_{\psi_{\mathcal{L}}}}
\newcommand{\ibQ}[1]{\norm{#1 }^2_{\psi_{\mathcal{L}}^1}}
\newcommand{\icQ}[1]{\norm{#1 }^2_{\psi_{\mathcal{L}}^2}}
\newcommand{\iau}[1]{\norm{#1 }^2_{\psi_{\mathcal{D}}}}
\newcommand{\ibu}[1]{\norm{#1 }^2_{\psi_{\mathcal{D}}^1}}

 \newcommand{\mQ}{\psi(\tau \mLQ)}

\newcommand{\idQ}[1]{\norm{#1 }^2_{\psi_{\mathcal{L}}^3}}

\newcommand{\be}{\mathbf{e}}

\newcommand{\ep}{\epsilon}

\newcommand{\bM}{\mathbf{M}}

 \newcommand{\ea}{\delta_\tau \be_{\mathbf{Q}}^{n+1}}
\newcommand{\eb}{\delta_\tau e_{u}^{n+1}}
\newcommand{\ec}{\delta_\tau e_{s}^{n+1}}

 \newcommand{\da}{\delta_\tau \mathbf{Q}(t_{n+1})}
 \newcommand{\db}{\delta_\tau u(t_{n+1})}
\def\ba{\mathbf{a}}

\def\A{\alpha}
\def\B{\beta}
\def\C{\gamma}

\def\bu{u}

\def\bQ{\mathbf{Q}_h}
\def\bu{u_h}

\newcommand{\mq}{\psi}

\def\ma{\mathcal{A}}

\def\ts{{s}_h}
\def\bX{\mathbf{X}}

 \def\ba{\begin{equation}\begin{aligned}}
 \def\ed{\end{aligned}\end{equation}}

\begin{document}
\renewcommand{\thefootnote}{\fnsymbol{footnote}}
\title{Energy stability and error estimates for a second-order structure-preserving exponential integrator method for smectic-A liquid crystals}

\author{Wenshuai Hu\thanks{Laboratory of Mathematics and Complex Systems, Ministry of Education and School of Mathematical Sciences,
        Beijing Normal University, Beijing 100875, China.\newline
        \texttt{202431130052@mail.bnu.edu.cn}}
    \and
    Guanghua Ji\thanks{Laboratory of Mathematics and Complex Systems, Ministry of Education and School of Mathematical Sciences,
        Beijing Normal University, Beijing 100875, China.\newline
        \texttt{ghji@bnu.edu.cn}; corresponding author}
    \and Xiao Li\thanks{Laboratory of Mathematics and Complex Systems, Ministry of Education and
    School of Mathematical Sciences, Beijing Normal University, Beijing 100875, China.\newline
        \texttt{lixiao@bnu.edu.cn}}
}

\maketitle
\begin{abstract}
In this work, we develop a second-order, linear, decoupled, and
structure-preserving numerical scheme for the modified Landau--de Gennes
model of smectic-A (SmA) liquid crystals. The main contributions are
threefold.
First, to the best of our knowledge, we propose the first integration of the generalized scalar auxiliary variable (GSAV) approach with a second-order exponential time-differencing Runge--Kutta (ETDRK2) discretization, leading to a second-order GSAV--ETD2 scheme.
Second, we prove that the proposed scheme satisfies an unconditional energy-dissipation law, thereby closing the theoretical gap in the energy-stability analysis of second-order GSAV exponential integrators of this class.
Third, by deriving a coercive discrete reformulation, we establish a fully discrete error estimate without imposing any coupling condition between $\tau$ and $h$, achieving the optimal convergence rate $\mathcal{O}(\tau^2+h^2)$.
   Numerical experiments  are presented to verify our theoretical
results  and to simulate  the self-assembly dynamics of the SmA phase.
\end{abstract}


\pagestyle{myheadings}
\thispagestyle{plain}
\markboth{W. Hu,
    G. Ji \and X. Li}{GSAV-ETD2 scheme for Smectic-A liquid crystals}
\section{Introduction}
Liquid crystals are mesophases that exhibit properties intermediate between those of conventional solids and liquids 
\cite{han2015microscopic,liu2007dynamics}.
Depending on the degree of orientational and positional order, they can form nematic, cholesteric, and smectic phases
\cite{ball2011orientability,wang2021modelling,biscari2007landau}.
Among smectic phases, Smectic-A (SmA) liquid crystals are characterized by a dual-order structure: the molecules form one-dimensional density layers, and the orientational director is aligned with the layer normal \cite{xia2021structural,xia2023variational,E1997modeling}. Although the classical Landau-de Gennes (LdG) theory has been highly successful in describing isotropic-nematic transitions and nematic defect structures \cite{fei2018isotropic,majumdar2010equilibrium}, it does not directly capture the positional ordering that is essential for the SmA phase \cite{shi2025modified}. The modified Landau--de Gennes (mLdG) theory addresses this limitation by coupling the orientational Q-tensor with a scalar positional order parameter \cite{Ball03052015,xia2021structural,xia2024simple}, thereby providing a thermodynamically consistent framework for the isotropic-nematic-smectic transition and resulting defect dynamics \cite{xia2023variational,shi2025modified}.

For the SmA liquid crystal system considered in this paper, the mLdG energy proposed by Xia et al. \cite{xia2021structural,xia2023variational} is given by
\begin{equation}
    E[\mathbf{Q}, u]
    = \int_{\Omega}
    \left(
    \frac{K}{2}|\nabla\mathbf Q|^2+
    f_{\mathrm{bn}}(\mathbf Q)+
    f_{\mathrm{bs}}(u)
    + f_{\mathrm{int}}(\mathbf Q,u)
    \right)\,\mathrm d\mathbf x,
\end{equation}
where $\Omega\subset\mathbb{R}^d$ ($d=2,3$) is a  bounded domain, $\mathbf Q$ is a symmetric and traceless $d\times d$ tensor order
parameter, and $u$ is a scalar positional order parameter.
Here,
$|\nabla \mathbf{Q}|^2=\sum_{i,j,k}(\partial_i Q_{jk})^2$ and $K>0$ is the elastic coefficient.

The bulk energy densities $f_{\mathrm{bn}}(\mathbf Q)$ and $f_{\mathrm{bs}}(u)$
from the LdG theory \cite{deGennes1974} and the Landau phase transition theory \cite{pevnyi2014modeling,izzo2020landau}, respectively,
are given by
\begin{align}
f_{\mathrm{bn}}(\mathbf Q)
&=
\frac{\A}{2}\operatorname{tr}(\mathbf Q^2)
-\frac{\B}{3}\operatorname{tr}(\mathbf Q^3)
+\frac{\C}{4}\bigl(\operatorname{tr}(\mathbf Q^2)\bigr)^2,
\\
f_{\mathrm{bs}}(u)
&=
\frac{a}{2}u^2+\frac{b}{3}u^3+\frac{c}{4}u^4,
\end{align}
where $\operatorname{tr}(\Phi)=\sum_i\Phi_{ii}$ for any
$\Phi\in\mathbb{R}^{d\times d}$.
The constants $\A$ and $a$
are temperature-dependent coefficients, while $\B\ge0$,
$\C>0$, $b\in\mathbb R$, and $c>0$ are material-dependent bulk constants, with
$b=0$ for symmetric molecules \cite{pevnyi2014modeling}.

Following the mLdG theory \cite{Ball03052015,shi2025modified,xia2021structural}, the coupling term $f_{\mathrm{int}}(\mathbf{Q},u)$  is given by
%
\begin{small}
    \begin{align*}
    f_{\mathrm{int}}(\mathbf{Q}, u) & = \begin{cases}
                                    B_0 \left| D^2 u + q^2 \left( \frac{\mathbf{Q}}{s_+} + \frac{\mathbf{I}_d}{d} \right) u \right|_{F}^2,                                                                                    & (d=3,\ \A < \frac{\B^2}{27\C})
~\text{or}~
(d=2,\ \A < 0), \\
                                   B_0 |D^2 u|_{F}^2, & \text{otherwise},
                               \end{cases}\\
                   s_+=&\dfrac{\B+\sqrt{\B^2-24\A\C}}{4\C}~(d=3,\ \A<\frac{\B^2}{27\C}),\quad
s_+=\sqrt{-2\A/\C}~(d=2,\ \A<0),
\end{align*}
\end{small}
where $|\Phi|_F=\sqrt{\Phi : \Phi}=\sqrt{\sum_{i,j} \Phi_{ij}^2}$ is the Frobenius norm for any  $\Phi$ in $\mathbb{R}^{d\times d}$,
$B_0 > 0$ is the coupling strength, $q = 2\pi/l$ is the wave number for the smectic layer thickness $l$, $\mathbf{I}_d$ is the $d \times d$ identity matrix, and $D^2 u$  denotes the Hessian matrix of $u$.

When $(d=3,\ \A < \frac{\B^2}{27\C})$
or
$(d=2,\ \A < 0)$, the constrained $L^2$-gradient flow associated with $E[\mathbf{Q},u]$ with periodic boundary conditions   can be written explicitly as
    \begin{equation}\label{eq1_9}
    \begin{aligned}
        \frac{\partial \mathbf{Q}}{\partial t} & = K\Delta\mathbf{Q} -  \A\mathbf{Q} + \B \left( \mathbf{Q}^2 - \frac{\operatorname{tr}(\mathbf{Q}^2)}{d}\mathbf{I}_d \right) - \C \operatorname{tr}(\mathbf{Q}^2)\mathbf{Q}  \\
                                               & \quad -  \frac{2B_0q^2}{s_+}   \left( u D^2 u - \frac{\operatorname{tr}(u D^2 u)}{d}\mathbf{I}_d \right) -  \frac{2B_0q^4}{s_+^2} \mathbf{Q}u^2,                   \\
        \frac{\partial u}{\partial t}          & = - 2B_0\Delta^2 u - au - bu^2 - cu^3  - 2B_0 D^2 u : (q^2\bM )                                                                                                              \\
                                               & \quad - 2B_0 q^2\nabla \cdot \left( \nabla \cdot \left( \bM u \right) \right)  - 2B_0  \left| q^2 \bM \right|_F^2 u,
    \end{aligned}
\end{equation}
where $\bM=\mathbf M(\mathbf Q)
:=\frac{\mathbf{Q}}{s_+} + \frac{\mathbf{I}_d}{d}$.
The decoupled case is recovered by setting $q=0$.

Building upon the LdG energy, the mathematical analysis and numerical simulation of nematic liquid crystals have been extensively studied over the past few decades \cite{izzo2020landau,pevnyi2014modeling,chen1976landau,huangGlobalWellposednessDynamical2015,huang2024errorestimateorderenergy,Ji2020BDF2}.
Recently, the formulation and analysis of the  mLdG theory have been systematically advanced \cite{Ball03052015,xia2021structural}.  Xia et al. \cite{xia2023variational}   established the existence of global minimizers for the mLdG energy  and derived a priori error estimates for its FEM scheme in
the decoupled case ($q=0$).
 Furthermore, Shi et al. \cite{shi2025modified} established the existence of global weak solutions to this system and showed analytically that the mLdG model captures the isotropic--nematic--smectic (I--N--S) phase transition as a function of temperature. These results motivate the development of an energy-stable, decoupled numerical scheme for \eqref{eq1_9}.

Energy-stable time discretizations have become a central tool for the numerical simulation of gradient flow systems. Representative approaches include convex splitting \cite{eyre1998unconditionally,baskaran2013convergence}, stabilized implicit-explicit (IMEX) schemes \cite{xu2006stability,shen2010numerical,feng2013stabilized}, discrete gradient methods \cite{du1991numerical,furihata2001stable}, exponential time differencing (ETD) \cite{du2018stabilized,du2019,du2021,liu2025maximum}, invariant energy quadratization (IEQ) \cite{yang2017linearly,xu2019efficient,yang2020convergence}, and scalar auxiliary variable (SAV) methods \cite{jiang2022improving,shen2019new,shen2018scalar}. ETD methods are attractive because they integrate stiff linear operators exactly, while SAV-type methods provide a flexible mechanism for constructing linear and unconditionally energy-stable schemes for nonlinear gradient flows \cite{du2021,shen2019new}.

Motivated by these complementary advantages, Ju et al.~\cite{ju2022generalized} combined the GSAV approach with exponential integration and proposed a GSAV--EI framework for Allen--Cahn-type gradient flows.
However, owing to the $h$-dependent bounds arising from the matrix exponential,
the analysis in \cite{ju2022generalized} was limited to temporal error estimates
with the spatial mesh fixed.  Building on this work, Hu et al.~\cite{hu2026structurepreservinggsavexponentialintegrator}
introduced a coercive discrete reformulation and established a fully discrete
error analysis for GSAV--EI schemes under independent temporal and spatial
refinement. Nevertheless, the extension of the GSAV--EI framework to second-order exponential time-differencing Runge--Kutta discretizations remains unexplored, and the corresponding energy stability and fully discrete error analysis have yet to be established.

In this paper, we develop a second-order, linear, decoupled, and
structure-preserving GSAV exponential time-differencing Runge--Kutta
(GSAV--ETD2) scheme for the mLdG model of SmA liquid crystals
\eqref{eq1_9}.
The main contributions of this work are summarized as follows:
\begin{itemize}
    \item For the first time, we combine the GSAV approach with an ETDRK2 discretization
and develop a linear and decoupled GSAV--ETD2 finite-difference scheme for the
fully coupled mLdG model of SmA liquid crystals.

    \item We establish an unconditional modified discrete energy dissipation law for the proposed scheme, thereby closing the gap in the energy-stability analysis of this class of schemes.


    \item By deriving an equivalent coercive predictor--corrector reformulation of the GSAV--ETD2 scheme, we establish the optimal fully discrete error estimate
$\mathcal{O}(\tau^2+h^2)$ under independent refinement of $\tau$ and $h$,
thereby demonstrating the viability of integrating the GSAV approach with an ETDRK2 discretization.
\end{itemize}

The rest of this paper is organized as follows.
In Section~\ref{se2}, we
construct the fully discrete GSAV-ETD2 finite difference scheme for the SmA model and prove its unconditional modified-energy stability.
In Section~\ref{se3}, we establish optimal fully discrete error estimates by a bootstrap argument based on higher-order regularity bounds for the numerical solution.
Section~\ref{section6} reports convergence tests and two- and
three-dimensional simulations of layer selection, confinement, and defect
dynamics. We conclude with some
perspectives in Section~\ref{section7}.
\section{Fully discrete GSAV-ETD2 scheme and energy stability analysis}\label{se2}
We first introduce some notation.
The space $\mathcal{S}^{(d)}$ is defined as the set of $d \times d$ symmetric and traceless matrices:
\begin{equation}
    \mathcal{S}^{(d)} \overset{\mathrm{def}}{=} \left\{ \bX \in \mathbb{R}^{d \times d} \;\middle|\; \operatorname{tr}(\bX) = 0, \bX^{ij} = \bX^{ji} \in \mathbb{R} \ \forall i,j = 1, \dots, d \right\}.
\end{equation}
The admissible space for $\mathbf{Q}$ is
\begin{equation*}
    \mathbf{H}^1(\Omega, \mathcal{S}^{(d)}) := \left\{ \mathbf{Q}: \Omega \to \mathcal{S}^{(d)} \mid \mathbf{Q}_{ij} \in H^1(\Omega) \text{ for } i,j = 1, \dots, d \right\}.
\end{equation*}
For $\mathbf{Q}\in \mathbf{H}^1(\Omega,\mathcal{S}^{(d)})$, we define its standard Frobenius-type
$L^2$ norm by
\begin{equation*}
    \|\mathbf{Q}\|_{L^2}^2
    :=
    \int_{\Omega} |\mathbf{Q}|_F^2\,d\mathbf{x},\quad
    \|\nabla \mathbf{Q}\|_{L^2}^2
    :=
    \int_{\Omega} |\nabla \mathbf{Q}|^2\,d\mathbf{x}.
\end{equation*}
The admissible space for $u$ is $H^2(\Omega)$.

The $L^2$ gradient-flow equations \eqref{eq1_9} imply the following energy dissipation law for the coupled system:
\begin{equation}\label{eq_dissipation}
    \frac{dE}{dt} = - \int_{\Omega} \left| \frac{\delta E}{\delta \mathbf{Q}} - \lambda \mathbf{I}_d - \nu + \nu^{T} \right|_F^2 d\mathbf{x} - \int_{\Omega} \left| \frac{\delta E}{\delta u} \right|^2 d\mathbf{x} \le 0.
\end{equation}
As a consequence of \eqref{eq_dissipation}, we obtain a priori estimates for system \eqref{eq1_9}.
\begin{lemma}[A priori estimates]\label{lem_continuous_regularity}
  Assume that the initial data
$\mathbf{Q}_0\in\mathbf{H}^1(\Omega,\mathcal{S}^{(d)})$ and
$u_0\in H^2(\Omega)$ satisfy
$E(\mathbf{Q}_0,u_0)<+\infty$. Then the energy dissipation law \eqref{eq_dissipation} and
the coercivity of $E$ imply that the nonlinear energy functional
$E_1[\mathbf{Q},u]$ remains uniformly bounded on $[0,T]$:
    \begin{align*}
          -C_*
    &\le
    E_1[\mathbf{Q},u]
    \coloneqq
    \int_{\Omega}
    2B_0(D^2u:q^2\mathbf{M}u)
    \,d\mathbf{x}
    +
    B_0
    \left\|
    q^2\mathbf{M}u
    \right\|_{L^2}^2
    \\&\qquad\qquad
    +
    \int_{\Omega}
    f_{\mathrm{bn}}(\mathbf{Q})
    +
    f_{\mathrm{bs}}(u)
    -
    \frac \kappa2(|\mathbf{Q}|_F^2 + |u|^2)
    d\mathbf{x}
    \le
    C^*,
    \end{align*}
 where $f_{\mathrm{bn}}(\mathbf{Q})=\frac{\A}{2}\operatorname{tr}(\mathbf{Q}^2)-\frac{\B}{3}\operatorname{tr}(\mathbf{Q}^3) + \frac{\C}{4}(\operatorname{tr}(\mathbf{Q}^2))^2$,
 $\kappa>0$ is a stabilization parameter, and $C_*,C^*>0$ are
constants depending only on the initial energy
$E(\mathbf{Q}_0,u_0)$, the domain $\Omega$, and the model parameters.
\end{lemma}
\begin{proof}
The result follows from a standard argument based on the energy dissipation law \eqref{eq_dissipation}; see \cite{hu2026structurepreservinggsavexponentialintegrator} for details.
\end{proof}

Let $\zeta\in C^1(\mathbb{R};(0,\infty))$ satisfy $\zeta'\ge 0$ on
$\mathbb{R}$. For a fixed stabilization parameter $\kappa>0$, we introduce the
scalar auxiliary variable
$s(t)\coloneqq E_1[\mathbf Q(t),u(t)]$
and define the modified energy and the function $g$ by
\begin{align}
\mathcal E[\mathbf Q,u,s]
&\coloneqq
\frac{K}{2}\|\nabla\mathbf Q\|_{L^2}^2
+
B_0\|\Delta u\|_{L^2}^2
+
\frac{\kappa}{2}
\left(
\|u\|_{L^2}^2
+
\|\mathbf Q\|_{L^2}^2
\right)
+s,
\\
g(\mathbf Q,u,s)
&\coloneqq
\frac{\zeta(s)}
{\zeta\bigl(E_1[\mathbf Q,u]\bigr)}.
\end{align}
Here $\|\Delta u\|_{L^2}^2=\|D^2u\|_{L^2}^2$ under periodic boundary
conditions.

For notational convenience, we define the linear operators and the
associated nonlinear terms at time $t$ as follows:
\ba
\mathcal L\mathbf Q(t)
&\coloneqq
-K\Delta\mathbf Q(t)+\kappa\mathbf Q(t),
&\qquad
\mNQz(t)
&\coloneqq
\mNQz\bigl(\mathbf Q(t),u(t)\bigr)
=
g(t)\mathbf H(t)
,
\\
\mathcal D u(t)
&\coloneqq
2B_0\Delta^2u(t)+\kappa u(t),
&\qquad
\mNuz(t)
&\coloneqq
\mNuz\bigl(\mathbf Q(t),u(t)\bigr)
=
g(t)\mu(t),
\ed
where $\mathbf H(\mathbf Q,u)$ and $\mu(\mathbf Q,u)$ are defined by
\ba\label{eqd1}
\mathbf H(\mathbf Q,u)
&:=
\kappa\mathbf Q-
\A\mathbf Q
+\B\left(
\mathbf Q^2-
\frac{\operatorname{tr}(\mathbf Q^2)}{d}\mathbf I_d
\right)
-\C\operatorname{tr}(\mathbf Q^2)\mathbf Q
-2B_0q^4\frac{\mathbf Q}{s_+^2}u^2
\\
&\quad-\frac{2B_0q^2}{s_+}
\left(
uD^2u-
\frac{\operatorname{tr}(uD^2u)}{d}\mathbf I_d
\right)
= \mathbf H_{\rm p}(\mathbf Q,u)+\mathbf H_{\rm d}(u),
\\[2pt]
\mu(\mathbf Q,u)
&:=
\kappa u-
a u-bu^2-cu^3
-2B_0|q^2\mathbf M|_F^2u
\\
&\quad-2B_0D^2u:(q^2\mathbf M)
-2B_0q^2\nabla\cdot
\left(\nabla\cdot(\mathbf M u)\right)
= \mu_{\rm p}(\mathbf Q,u)+\mu_{\rm d}(\mathbf Q,u).
\ed
Here, the subscripts $\mathrm p$ and $\mathrm d$ denote the derivative-free
and derivative-dependent contributions, respectively.

The system \eqref{eq1_9} can then be equivalently reformulated in terms of
$(\mathbf Q,u,s)$ as follows:
\begin{align}
\frac{\partial \mathbf{Q}}{\partial t}
&=
-\mathcal L\mathbf Q
+\mNQz,\label{eq1_9_continuous_1}
\\
\frac{\partial u}{\partial t}
&=
-\mathcal D u
+\mNuz,\label{eq1_9_continuous_2}
\\
\frac{d s}{dt}
&=-
\int_{\Omega}
\mNQz :\frac{\partial\mathbf{Q}}{\partial t}
\,d\mathbf{x}
-
\int_{\Omega}
\mNuz\frac{\partial u}{\partial t}
\,d\mathbf{x}.
\end{align}

\subsection{Spatial Discretization and Discrete Function Spaces}
For definiteness, we present the spatial discretization of
\eqref{eq1_9} in three dimensions on the periodic box
$\Omega=[0,L_d]^3$; the two-dimensional construction follows by
omitting the $z$-direction terms. Given a positive integer $J$, the
uniform mesh size in each spatial direction is $h=L_d/J$. All variables
are stored at the primary mesh points. We denote by $\mathbf{E}$ the set
of mesh points, defined by
\begin{equation*}
    \mathbf{E} = \left\{ (x_p, y_q, z_r) = (ph, qh, rh) \mid p, q, r = 0, 1, \dots, J \right\}.
\end{equation*}
The corresponding periodic grid function space is given by
\begin{equation*}
    E_h^{\text{per}} \!= \{ U : \mathbf{E} \to \mathbb{R} \mid U_{0,q,r} \!= U_{J,q,r}, \, U_{p,0,r} = U_{p,J,r}, \, U_{p,q,0} \!= U_{p,q,J}, 0 \le p, q, r \le J \}.
\end{equation*}
For ease of notation, we handle the periodic boundary conditions by assigning the values of the corresponding interior nodes to the ghost nodes. Specifically, for any $U \in E_h^{\text{per}}$, we have
\begin{equation*}
    U_{-1,q,r} = U_{J-1,q,r}, \quad U_{J+1,q,r} = U_{1,q,r}, \quad \forall 0 \le q, r \le J.
\end{equation*}
Analogous periodic extensions apply to the $q$ and $r$ directions.

Then, we define the forward, backward, and central difference operators for the grid function $U \in E_h^{\text{per}}$ along the $x$-direction as follows:
    \begin{align*}
    (D_1^+ U)_{p,q,r}\! =\! \frac{U_{p+1,q,r} \!-\! U_{p,q,r}}{h},
    (D_1^- U)_{p,q,r} \!= \!\frac{U_{p,q,r} \!- \!U_{p-1,q,r}}{h},
    D_1^c U\!=\!\frac{1}{2}(D_1^+ U \!+\! D_1^- U\!).
\end{align*}
where the difference operators $D_2^{\pm,c}$ and $D_3^{\pm,c}$ along the $y$- and $z$-directions are defined analogously. Moreover, we denote the mixed central difference as $D_{k,l}^c U = D_k^c D_l^c U$ for $k,l \in \{1,2,3\}$.

In a collocated grid, to maintain a compact stencil for the Laplacian, we define the discrete gradient operator $\nabla_h : E_h^{\text{per}} \to (E_h^{\text{per}})^3$ using forward differences:
\begin{equation*}
    (\nabla_h U)_{p,q,r} = ( (D_1^+ U)_{p,q,r}, (D_2^+ U)_{p,q,r}, (D_3^+ U)_{p,q,r} )^T,
\end{equation*}
and the corresponding discrete divergence operator $\nabla_h \cdot : (E_h^{\text{per}})^3 \to E_h^{\text{per}}$ using backward differences:
\begin{equation*}
    \nabla_h \cdot (U^{(1)}, U^{(2)}, U^{(3)})^T = D_1^- U^{(1)} + D_2^- U^{(2)} + D_3^- U^{(3)}.
\end{equation*}
In addition, for any scalar-valued grid function
$U\in E_h^{\mathrm{per}}$, we define its discrete Hessian as the
matrix-valued grid function
\begin{equation*}
    D_h^2U
    :=
    \left(D_{i,j}^cU\right)_{i,j=1}^3,
    \qquad
    D_{i,j}^cU:=D_i^cD_j^cU.
\end{equation*}
Because all variables are collocated, the discrete inner products no longer require averaging operators. They are simply given by:
\begin{align*}
    \langle U, V \rangle_h &= h^3 \sum_{p,q,r=1}^{J} U_{p,q,r} V_{p,q,r} \quad \forall U, V \in E_h^{\text{per}},\\
    [\mathbf{U}, \mathbf{V}]_h &= \sum_{k=1}^3 \langle U^{(k)}, V^{(k)} \rangle_h \quad \forall \mathbf{U}, \mathbf{V} \in (E_h^{\text{per}})^3.
\end{align*}
Then, for any $U \in E_h^{\text{per}}$, the corresponding discrete
$L^2$, $L^\infty$, $H^1$, and $H^2$ norms are defined by
    \begin{align*}
    \|U\|_h^2       & = \langle U, U \rangle_h, \quad
    \|\nabla_h U\|_h^2 = [\nabla_h U, \nabla_h U]_h ,\quad
        \|U\|_{H_h^1}^2  = \|U\|_h^2 + \|\nabla_h U\|_h^2, \\
    \|U\|_{H_h^2}^2 &= \!\|U\|_{H_h^1}^2 + \|\Delta_h U\|_h^2,
    \|U\|_{H_h^4}^2 \!= \!\|U\|_{H_h^2}^2 + \|\Delta_h^2 U\|_h^2,
    \|U\|_\infty \!= \!\max_{0 \le p,q,r \le J} |U_{p,q,r}|.
\end{align*}
The discrete Laplacian is defined as $\Delta_h U = \sum_{k=1}^3 D_k^+ D_k^- U$, and the discrete biharmonic operator is defined as $\Delta_h^2 U = \Delta_h (\Delta_h U)$.
Using summation-by-parts on the collocated grid,  for any $U, V \in E_h^{\text{per}}$, we have
\begin{equation*}
    \begin{aligned}
        \langle D_k^- D_k^+ U, V \rangle_h & = - \langle D_k^+ U, D_k^+ V \rangle_h = \langle U, D_k^- D_k^+ V \rangle_h,                                      \\
        \langle D_{k,l}^c U, V \rangle_h   & = - \langle D_l^c U, D_k^c V \rangle_h = - \langle D_k^c U, D_l^c V \rangle_h = \langle U, D_{k,l}^c V \rangle_h,
    \end{aligned}
\end{equation*}
which also implies that
\begin{equation}\label{eq_c1}
    \langle -\Delta_h U, V \rangle_h =  [\nabla_h U, \nabla_h V]_h = \langle U, -\Delta_h V \rangle_h,\quad
    \langle \Delta_h^2 U, V \rangle_h = \langle U, \Delta_h^2 V \rangle_h.
\end{equation}

For tensor-valued grid functions, we first introduce the general
matrix-valued grid function space
\begin{equation*}
    \mathbb{M}_h^{(3)}
    :=
    \left\{
    \bm{\Phi}_h=(\Phi_h^{ij})_{i,j=1}^3
    \ \middle|\
    \Phi_h^{ij}\in E_h^{\mathrm{per}},
    \quad i,j=1,2,3
    \right\}.
\end{equation*}
The discrete function space for symmetric and traceless
$\mathbf{Q}$-tensor fields is then defined as
\begin{equation*}
    \mathcal{S}_h^{(3)}
    :=
    \left\{
    \bm{\Phi}_h\in\mathbb{M}_h^{(3)}
    \ \middle|\
    \bm{\Phi}_h^{T}=\bm{\Phi}_h,\quad
    \operatorname{tr}(\bm{\Phi}_h)=0
    \right\}.
\end{equation*}
The discrete Frobenius product  and the discrete $L^4$ norm for any $\bm{\Phi}_h, \bm{\Psi}_h \in  \mathbb{M}_h^{(3)}$ are defined as follows:
\begin{equation*}
    \langle \bm{\Phi}_h, \bm{\Psi}_h \rangle_h =  \sum_{i,j=1}^3 \langle {\Phi}_h^{i,j}, {\Psi}_h^{i,j} \rangle_h,\quad  \|\bm{\Phi}_h\|_{L_h^4}^4 = \big\langle |\bm{\Phi}_h|_F^2, |\bm{\Phi}_h|_F^2  \big\rangle_h.
\end{equation*}
For any $\bm{\Phi}_h \in  \mathbb{M}_h^{(3)}$, the discrete $L^2$,
$L^\infty$, $H^1$, and $H^2$ norms are defined by
\begin{equation*}
\begin{aligned}
\|\bm{\Phi}_h\|_h^2 &:= {\sum_{i,j=1}^3 \| {\Phi}_h^{i,j}\|_h^2}, \quad
\|\nabla_h \bm{\Phi}_h\|_h^2 := {\sum_{i,j=1}^3 \| \nabla_h{\Phi}_h^{i,j}\|_h^2},\\ \|\Delta_h \bm{\Phi}_h\|_h^2 &:= {\sum_{i,j=1}^3 \|\Delta_h {\Phi}_h^{i,j}\|_h^2},\quad
\|\bm{\Phi}_h\|_\infty := \max_{0 \le p,q,r \le J} |(\bm{\Phi}_{h})_{p,q,r}|_F    , \\
\|\bm{\Phi}_h\|_{H_h^1}^2 &:= {\|\bm{\Phi}_h\|_h^2 + \|\nabla_h \bm{\Phi}_h\|_h^2}
, \quad
\|\bm{\Phi}_h\|_{H_h^2}^2 := {\|\bm{\Phi}_h\|_{H_h^1}^2 + \|\Delta_h \bm{\Phi}_h\|_h^2}.
\end{aligned}
\end{equation*}

For any $\bm{\Phi}_h\in\mathbb M_h^{(3)}$, we define the discrete
central double-divergence operator below, which is the discrete adjoint of $D_h^2$:
\begin{equation}\label{eq_discrete_hessian_adjoint}
\begin{aligned}
\operatorname{div}_{h,c}^{\,2}\bm{\Phi}_h
:=
\sum_{i,j=1}^3D_i^cD_j^c\Phi_h^{ij},\quad
\left\langle D_h^2U,\bm{\Phi}_h\right\rangle_h
=
\left\langle
U,\operatorname{div}_{h,c}^{\,2}\bm{\Phi}_h
\right\rangle_h,
\qquad U\in E_h^{\mathrm{per}}.
\end{aligned}
\end{equation}
Moreover, by the discrete Fourier transform and Parseval's identity,
we have
\ba\label{eq_discrete_hessian_bound}
    \|D_h^2U\|_h
    \le
    \|\Delta_hU\|_h,\quad  \norm{ \operatorname{div}_{h,c}^{\,2}\bm{\Phi}_h }_h^2 \le  \norm{ \Delta_h \bm{\Phi}_h }_h^2,
    \qquad U\in E_h^{\mathrm{per}},~\bm{\Phi}_h\in\mathbb M_h^{(3)}.
\ed
Finally, we denote by
$\mathbf H_h^2(\Omega,\mathbb M_h^{(3)})$ and $H_h^4(\Omega)$
the spaces $\mathbb M_h^{(3)}$ and $E_h^{\mathrm{per}}$ endowed with
the discrete $H_h^2$- and $H_h^4$-norms, respectively. Moreover, we set
$\mathbf H_h^2(\Omega,\mathcal S_h^{(3)})
:=
\mathbf H_h^2(\Omega,\mathbb M_h^{(3)})
\cap\mathcal S_h^{(3)}.$

The two-dimensional discretization is obtained analogously by
omitting the $z$-direction and restricting all spatial and tensor
indices to $\{1,2\}$, with the discrete operators and norms understood
accordingly. Hereafter, $\mathbb{M}_h^{(d)}$ and
$\mathcal{S}_h^{(d)}$ denote the corresponding spaces for
$d\in\{2,3\}$.
\subsection{Fully discrete GSAV-ETD2 scheme}
Let $\tau > 0$ be the uniform time step size and $t_n = n\tau$ for $n \ge 0$.
Let $\mathbf{Q}_h$, $u_h$, and $s_h$ denote the fully discrete numerical approximations of the exact solutions $\mathbf{Q}(t)$, $u(t)$, and $s(t)$, respectively.
The discrete counterparts of $E_1$ and $g$ are defined as follows:
\begin{equation}\label{eq_s_ex_n}
    \begin{aligned}
        E_{1h}[\mathbf Q_h,u_h]
       & : = 2 B_0 q^2 \langle D_h^2 u_h, \bM_h u_h \rangle_h + B_0 q^4 \| \bM_h  u_h  \|_h^2 + \langle f_{\mathrm{bn}}(\mathbf{Q}_h ), 1 \rangle_h \\&  + \langle f_{\mathrm{bs}}(u_h ), 1 \rangle_h-\frac \kappa2(\|u_h \|_h^2 + \|\mathbf{Q}_h \|_h^2),\\
         g_h(\mathbf Q_h, u_h, s_h) &:=  \zeta(s_h)/\zeta(E_{1h}[\mathbf Q_h,u_h]),
    \end{aligned}
\end{equation}
where $\bM_h(\mathbf Q_h) = \frac{\mathbf{Q}_h}{s_+} + \frac{\mathbf{I}_d}{d}$. The subscript $h$ in $s_h^n$ and $g_h^n$ emphasizes their dependence
on the spatial discretization.


We next define the discrete linear operators
$\mLQ:\mathcal S_h^{(d)}\to\mathcal S_h^{(d)}$ and
$\mLu:E_h^{\mathrm{per}}\to E_h^{\mathrm{per}}$ by
\begin{align} \label{eq_L}
    \mLQ\mathbf{Q}_h = -K\Delta_h \mathbf{Q}_h +  \kappa \mathbf{Q}_h,\quad
    \mLu u_h = 2B_0\Delta_h^2 u_h +  \kappa  u_h.
\end{align}
The nonlinear operators $\mNQ: \mathcal{S}_h^{(d)} \times E_h^{\text{per}} \times \mathbb{R} \to \mathcal{S}_h^{(d)}$ and $\mNu: \mathcal{S}_h^{(d)} \times E_h^{\text{per}} \times \mathbb{R} \to E_h^{\text{per}}$ are defined by
\ba \label{eq_N_Q_def}
\mathbf{N}_{h}(\mathbf{Q}_h, u_h, s_h)  = g_h \muQ, \quad
\mathcal{N}_{h}(\mathbf{Q}_h, u_h, s_h)  = g_h \mu_{h},
\ed
where $\muQ$ and $\mu_{h}$ are defined as
    \ba\label{eqc8}
\muQ(\mathbf{Q}_h,u_h)
&:=
\kappa\mathbf Q_h-\A\mathbf Q_h
+\B\left(
\mathbf Q_h^2
-\frac{\operatorname{tr}(\mathbf Q_h^2)}{d}\mathbf I_d
\right)
-\C\operatorname{tr}(\mathbf Q_h^2)\mathbf Q_h
-\frac{2B_0q^4}{s_+^2}\mathbf Q_hu_h^2
\\
&\quad-
\frac{2B_0q^2}{s_+}
\left(
u_hD_h^2u_h
-
\frac{\operatorname{tr}(u_hD_h^2u_h)}{d}\mathbf{I}_d
\right)
=
\mathbf H_{\rm p}(\mathbf Q_h,u_h)
+\mathbf H_{{\rm d},h}(u_h),
\\
\mu_h(\mathbf{Q}_h,u_h)
&:=
\kappa u_h-au_h-bu_h^2-cu_h^3
-2B_0\lvert q^2\mathbf M_h\rvert_F^2u_h
\\
&\quad-
2B_0q^2\left(
\mathbf M_h:D_h^2u_h
+\operatorname{div}_{h,c}^{\,2}(\mathbf M_hu_h)
\right)=
\mu_{\rm p}(\mathbf Q_h,u_h)
+\mu_{{\rm d},h}(\mathbf Q_h,u_h).
\ed

 The discrete initial values are chosen as
\begin{equation*}
    \mathbf{Q}_h^0=\mathcal{I}_h\mathbf{Q}(0),
    \qquad
    u_h^0=\mathcal{I}_h u(0),
    \qquad
    s_h^0=E_{1h}^0,
\end{equation*}
where $\mathcal I_h$ denotes the standard grid-restriction operator.
We use the standard $\varphi$-functions
\ba \label{eq3a}
\varphi_1(z)
&\coloneqq
\int_0^1 e^{-\varsigma z}\,d\varsigma
=
\frac{1-e^{-z}}{z},
\\
\varphi_2(z)
&\coloneqq
\int_0^1(1-\varsigma)e^{-\varsigma z}\,d\varsigma
=
\frac{e^{-z}-1+z}{z^2}.
\ed
For brevity, for any grid function $\phi$
(in particular, $\phi\in\{\bQ,\bu,\ts\}$), we introduce the notation
\begin{align}
\bigl(
\delta^{\mathrm p}\phi^{n,*},
\delta^{\mathrm c}\phi^{n+1},
\delta\phi^{n+1}
\bigr)
&:=
\bigl(
\phi^{n,*}-\phi^n,
\phi^{n+1}-\phi^{n,*},
\phi^{n+1}-\phi^n
\bigr).
\label{eq:time_increments}\\
\bigl(
\delta_\tau^{\mathrm p}\phi^{n,*},
\delta_\tau^{\mathrm c}\phi^{n+1},
\delta_\tau\phi^{n+1}
\bigr)
&:=
\frac{1}{\tau}
\bigl(
\delta^{\mathrm p}\phi^{n,*},
\delta^{\mathrm c}\phi^{n+1},
\delta\phi^{n+1}
\bigr).
\label{eq:time_difference_quotients}
\end{align}

Following \cite{ju2022generalized,hu2026structurepreservinggsavexponentialintegrator},
we introduce a first-order predictor for \eqref{eq1_9}. Given
the numerical solution $(\bQ^n,u_h^n,s_h^n)$, the predicted stage values
$(\bQ^{n,*},u_h^{n,*},s_h^{n,*})$ are computed as follows:
\begin{align}
    \bQ^{n,*} & = e^{-\tau \mLQ} \bQ^{n} + \tau \varphi_1(\tau \mLQ) \mNQ^n, \label{eq4_6}                                                                                       \\
    \bu^{n,*} & = e^{-\tau \mLu} \bu^{n} + \tau \varphi_1(\tau \mLu) \mNu^n,\label{eq4_7}                                                                                        \\
    s_h^{n,*} & = s_h^n -  \left\langle \mNQ^n, \dpred\bQ^{n,*}  \right\rangle_h - \left\langle \mNu^n, \dpred\bu^{n,*}  \right\rangle_h ,\label{eq4_8}
\end{align}
where $\mNQ^n = \mNQ(\bQ^n, u_h^n, s_h^n)$ and $\mNu^n = \mNu(\bQ^n, u_h^n, s_h^n)$.

We then formulate the following GSAV--ETD2 scheme for
\eqref{eq1_9}. For each $n\geq 0$, given the numerical solution
$(\bQ^n,u_h^n,s_h^n)$ and the predicted stage values
$(\bQ^{n,*},u_h^{n,*},s_h^{n,*})$ obtained from
\eqref{eq4_6}--\eqref{eq4_8}, we determine
$(\bQ^{n+1},u_h^{n+1},s_h^{n+1})$ as follows:
\begin{align}
    \bQ^{n+1} & = e^{-\tau \mLQ} \bQ^{n} + \tau \varphi_1\left( \tau \mLQ \right) \mNQ^n + \tau \varphi_2(\tau \mLQ) (\mNQ^{n,*} - \mNQ^n),\label{eq4_9}                                                                                       \\
    \bu^{n+1} & = e^{-\tau \mLu} \bu^{n} + \tau \varphi_1(\tau \mLu) \mNu^n + \tau \varphi_2(\tau \mLu) (\mNu^{n,*} - \mNu^n),\label{eq4_10}                                                                                        \\
   s_h^{n+1} & = s_h^n - \frac{1}{2} \left( \left\langle (\mNQ^n+\mNQ^{n,*}), \delta \bQ^{n+1} \right\rangle_h + \left\langle (\mNu^n+\mNu^{n,*}), \delta \bu^{n+1} \right\rangle_h \right)     \no \\  &-\frac{1}{4}\langle  \mathcal{L}_h \dx \bQ^{n+1}, \dx \bQ^{n+1} \rangle_h-\frac{1}{4}\langle \mathcal{D}_h \dx \bu^{n+1}, \dx \bu^{n+1} \rangle_h,\label{eq4_12}
\end{align}
where  $\mNQ^{n,*} = \mNQ(\bQ^{n,*}, u_h^{n,*}, s_h^{n,*})$ and $\mNu^{n,*} = \mNu(\bQ^{n,*}, u_h^{n,*}, s_h^{n,*})$.

For later use, we define the following discrete quantities:
\[
E_{1h}^n:=E_{1h}[\bQ^n,u_h^n],\qquad
E_{1h}^{n,*}:=E_{1h}[\bQ^{n,*},u_h^{n,*}],
\qquad
g_h^n:={\zeta(s_h^n)}/{\zeta(E_{1h}^n)},\qquad
g_h^{n,*}:={\zeta(s_h^{n,*})}/{\zeta(E_{1h}^{n,*})}.
\]

\section{Unconditional energy stability analysis}
In this section, we establish the unconditional energy stability of the GSAV-ETD2 scheme \eqref{eq4_9}-\eqref{eq4_12}.
The key analytical step is to rewrite the exponential updates in
equivalent quasi-implicit forms.  For $z\ge0$, let
$\bar{\varphi}_j(z):=1/\varphi_j(z)$, $j=1,2$, and introduce the
operator weights
\begin{equation}\label{eq_psi_functions}
\begin{aligned}
\psi(z)&:=\bar{\varphi}_1(z)-z=\frac{z}{e^z-1},
&\qquad
\psi_1(z)&:=\psi(z)+\frac z2
=\frac z2\operatorname{coth}\left(\frac z2\right),\\
\psi_2(z)&:=\bar{\varphi}_2(z)-z
=\frac{\varphi_1(z)}{\varphi_2(z)}.
\end{aligned}
\end{equation}
The values at $z=0$ are understood by continuity.

The following lemma records the equivalent stage and nodal forms used
in the energy stability and error analyses.
\begin{lemma}[Coercive reformulation ]
\label{lem:etd2_reformulation}
Let $\kappa>0$ and $\tau>0$.  The predictor equations
\eqref{eq4_6}--\eqref{eq4_7} are equivalent to
\begin{align}
\mq(\tau\mLQ)\dtpred\bQ^{n,*}
+\mLQ\bQ^{n,*}&=\mNQ^n,\label{eqc3}\\
\mq(\tau\mLu)\dtpred\bu^{n,*}
+\mLu\bu^{n,*}&=\mNu^n.\label{eqc5}
\end{align}
With these predictor relations, the corrector equations
\eqref{eq4_9}--\eqref{eq4_10} are equivalent to the stage form
\begin{align}
\psi_2(\tau\mLQ)\dxt\bQ^{n+1}
+\mq(\tau\mLQ)\dtpred\bQ^{n,*}
+\mLQ\bQ^{n+1}&=\mNQ^{n,*},\label{eqc2}\\
\psi_2(\tau\mLu)\dxt\bu^{n+1}
+\mq(\tau\mLu)\dtpred\bu^{n,*}
+\mLu\bu^{n+1}&=\mNu^{n,*},\label{eqc4}
\end{align}
and, equivalently, to the nodal form
\begin{align}
\mq(\tau\mLQ)\delta_\tau\bQ^{n+1}
+\mLQ\bQ^{n+1}
&=\mNQ^n+\frac{\varphi_2(\tau\mLQ)}{\varphi_1(\tau\mLQ)}
(\mNQ^{n,*}-\mNQ^n),\label{eq:etd2_nodal_Q}\\
\mq(\tau\mLu)\delta_\tau\bu^{n+1}
+\mLu\bu^{n+1}
&=\mNu^n+\frac{\varphi_2(\tau\mLu)}{\varphi_1(\tau\mLu)}
(\mNu^{n,*}-\mNu^n).\label{eq:etd2_nodal_u}
\end{align}
Together with the scalar updates \eqref{eq4_8} and \eqref{eq4_12},
either form defines the same GSAV--ETD2 scheme.
\end{lemma}
\begin{proof}
We write $(V_h,\Lambda_h,F_h)$ for either
$(\bQ,\mLQ,\mNQ)$ or $(\bu,\mLu,\mNu)$.
The scalar identities
\[
\bar{\varphi}_1(z)=\mq(z)+z,\qquad
e^{-z}\bar{\varphi}_1(z)=\mq(z),\qquad
z\varphi_2(z)=1-\varphi_1(z)
\]
hold for $z\ge0$, with continuous extensions at zero.
Since $\Lambda_h$ is symmetric positive definite and
$\varphi_1(z),\varphi_2(z)>0$ for $z\ge0$, these identities also hold
for the corresponding operators, and both $\varphi_j(\tau\Lambda_h)$
are invertible.
Applying $\tau^{-1}\bar{\varphi}_1(\tau\Lambda_h)$ to the predictor
gives \eqref{eqc3}--\eqref{eqc5}.

Subtracting the predictor from the corrector and applying
$\tau^{-1}\bar{\varphi}_2(\tau\Lambda_h)$ gives
\begin{equation}
\bar{\varphi}_2(\tau\Lambda_h)\dxt V_h^{n+1}
=F_h^{n,*}-F_h^n.
\label{eq:etd2_stage_difference}
\end{equation}
Adding the transformed predictor and using
$\Lambda_hV_h^{n,*}=\Lambda_hV_h^{n+1}
-\Lambda_h(V_h^{n+1}-V_h^{n,*})$ yields
\eqref{eqc2}--\eqref{eqc4}.
Applying $\tau^{-1}\bar{\varphi}_1(\tau\Lambda_h)$ directly to the
corrector instead yields \eqref{eq:etd2_nodal_Q}--\eqref{eq:etd2_nodal_u}.
All these transformations are reversible, which proves the equivalence.
\end{proof}

\begin{remark}
The significance of the above reformulation is not merely algebraic.
Taking the inner product of
$\mq(\tau\Lambda_h)\delta_\tau V_h^{n+1}
+\Lambda_h V_h^{n+1}$ with $\delta_\tau V_h^{n+1}$ yields
\ba \label{eq:etd2_coercive_testing}
&\left\langle
\psi(\tau\Lambda_h)\delta_\tau V_h^{n+1},
\delta_\tau V_h^{n+1}
\right\rangle_h
+
\left\langle
\Lambda_hV_h^{n+1},
\delta_\tau V_h^{n+1}
\right\rangle_h
\\
&\qquad =
\frac{1}{2\tau}
\left(
\left\langle
\Lambda_hV_h^{n+1},
 V_h^{n+1}
\right\rangle_h
-
\left\langle
\Lambda_h V_h^{n},
 V_h^{n}
\right\rangle_h
\right)
+
\left\langle
\psi_1(\tau\Lambda_h)\delta_\tau V_h^{n+1},
\delta_\tau V_h^{n+1}
\right\rangle_h.
\ed
Since $\psi_1(z)\ge1$ for $z\ge0$, the last term provides
mesh-uniform control of
$\|\delta_\tau V_h^{n+1}\|_h^2$ without imposing the restriction
 $\tau\lambda_{\max}(\Lambda_h)\lesssim1$ used in \cite{ju2022generalized}. This mesh-uniform coercivity is
essential in the fully discrete error analysis.

Moreover, the coercive structure \eqref{eq:etd2_nodal_Q}--\eqref{eq:etd2_nodal_u}   allows the exponential scheme to be analyzed through standard
backward Euler-type energy arguments while retaining its exact
treatment of the linear dissipative part.
\end{remark}

Since $\mLQ$ and $\mLu$ are symmetric positive definite, they naturally induce
the following inner products and associated energy norms:
\ba
\ilQ{ \cdot }  := \langle \mathcal{L}_h \cdot,\cdot  \rangle_h= K  \norm{\nabla_h \cdot }_h^2 +  \kappa  \norm{\cdot}_h^2, \quad \ilu{\cdot} :=\langle \mathcal{D}_h \cdot,\cdot  \rangle_h= 2B_0  \norm{\Delta_h \cdot }_h^2 + \kappa \norm{\cdot}_h^2.
\ed
For any $\mathbf{U} \in \mathcal{S}_h^{(d)}$, we define the following weighted energy norms:
\ba\label{eq_inner_Q}
\iaQ{\mathbf{U}}
&:= \langle \psi(\tau \mathcal{L}_h)\mathbf{U}, \mathbf{U} \rangle_h,
&\qquad
\ibQ{\mathbf{U}}
&:= \langle \psi_1(\tau \mathcal{L}_h)\mathbf{U}, \mathbf{U} \rangle_h,
\\
\icQ{\mathbf{U}}
&:= \langle \psi(\tau \mathcal{L}_h)\mathcal{L}_h\mathbf{U}, \mathbf{U} \rangle_h,
&\qquad
\idQ{\mathbf{U}}
&:= \langle \psi_1(\tau \mathcal{L}_h)\mathcal{L}_h\mathbf{U}, \mathbf{U} \rangle_h.
\ed
The scalar relations
\begin{equation}\label{eq_a2}
0 \le \psi^2(z) \le \psi(z) \le 1 \le \psi_1(z),
\qquad
0 \le \psi(z)z \le z\le \psi_1(z)z
\end{equation}
and the spectral mapping theorem directly give the norm comparisons
\ba\label{eq_norm_bounds_1}
\|\psi(\tau \mathcal{L}_h) \mathbf{U}\|_h
\le \|\mathbf{U}\|_{\psi_\mathcal{L}}
\le \|\mathbf{U}\|_h
\le \|\mathbf{U}\|_{\psi_\mathcal{L}^1},\quad \|\mathbf{U}\|_{\psi_{\mathcal L}^2}
\le\|\mathbf{U}\|_{\mathcal L_h}
\le \|\mathbf{U}\|_{\psi_{\mathcal L}^{3}}.
\ed
The corresponding norms and inequalities associated with $\mathcal{D}_h$
are defined analogously by replacing $\mathcal{L}_h$ with $\mathcal{D}_h$
in the above expressions.
\begin{lemma}\label{lem_phi_ineq}
    For any $z \ge 0$, the following inequalities hold:
    \ba
    0 \le& \varphi_1(z) \le 1, \quad 0 \le \varphi_2(z) \le \frac{1}{2}, \quad 0 \le z\varphi_2(z) \le 1,\\
    \varphi_2(z) \le& \varphi_1(z) \le 2\varphi_2(z),\quad
        1 \le \frac{\varphi_1(z)}{\varphi_2(z)} \le 2,\quad
\frac{1}{2} \le \frac{\varphi_2(z)}{\varphi_1(z)} \le 1,\\
\bar{\varphi}_1(z)\le&\bar{\varphi}_2(z) \le 2\bar{\varphi}_1(z),\quad \frac{1}{2}\bar{\varphi}_2(z) - \bar{\varphi}_1(z) + \frac{1}{2} z \ge 0.
    \ed
\end{lemma}
The proof of Lemma~\ref{lem_phi_ineq} is given in
Appendix~\ref{app:proof_phi_ineq}.
We define the modified discrete energy $\mathcal{E}_h$ by
\ba\label{energy_discrete}
\mathcal{E}_h[\mathbf{Q}_h, u_h,s_h] = \frac{K}{2} \|\nabla_h \mathbf{Q}_h\|_h^2  + B_0 \left\| \Delta_h u_h  \right\|_h^2 + \frac{\kappa }{2} \left( \|\mathbf{Q}_h\|_h^2 + \|u_h\|_h^2 \right) + s_h.
\ed
At time $t_n$, we denote the discrete energy by
\ba
\mathcal{E}_h^n = \mathcal{E}_h[\bQ^n, u_h^n, s_h^n],\quad
\mathcal{E}_h^{n,*} = \mathcal{E}_h[\bQ^{n,*}, u_h^{n,*}, s_h^{n,*}],\quad n \ge 0.
\ed
\begin{theorem}[Unconditional energy stability]\label{them3_1}
 For any $\kappa > 0$ and any time step size $\tau > 0$, the GSAV-ETD2 scheme \eqref{eq4_6}--\eqref{eq4_12} satisfies the following discrete energy dissipation law:
    \ba
    \mathcal{E}_h^{n+1} \le \mathcal{E}_h^n,\quad \forall n \ge 0.
    \ed
\end{theorem}
\begin{proof}
    By Lemma~\ref{lem:etd2_reformulation}, the predictor satisfies
\eqref{eqc3}--\eqref{eqc5}. Taking their discrete inner products with
$\dpred\bQ^{n,*}$ and $\dpred\bu^{n,*}$, respectively, gives
    \ba \label{eq34}
    \langle \mLQ\bQ^{n,*}, \dpred\bQ^{n,*} \rangle_h &=-\frac{1}{\tau}\langle \mq(\tau \mLQ) \dpred\bQ^{n,*}, \dpred\bQ^{n,*} \rangle_h+\id{\mNQ^n}{\dpred\bQ^{n,*}}
    \\&= -\frac{1}{\tau}\iaQ{\dpred\bQ^{n,*}}  + g_h^n \langle \muQ^n, \dpred\bQ^{n,*} \rangle_h, \\
    \langle \mLu\bu^{n,*}, \dpred\bu^{n,*} \rangle_h &= -\frac{1}{\tau}\langle \mq(\tau \mLu) \dpred\bu^{n,*}, \dpred\bu^{n,*} \rangle_h  + \langle \mNu^n, \dpred\bu^{n,*} \rangle_h \\
    &=-\frac{1}{\tau}\iau{\dpred\bu^{n,*}} + g_h^n \langle \mu_h^n, \dpred\bu^{n,*} \rangle_h.
    \ed
    The energy difference $\mathcal{E}_h^{n,*}-\mathcal{E}_h^n$ is
    \ba \label{eq_diff_E}
        \mathcal{E}_h^{n,*} - {\mathcal{E}_h}^n & = \frac{K}{2}\ia{\nabla_h  \bQ^{n,*} }- \frac{K}{2}\ia{\nabla_h \bQ^{n}} + B_0\ia{\Delta_h\bu^{n,*}} - B_0\ia{\Delta_h\bu^{n}}  +  s_h^{n,*} -s_h^n  \\&
        + \frac{\kappa }{2} \left( \|\bQ^{n,*}\|_h^2 - \|\bQ^n\|_h^2 + \|\bu^{n,*}\|_h^2 - \|\bu^n\|_h^2 \right)  \\
                                            & =  \id{\mLQ\bQ^{n,*}}{\dpred\bQ^{n,*}}  - \frac{1}{2} \ilQ{\dpred\bQ^{n,*}}  + \langle \mLu\bu^{n,*}, \dpred\bu^{n,*} \rangle_h  \\
                                            &
        - \frac{1}{2} \ilu{ \dpred\bu^{n,*}}  - g_h^n \left( \langle \muQ^n, \dpred\bQ^{n,*} \rangle_h  + \langle \mu_h^n, \dpred\bu^{n,*} \rangle_h \right).
    \ed
    Substituting \eqref{eq34} into \eqref{eq_diff_E} and using
    \eqref{eq_a2} gives
    \begin{equation}\label{eq_proof1}
        \begin{aligned}
            {\mathcal{E}}^{n,*}_h - \mathcal{E}_h^n
            & = -\frac{1}{\tau}\iaQ{\dpred\bQ^{n,*}} - \frac{1}{2} \ilQ{\dpred\bQ^{n,*}}
            - \frac{1}{\tau}\iau{\dpred\bu^{n,*}} - \frac{1}{2} \ilu{\dpred\bu^{n,*}} \\
            & = -\frac{1}{2\tau}\left(\iaQ{\dpred\bQ^{n,*}}+\left\langle(\mq(\tau\mathcal{L}_h)+\tau\mathcal{L}_h)\dpred\bQ^{n,*},\dpred\bQ^{n,*}\right\rangle_h\right) \\
            & \quad - \frac{1}{2\tau}\left(\iau{\dpred\bu^{n,*}}+\left\langle(\mq(\tau\mathcal{D}_h)+\tau\mathcal{D}_h)\dpred\bu^{n,*},\dpred\bu^{n,*}\right\rangle_h\right) \\
            & = -\frac{1}{2\tau}\left(\iaQ{\dpred\bQ^{n,*}}+\left\langle\bar{\varphi}_1(\tau\mathcal{L}_h)\dpred\bQ^{n,*},\dpred\bQ^{n,*}\right\rangle_h\right) \\
            & \quad - \frac{1}{2\tau}\left(\iau{\dpred\bu^{n,*}}+\left\langle\bar{\varphi}_1(\tau\mathcal{D}_h)\dpred\bu^{n,*},\dpred\bu^{n,*}\right\rangle_h\right)\leq 0.
        \end{aligned}
    \end{equation}

Next, the energy difference $\mathcal{E}_h^{n+1}-\mathcal{E}_h^{n,*}$ is
    \begin{align}
        \mathcal{E}_h^{n+1} - \mathcal{E}_h^{n,*}
        & = \frac{K}{2}\ia{\nabla_h  \bQ^{n+1} }- \frac{K}{2}\ia{\nabla_h \bQ^{n,*}}
        + B_0\ia{\Delta_h\bu^{n+1}} - B_0\ia{\Delta_h\bu^{n,*}} + s_h^{n+1}-s_h^{n,*} \no\\
        & \quad + \frac{\kappa }{2} \left( \|\bQ^{n+1}\|_h^2 - \|\bQ^{n,*}\|_h^2 + \|\bu^{n+1}\|_h^2 - \|\bu^{n,*}\|_h^2 \right) \no\\
        & = \id{\mLQ\bQ^{n+1}}{\dx \bQ^{n+1}} - \frac{1}{2} \ilQ{\dx \bQ^{n+1}}
        + \langle \mLu\bu^{n+1}, \dx \bu^{n+1} \rangle_h \no\\
        & \quad - \frac{1}{2} \ilu{\dx \bu^{n+1}} + s_h^{n+1} -s_h^{n,*}.\label{eq4t}
    \end{align}
Using the stage form \eqref{eqc2}--\eqref{eqc4} from
Lemma~\ref{lem:etd2_reformulation} and taking the inner products with
$\dx\bQ^{n+1}$ and $\dx\bu^{n+1}$, respectively, we obtain
    \ba \no
    \langle \mLQ\bQ^{n+1}, \dx \bQ^{n+1} \rangle_h
    &= -\frac{1}{\tau}\langle \psi_2(\tau \mLQ) \dx \bQ^{n+1}, \dx \bQ^{n+1} \rangle_h \\
    &\quad -\frac{1}{\tau}\langle \mq(\tau \mLQ) \dpred\bQ^{n,*},\dx \bQ^{n+1} \rangle_h
    + \id{\mNQ^{n,*}}{\dx \bQ^{n+1}}, \\
    \langle \mLu\bu^{n+1}, \dx \bu^{n+1} \rangle_h
    &= -\frac{1}{\tau}\langle \psi_2(\tau \mLu) \dx \bu^{n+1}, \dx \bu^{n+1} \rangle_h \\
    &\quad -\frac{1}{\tau}\langle \mq(\tau \mLu) \dpred\bu^{n,*},\dx \bu^{n+1} \rangle_h
    + \langle \mNu^{n,*}, \dx\bu^{n+1} \rangle_h.
    \ed
    The update for $s$ in \eqref{eq4_12} gives
    \begin{align*}
        & s_h^{n+1}-s_h^{n,*} +\id{\mNQ^{n,*}}{\dx \bQ^{n+1}} + \langle \mNu^{n,*}, \dx \bu^{n+1} \rangle_h\\
        & = - \frac{1}{2} \left( \left\langle (\mNQ^n+\mNQ^{n,*}), \delta \bQ^{n+1} \right\rangle_h + \left\langle (\mNu^n+\mNu^{n,*}), \delta \bu^{n+1} \right\rangle_h \right) \\
        & +  \left( \langle \mNQ^n, \dpred\bQ^{n,*} \rangle_h  + \langle \mNu^n, \dpred\bu^{n,*} \rangle_h \right)+\left( \langle \mNQ^{n,*}, \dx \bQ^{n+1} \rangle_h  + \langle \mNu^{n,*}, \dx \bu^{n+1} \rangle_h \right)\\
        &-\frac{1}{4\tau}\langle \tau \mathcal{L}_h \dx \bQ^{n+1}, \dx \bQ^{n+1} \rangle_h-\frac{1}{4\tau}\langle\tau \mathcal{D}_h \dx \bu^{n+1}, \dx \bu^{n+1} \rangle_h\\
        & = \frac{1}{2} \left( \left\langle ( \mNQ^{n,*} - \mNQ^n), \dx \bQ^{n+1} \right\rangle_h + \left\langle (\mNu^{n,*} - \mNu^n), \dx \bu^{n+1} \right\rangle_h \right) \\
&- \frac{1}{2} \left( \left\langle ( \mNQ^{n,*} - \mNQ^n), \dpred\bQ^{n,*} \right\rangle_h + \left\langle (\mNu^{n,*} - \mNu^n), \dpred\bu^{n,*} \right\rangle_h \right)        \\
        &-\frac{1}{4\tau}\langle \tau \mathcal{L}_h \dx \bQ^{n+1}, \dx \bQ^{n+1} \rangle_h-\frac{1}{4\tau}\langle\tau \mathcal{D}_h \dx \bu^{n+1}, \dx \bu^{n+1} \rangle_h
    \end{align*}

    By \eqref{eq:etd2_stage_difference}, we have
\ba
\frac1\tau\bar{\varphi}_2(\tau\mLQ)(\bQ^{n+1}-\bQ^{n,*})=\mNQ^{n,*}-\mNQ^n,\\
\frac1\tau\bar{\varphi}_2(\tau\mLu)(\bu^{n+1}-\bu^{n,*})=\mNu^{n,*}-\mNu^n.
\ed
    Substituting these relations into \eqref{eq4t} gives
    \begin{align*}
        \mathcal{E}_h^{n+1} - \mathcal{E}_h^{n,*} & =  \frac{1}{2\tau} \left( \left\langle \bar{\varphi}_2(\tau \mathcal{L}_h)(\dx \bQ^{n+1}), \dx \bQ^{n+1} \right\rangle_h + \left\langle \bar{\varphi}_2(\tau \mathcal{D}_h)(\dx \bu^{n+1}), \dx \bu^{n+1} \right\rangle_h \right)\\
        & - \frac{1}{2\tau} \left( \left\langle \bar{\varphi}_2(\tau \mathcal{L}_h)(\dx \bQ^{n+1}), \dpred\bQ^{n,*} \right\rangle_h + \left\langle \bar{\varphi}_2(\tau \mathcal{D}_h)(\dx \bu^{n+1}), \dpred\bu^{n,*} \right\rangle_h \right)\\
        &   -\frac{1}{\tau}\langle \psi_2(\tau \mLQ) \dx \bQ^{n+1}, \dx \bQ^{n+1} \rangle_h-\frac{1}{\tau}\langle \mq(\tau \mLQ) \dx \bQ^{n+1}, \dpred\bQ^{n,*} \rangle_h    - \frac{1}{2} \ilQ{\dx \bQ^{n+1}}  \\
        & -\frac{1}{\tau}\langle \psi_2(\tau \mLu) \dx \bu^{n+1}, \dx \bu^{n+1} \rangle_h -\frac{1}{\tau}\langle \mq(\tau \mLu) \dx \bu^{n+1}, \dpred\bu^{n,*} \rangle_h- \frac{1}{2} \ilu{ \dx \bu^{n+1}}\\
        &-\frac{1}{4\tau}\langle \tau \mathcal{L}_h \dx \bQ^{n+1}, \dx \bQ^{n+1} \rangle_h-\frac{1}{4\tau}\langle\tau \mathcal{D}_h \dx \bu^{n+1}, \dx \bu^{n+1} \rangle_h\\
                =& - \frac{1}{2\tau} \left( \left\langle \bar{\varphi}_2(\tau \mathcal{L}_h)(\dx \bQ^{n+1}), \dpred\bQ^{n,*} \right\rangle_h + \left\langle \bar{\varphi}_2(\tau \mathcal{D}_h)(\dx \bu^{n+1}), \dpred\bu^{n,*} \right\rangle_h \right)\\
        &   -\frac{1}{2\tau}\langle (\bar{\varphi}_2(\tau \mathcal{L}_h)-\tau \mathcal{L}_h) \dx \bQ^{n+1}, \dx \bQ^{n+1} \rangle_h-\frac{1}{\tau}\langle \mq(\tau \mLQ) \dx \bQ^{n+1}, \dpred\bQ^{n,*} \rangle_h  \\
        & -\frac{1}{2\tau}\langle (\bar{\varphi}_2(\tau \mathcal{D}_h)-\tau \mathcal{D}_h) \dx \bu^{n+1}, \dx \bu^{n+1} \rangle_h -\frac{1}{\tau}\langle \mq(\tau \mLu) \dx \bu^{n+1}, \dpred\bu^{n,*} \rangle_h\\        &-\frac{1}{4\tau}\langle \tau \mathcal{L}_h \dx \bQ^{n+1}, \dx \bQ^{n+1} \rangle_h-\frac{1}{4\tau}\langle\tau \mathcal{D}_h \dx \bu^{n+1}, \dx \bu^{n+1} \rangle_h.
    \end{align*}
   Consequently, the change in the discrete energy from $\mathcal E_h^n$ to $\mathcal E_h^{n+1}$ is given by
    \begin{align*}
        \mathcal{E}_h^{n+1} - \mathcal{E}_h^{n} =& - \frac{1}{2\tau} \left( \left\langle \bar{\varphi}_2(\tau \mathcal{L}_h)(\dx \bQ^{n+1}), \dpred\bQ^{n,*} \right\rangle_h + \left\langle \bar{\varphi}_2(\tau \mathcal{D}_h)(\dx \bu^{n+1}), \dpred\bu^{n,*} \right\rangle_h \right)\\
        &   -\frac{1}{2\tau}\langle (\bar{\varphi}_2(\tau \mathcal{L}_h)-\tau \mathcal{L}_h) \dx \bQ^{n+1}, \dx \bQ^{n+1} \rangle_h-\frac{1}{\tau}\langle \mq(\tau \mLQ) \dx \bQ^{n+1}, \dpred\bQ^{n,*} \rangle_h  \\
        & -\frac{1}{2\tau}\langle (\bar{\varphi}_2(\tau \mathcal{D}_h)-\tau \mathcal{D}_h) \dx \bu^{n+1}, \dx \bu^{n+1} \rangle_h -\frac{1}{\tau}\langle \mq(\tau \mLu) \dx \bu^{n+1}, \dpred\bu^{n,*} \rangle_h\\
        &-\frac{1}{2\tau}(\iaQ{\dpred\bQ^{n,*}}+\left\langle\bar{\varphi}_1(\tau\mathcal{L}_h)\dpred\bQ^{n,*},\dpred\bQ^{n,*}\right\rangle_h)\\
        &-\frac{1}{2\tau}(\iau{\dpred\bu^{n,*}}+\left\langle\bar{\varphi}_1(\tau\mathcal{D}_h)\dpred\bu^{n,*},\dpred\bu^{n,*}\right\rangle_h)\\
                &-\frac{1}{4\tau}\langle \tau \mathcal{L}_h \dx \bQ^{n+1}, \dx \bQ^{n+1} \rangle_h-\frac{1}{4\tau}\langle\tau \mathcal{D}_h \dx \bu^{n+1}, \dx \bu^{n+1} \rangle_h\\
        =& - \frac{1}{2\tau} \left( \left\langle \bar{\varphi}_2(\tau \mathcal{L}_h)(\dx \bQ^{n+1}), \dpred\bQ^{n,*} \right\rangle_h + \left\langle \bar{\varphi}_2(\tau \mathcal{D}_h)(\dx \bu^{n+1}), \dpred\bu^{n,*} \right\rangle_h \right)\\
        &-\frac{1}{2\tau}\left\langle\bar{\varphi}_1(\tau\mathcal{L}_h)\dpred\bQ^{n,*},\dpred\bQ^{n,*}\right\rangle_h - \frac{1}{2\tau}\left\langle\bar{\varphi}_1(\tau\mathcal{D}_h)\dpred\bu^{n,*},\dpred\bu^{n,*}\right\rangle_h\\
        &   -\frac{1}{2\tau}\langle (\bar{\varphi}_2(\tau \mathcal{L}_h)-\bar{\varphi}_1(\tau \mathcal{L}_h)) \dx \bQ^{n+1}, \dx \bQ^{n+1} \rangle_h-\frac{1}{2\tau}\langle \mq(\tau \mLQ) \delta \bQ^{n+1}, \delta \bQ^{n+1} \rangle_h \\
        & -\frac{1}{2\tau}\langle (\bar{\varphi}_2(\tau \mathcal{D}_h)-\bar{\varphi}_1(\tau \mathcal{D}_h)) \dx \bu^{n+1}, \dx \bu^{n+1} \rangle_h -\frac{1}{2\tau}\langle \mq(\tau \mLu)  \delta \bu^{n+1}, \delta \bu^{n+1} \rangle_h\\        &-\frac{1}{4\tau}\langle \tau \mathcal{L}_h \dx \bQ^{n+1}, \dx \bQ^{n+1} \rangle_h-\frac{1}{4\tau}\langle\tau \mathcal{D}_h \dx \bu^{n+1}, \dx \bu^{n+1} \rangle_h.
    \end{align*}
    Applying Young's inequality and Lemma~\ref{lem_phi_ineq}, we obtain the following bound for the cross term:
    \begin{align*}
        &-\frac{1}{2\tau} \left( \left\langle \bar{\varphi}_2(\tau \mathcal{L}_h)(\dx \bQ^{n+1}), \dpred\bQ^{n,*} \right\rangle_h + \left\langle \bar{\varphi}_2(\tau \mathcal{D}_h)(\dx \bu^{n+1}), \dpred\bu^{n,*} \right\rangle_h \right)\\&-
          \frac{1}{4\tau} \left( \left\langle \bar{\varphi}_2(\tau \mathcal{L}_h)(\dx \bQ^{n+1}), \dx \bQ^{n+1} \right\rangle_h + \left\langle \bar{\varphi}_2(\tau \mathcal{D}_h)(\dx \bu^{n+1}), \dx \bu^{n+1} \right\rangle_h \right)\\
        & - \frac{1}{2\tau} \left( \left\langle \bar{\varphi}_1(\tau \mathcal{L}_h)(\dpred\bQ^{n,*}), \dpred\bQ^{n,*} \right\rangle_h + \left\langle \bar{\varphi}_1(\tau \mathcal{D}_h)(\dpred\bu^{n,*}), \dpred\bu^{n,*} \right\rangle_h \right)\\&=
     -\frac{1}{2\tau} \left( \left\langle \bar{\varphi}_2(\tau \mathcal{L}_h)(\dx \bQ^{n+1}), \dpred\bQ^{n,*} \right\rangle_h + \left\langle \bar{\varphi}_2(\tau \mathcal{D}_h)(\dx \bu^{n+1}), \dpred\bu^{n,*} \right\rangle_h \right)\\&-
          \frac{1}{4\tau} \left( \left\langle \bar{\varphi}_2(\tau \mathcal{L}_h)(\dx \bQ^{n+1}), \dx \bQ^{n+1} \right\rangle_h + \left\langle \bar{\varphi}_2(\tau \mathcal{D}_h)(\dx \bu^{n+1}), \dx \bu^{n+1} \right\rangle_h \right)
          \\&-
          \frac{1}{4\tau} \left( \left\langle \bar{\varphi}_2(\tau \mathcal{L}_h)(\dpred\bQ^{n,*}), \dpred\bQ^{n,*}\right\rangle_h + \left\langle \bar{\varphi}_2(\tau \mathcal{D}_h)(\dpred\bu^{n,*}), \dpred\bu^{n,*} \right\rangle_h \right)\\
        & - \frac{1}{2\tau} \left( \left\langle (\bar{\varphi}_1(\tau \mathcal{L}_h)-\frac{1}{2} \bar{\varphi}_2(\tau \mathcal{L}_h))(\dpred\bQ^{n,*}), \dpred\bQ^{n,*} \right\rangle_h + \left\langle (\bar{\varphi}_1(\tau \mathcal{D}_h)-\frac{1}{2} \bar{\varphi}_2(\tau \mathcal{D}_h))(\dpred\bu^{n,*}), \dpred\bu^{n,*} \right\rangle_h \right)\\&
        =-\frac{1}{4\tau} \left( \left\langle \bar{\varphi}_2(\tau \mathcal{L}_h)\delta \bQ^{n+1}, \delta \bQ^{n+1}\right\rangle_h + \left\langle \bar{\varphi}_2(\tau \mathcal{D}_h)\delta \bu^{n+1}, \delta \bu^{n+1} \right\rangle_h \right)\\
        & - \frac{1}{2\tau} \left( \left\langle (\bar{\varphi}_1(\tau \mathcal{L}_h)-\frac{1}{2} \bar{\varphi}_2(\tau \mathcal{L}_h))(\dpred\bQ^{n,*}), \dpred\bQ^{n,*} \right\rangle_h + \left\langle (\bar{\varphi}_1(\tau \mathcal{D}_h)-\frac{1}{2} \bar{\varphi}_2(\tau \mathcal{D}_h))(\dpred\bu^{n,*}), \dpred\bu^{n,*} \right\rangle_h \right).
    \end{align*}
    Substituting this bound into the expression for $\mathcal{E}_h^{n+1}-\mathcal{E}_h^n$ gives
    \begin{align*}
        \mathcal{E}_h^{n+1} - \mathcal{E}_h^{n}
        \le& -\frac{1}{2\tau}\left\langle \left(\frac{1}{2}\bar{\varphi}_2(\tau \mathcal{L}_h)-\bar{\varphi}_1(\tau \mathcal{L}_h)\right) \dx \bQ^{n+1}, \dx \bQ^{n+1} \right\rangle_h \\
        & -\frac{1}{2\tau}\left\langle (\mq(\tau \mLQ)+\frac 12 \bar{\varphi}_2(\tau \mathcal{L}_h)) \delta \bQ^{n+1}, \delta \bQ^{n+1} \right\rangle_h \\
        & -\frac{1}{2\tau}\left\langle \left(\frac{1}{2}\bar{\varphi}_2(\tau \mathcal{D}_h)-\bar{\varphi}_1(\tau \mathcal{D}_h)\right) \dx \bu^{n+1}, \dx \bu^{n+1} \right\rangle_h \\
        & -\frac{1}{2\tau}\left\langle (  \mq(\tau \mLu)+\frac{1}{2}\bar{\varphi}_2(\tau \mathcal{D}_h)  ) \delta \bu^{n+1}, \delta \bu^{n+1} \right\rangle_h \\
        & -\frac{1}{4\tau}\left\langle \tau \mathcal{L}_h \dx \bQ^{n+1}, \dx \bQ^{n+1} \right\rangle_h
        -\frac{1}{4\tau}\left\langle\tau \mathcal{D}_h \dx \bu^{n+1}, \dx \bu^{n+1} \right\rangle_h\\
        \le& -\frac{1}{2\tau}\left\langle \left(\frac{1}{2}\bar{\varphi}_2(\tau \mathcal{L}_h)-\bar{\varphi}_1(\tau \mathcal{L}_h) + \frac{1}{2}\tau \mathcal{L}_h\right) \dx \bQ^{n+1}, \dx \bQ^{n+1} \right\rangle_h \\
        & -\frac{1}{2\tau}\left\langle \left(\frac{1}{2}\bar{\varphi}_2(\tau \mathcal{D}_h)-\bar{\varphi}_1(\tau \mathcal{D}_h) + \frac{1}{2}\tau \mathcal{D}_h\right) \dx \bu^{n+1}, \dx \bu^{n+1} \right\rangle_h \leq 0,
    \end{align*}
    where the last inequality follows from $\frac{1}{2}\bar{\varphi}_2(z) - \bar{\varphi}_1(z) + \frac{1}{2} z\ge 0$ for $z \ge 0$.
\end{proof}
\begin{remark}
\label{rem:energy_closing_correction_terms}
The last two terms in the scalar corrector equation~\eqref{eq4_12}
are introduced solely to close the discrete modified-energy estimate; they
do not alter the updates of the primary variables $\bQ^{n+1}$ and
$u_h^{n+1}$.
 Uniformly
with respect to $h$, the operator estimate relevant to the fully discrete
analysis gives
\[
\begin{aligned}
&\left\langle\mathcal L_h\dx\bQ^{n+1},\dx\bQ^{n+1}\right\rangle_h
+\left\langle\mathcal D_h\dx\bu^{n+1},\dx\bu^{n+1}\right\rangle_h\\
&\qquad\le C\tau\left(
\left\|\mNQ^{n,*}-\mNQ^n\right\|_h^2
+\left\|\mNu^{n,*}-\mNu^n\right\|_h^2
\right)
=\mathcal O(\tau^3).
\end{aligned}
\]
Thus these terms are still of the local order required by a second-order
method and do not reduce its temporal accuracy.
\end{remark}
\section{Convergence analysis}\label{se3}
Throughout this section, let  $(\bQ^{n,*}, u_h^{n,*}, s_h^{n,*})$ denote the stage values generated by \eqref{eq4_6}--\eqref{eq4_8} and $(\mathbf{Q}_h^n,u_h^n,s_h^n)$ denote the numerical
solution generated by the GSAV-ETD2 scheme \eqref{eq4_9}--\eqref{eq4_12}.
Let $(\mathbf{Q}(t),u(t),s(t))$ denote the exact solution to the  system
\eqref{eq1_9}, satisfying the following regularity assumptions:
         \ba \label{eqa_19}
&\mathbf{Q}(t) \in L^{\infty}(0,T; \mathbf{H}^4(\Omega, \mathcal{S}^{(d)})) \cap W^{2,\infty}(0,T; \mathbf{H}^2(\Omega, \mathcal{S}^{(d)})),\\& u(t) \in L^\infty(0,T; H^6(\Omega)) \cap W^{2,\infty}(0,T; H^2(\Omega)).
            \ed

For notational simplicity, we omit $\mathcal{I}_h$ in the following analysis,
when exact solutions such as $\mathbf{Q}(t_n)$ and $u(t_n)$ appear in discrete operators or norms.

We define the error functions at time $t_n$ and the corresponding
difference quotients as follows:
\ba
\bigl(\be_{\mathbf Q}^{n},e_u^{n},e_s^{n}\bigr)
&:=
\bigl(
\mathbf Q_h^n-\mathbf Q(t_n),
u_h^n-u(t_n),
s_h^n-s(t_n)
\bigr),\\
\bigl(\be_{\mathbf Q}^{n,*},e_u^{n,*},e_s^{n,*}\bigr)
&:=
\bigl(
\mathbf Q_h^{n,*}-\mathbf Q(t_{n+1}),
u_h^{n,*}-u(t_{n+1}),
s_h^{n,*}-s(t_{n+1})
\bigr),\\
(\ea,\eb,\ec)
&:=
\frac{1}{\tau}
\bigl(
\be_{\mathbf Q}^{n+1}-\be_{\mathbf Q}^{n},
e_u^{n+1}-e_u^{n},
e_s^{n+1}-e_s^{n}
\bigr),\\
\delta_\tau\mathbf Q(t_{n+1})
&:=
\frac{\mathbf Q(t_{n+1})-\mathbf Q(t_n)}{\tau},
\qquad
\delta_\tau u(t_{n+1})
:=
\frac{u(t_{n+1})-u(t_n)}{\tau}.
\ed
 The initial errors are defined by
\begin{equation}\label{eqcv1}
\begin{aligned}
    \be_{\mathbf Q}^0
    :=\mathbf Q_h^0-\mathbf Q(0),\quad
    e_u^0:=u_h^0-u(0),\quad
    e_s^0
    :=s_h^0-s(0)
    =s_h^0-E_1[\mathbf Q(0),u(0)],
\end{aligned}
\end{equation}
Under  \eqref{eqa_19} and the second-order consistency of the spatial discretization, the initial errors satisfy
\begin{equation}\label{eq_initial_error_bound}
    \|\be_{\mathbf{Q}}^{0}\|_{H_h^1}
    +\|e_u^{0}\|_{H_h^2}
    +|e_s^{0}|
    =    \left|
    E_{1h}[\mathbf{Q}(0),  u(0)]
    -
    E_1[\mathbf{Q}(0),u(0)]
    \right|
    \le Ch^2,
\end{equation}
where $C$ is a positive constant independent of $\tau$ and $h$.

\begin{theorem}[Fully discrete error estimate]\label{them4_1}
Under the regularity assumptions in \eqref{eqa_19}, when $\tau$ and $h$
are sufficiently small, the GSAV-ETD2 scheme \eqref{eq4_6}--\eqref{eq4_12}
satisfies the following error estimate:
\ba
\|\be_{\mathbf{Q}}^{n}\|_{H_h^1}
 + \|e_u^{n}\|_{H_h^2}
 + |e_s^{n}|
\le  C_{err}(\tau^2 + h^2),
\qquad 0\le n\le N_T:=\lfloor T/\tau\rfloor.\label{eq4a}
\ed
\end{theorem}

Although Theorem~\ref{them3_1} provides unconditional dissipation of
the modified energy, the auxiliary variable $s_h$ enters
\eqref{energy_discrete} linearly. Consequently, without an a priori
lower bound for $s_h^n$, the discrete energy law alone does not yield
separate coercive bounds for the physical variables required in the
nonlinear error analysis. We therefore introduce a temporary bootstrap
lower bound for $s_h^n$, derive the necessary uniform estimates, and
subsequently close this assumption using the final error bound.

Let $M_s=\|s(t)\|_{L^\infty(0,T)}$ and choose $C_s=M_s+1$.
We make the following temporary bootstrap
assumption:
\begin{equation}
\label{eq:bootstrap_discrete}
s_h^n\ge -C_s,\qquad 0\le n\le m,\qquad \text{for}\quad 0\le m\le N_T.
\end{equation}
This is valid at $m=0$, since $s_h^0=E_1[\mathbf{Q}(0),u(0)]+e_s^0
\ge -M_s-Ch^2>-M_s-1=-C_s$ for sufficiently small $h$.
Once \eqref{eq4a} has been verified for $0\le n\le m+1$ under \eqref{eq:bootstrap_discrete}, a standard
induction argument extends the bound to all $0\le n\le N_T$.
\subsection{Uniform boundedness of the scalar factors $g_h^n$ and $g_h^{n,*}$}
~
\begin{lemma}\label{gn_bound0}
    Under the bootstrap assumption \eqref{eq:bootstrap_discrete}, the scalar factor $g_h^n$ defined in \eqref{eq_s_ex_n} admits positive bounds $G_*$ and $G^*$, independent of $\tau$ and $h$, such that
    \begin{equation}
        0 <G_*\le g_h^n \le G^*< \infty, \quad 0\le n\le m.
    \end{equation}
\end{lemma}
\begin{proof}
Using the discrete energy definition \eqref{energy_discrete},
the energy law established in Theorem~\ref{them3_1},
and the bootstrap bound \eqref{eq:bootstrap_discrete}, we obtain
    \ba\label{eq_a5}
   -C_s \le  s_h^n \le \mathcal{E}_h^n \le \mathcal{E}_h^0,
   \quad 0\le n\le m.
    \ed
It then follows from \eqref{energy_discrete} that
    \begin{equation}\label{eqa_p}
      \frac{K}{2} \|\nabla_h \mathbf{Q}_h^n\|_h^2  + B_0 \left\| \Delta_h u_h^n \right\|_h^2 + \frac \kappa2 (\|\mathbf{Q}_h^n\|_h^2 + \|u_h^n\|_h^2) \le \mathcal{E}_h^0 + C_s, \quad 0\le n\le m.
    \end{equation}
 Consequently, by the discrete norm equivalences, there exists a positive constant $C_1$, independent of $\tau$ and $h$, such that
    \begin{equation}\label{eqc7}
        \|\mathbf{Q}_h^n\|_{H_h^1}^2 + \|u_h^n\|_{H_h^2}^2 \le C_1(\mathcal{E}_h^0,C_s), \quad 0\le n\le m.
    \end{equation}
 Using \eqref{eqc7} and the discrete Sobolev embeddings,
 there exists a positive constant $\widetilde{C}^*$, independent of $\tau$ and $h$, such that
    \begin{align}
 | E_{1h}^n|&\le C\Bigl(
\left|\left\langle D_h^2u_h^n,\bM_h^nu_h^n\right\rangle_h\right|
+\|\bM_h^nu_h^n\|_h^2
+\left|\left\langle f_{\mathrm{bn}}(\mathbf Q_h^n),1\right\rangle_h\right|
+\left|\left\langle f_{\mathrm{bs}}(u_h^n),1\right\rangle_h\right|
+\|u_h^n\|_h^2+\|\mathbf Q_h^n\|_h^2\Bigr)\no\\
&\le C\Bigl(
\|D_h^2u_h^n\|_h\|\bM_h^n\|_h\|u_h^n\|_\infty
+\|\bM_h^n\|_h^2\|u_h^n\|_\infty^2
+1+\|\mathbf Q_h^n\|_{H_h^1}^4+\|u_h^n\|_\infty^4\Bigr)
\le \widetilde{C}^*.\label{eq_E1_bounds}
    \end{align}
    Combining \eqref{eq_a5} and \eqref{eq_E1_bounds}, we have
    \begin{equation*}
        0< G_*:= {\frac{\zeta(-C_s)}{\zeta(\widetilde{C}^*)}}\le g_h^n \le {\frac{\zeta(\mathcal{E}_h^0)}{\zeta(-\widetilde{C}^*)}} =: G^*, \quad 0\le n\le m,
    \end{equation*}
which completes the proof.
\end{proof}

\begin{lemma} \label{lem_analytic_semigroup}
    For any finite-dimensional self-adjoint positive-semidefinite operator $A$, the discrete analytic semigroup $e^{-tA}$ satisfies the smoothing estimate \cite{pazy2012semigroups}
    \begin{equation} \label{eq_smoothing_estimate}
        \vertiii{A^r e^{-tA}}_h\le \sup_{\lambda \ge 0} \lambda^r e^{-t\lambda} = C_r t^{-r}, \quad \forall t > 0, \; r \geq 0,
    \end{equation}
    where $\vertiii{\cdot}_h$ denotes the induced $\ell^2$ operator norm and $C_r = (r/e)^r$ depends only on $r$.
\end{lemma}

\begin{lemma}[Discrete fractional multiplication]\label{lem_discrete_fractional_product}
Let $d\le3$, $r>d/2$, and $0\le s\le r$. On the uniform periodic grid,
there is a constant $C=C(r,s,d,L_d)$, independent of $h$, such that
\begin{align}
\|v_hw_h\|_{H_h^s}
&\le C\|v_h\|_{H_h^r}\|w_h\|_{H_h^s},
\label{eq_discrete_fractional_multiplier}\\
\|v_h\|_\infty
&\le C\|v_h\|_{H_h^r},
\label{eq_discrete_fractional_embedding}
\end{align}
where $\|v_h\|_{H_h^s}
:=
\|(I-\Delta_h)^{s/2}v_h\|_h$.
The same estimates hold componentwise for tensor-valued grid
functions. Moreover, for every $0\le\rho<1$,
\begin{align}
c_\rho\|\mathbf V_h\|_{H_h^{2\rho}}
&\le \|\mLQ^\rho \mathbf V_h\|_h
\le C_\rho\|\mathbf V_h\|_{H_h^{2\rho}},\notag\\
c_\rho\|v_h\|_{H_h^{4\rho}}
&\le \|\mLu^\rho v_h\|_h
\le C_\rho\|v_h\|_{H_h^{4\rho}}.
\label{eq_operator_fractional_equivalence}
\end{align}
where $c_\rho,C_\rho>0$ are independent of $h$.
\end{lemma}
\begin{proof}
The discrete Fourier transform diagonalizes $-\Delta_h$, $\mLQ$, and
$\mLu$. On the fundamental frequency lattice, the symbol
$1+\lambda_h(\boldsymbol k)$ of $I-\Delta_h$ is uniformly equivalent
to $1+|\boldsymbol k|^2$. The discrete Fourier convolution formula,
followed by Cauchy--Schwarz, therefore gives
\eqref{eq_discrete_fractional_multiplier} and
\eqref{eq_discrete_fractional_embedding} with constants independent of
the number of grid points. Since the symbols
$\kappa+K\lambda_h(\boldsymbol k)$ and
$\kappa+2B_0\lambda_h(\boldsymbol k)^2$ are uniformly equivalent to
$1+\lambda_h(\boldsymbol k)$ and
$(1+\lambda_h(\boldsymbol k))^2$, respectively, the same spectral
comparison gives \eqref{eq_operator_fractional_equivalence}.
\end{proof}
\begin{lemma}\label{lem_regularit}
 Under the bootstrap assumption \eqref{eq:bootstrap_discrete}, there exists a constant $C>0$, independent of $\tau$ and $h$, such that the following regularity estimates hold for  $0\le n\le m$:
\begin{align}
\|\bQ^{n,*}\|_{H_h^p}
&\le
C\|\bQ^n\|_{H_h^p}
 + C_{\mathbf N}(C_1),\quad 1\le p\le 2 \label{eq:Q-regularity}\\
\|u_h^{n,*}\|_{H_h^2}
&\le
C(C_1,C_{\mathcal N}). \label{eq:u-regularity}
\end{align}
In particular, the predictor preserves the available spatial regularity:
an $H_h^p$-bound for $\bQ^n$ yields the corresponding $H_h^p$-bound
for $\bQ^{n,*}$ for every $1\le p\le2$.
\end{lemma}

\begin{proof}
    Lemma~\ref{gn_bound0} and the regularity estimate
\eqref{eqc7} imply that
    \ba\label{eqcf}
    \|\mNQ^n\|_h
    &\le
    C_{\mathbf N}(C_1),  \quad 0\le n\le m.
    \ed
Applying $\mLQ^{p/2}$ to both sides of \eqref{eq4_6} and taking
the discrete $L^2$ norm gives
\ba
\|\mLQ^{p/2}\bQ^{n,*}\|_h&
\le
\|\mLQ^{p/2}e^{-\tau\mLQ}\bQ^n\|_h
+
\tau\|\mLQ^{p/2}\varphi_1(\tau\mLQ)\mNQ^n\|_h\\
& \le C\|\bQ^n\|_{H_h^p}+\tau^{1-p/2}
\left\|
(\tau\mLQ)^{p/2-1}
\bigl(I-e^{-\tau\mLQ}\bigr)\mNQ^n
\right\|_h
\\
& \le C\|\bQ^n\|_{H_h^p}+C_{\mathbf N},
\ed
where the last inequality follows from the spectral bound
$\sup_{\lambda\ge0}\lambda^{p/2-1}(1-e^{-\lambda})\le 1$ with $\tau\le 1$.

We next establish a uniform bound for $\mNu^n$ in the discrete
negative norm $H_h^{-2}$. Recall that
\ba \label{eqzn}
\|v_h\|_{H_h^{-2}}
\coloneqq
\sup_{0\ne \phi_h\in E_h^{per}}
\frac{|\langle v_h,\phi_h\rangle_h|}
     {\|\phi_h\|_{H_h^2}}.
\ed
Since $\mLu$ is self-adjoint and positive definite, the preceding
coercivity estimate also yields, by duality, that
\begin{align}
\|\mLu^{-1/2}f_h\|_h
&=
\sup_{\phi_h\in E_h^{per}\setminus\{0\}}
\frac{
\left|\langle f_h,\phi_h\rangle_h\right|
}{
\|\mLu^{1/2}\phi_h\|_h
}
\le
C
\sup_{\phi_h\in E_h^{per}\setminus\{0\}}
\frac{
\left|\langle f_h,\phi_h\rangle_h\right|
}{
\|\phi_h\|_{H_h^2}
}
=
C\|f_h\|_{H_h^{-2}}.
\label{eq:Luminus_half_Hminus2}
\end{align}

We focus on the double-divergence term contained in $\mNu^n$.
Applying the discrete summation-by-parts formula twice and using
the discrete Sobolev embedding
$H_h^2\hookrightarrow L_h^\infty$, we obtain, for any
$\phi_h\in E_h^{per}$,
\begin{align}
&\left|
\left\langle
g_h^n\operatorname{div}_{h,c}^{\,2}
\bigl(\bM_h^n u_h^n\bigr),
\phi_h
\right\rangle_h
\right|+\left|
\left\langle
g_h^n\bM_h^n:D_h^2u_h^n,\phi_h
\right\rangle_h
\right|
\nonumber\\
&\qquad\le
|g_h^n|
\|\bM_h^n u_h^n\|_h
\|D_h^2\phi_h\|_h+|g_h^n|
\|\bM_h^n\|_h
\|D_h^2u_h^n\|_h
\|\phi_h\|_{\infty}\nonumber\\
&\qquad\le
C|g_h^n|
\bigl(1+\|\bQ^n\|_{H_h^1}\bigr)
\|u_h^n\|_{H_h^2}
\|\phi_h\|_{H_h^2}.
\label{eq:double-div-Hminus2}
\end{align}
The remaining lower-order terms are estimated in the same manner.
Therefore, Lemma~\ref{gn_bound0} and estimate~\eqref{eqc7} imply that
there exists a positive constant $C_{\mathcal N}$, independent of
$\tau$ and $h$, such that
\begin{equation}
\|\mNu^n\|_{H_h^{-2}}
\le
C|g_h^n|
\bigl(1+\|\bQ^n\|_{H_h^1}\bigr)
\|u_h^n\|_{H_h^2}
+C
\le C_{\mathcal N}(C_1,G^*),\quad 0\le n\le m.
\label{eq:Nu-Hminus2-bound}
\end{equation}
Applying $\mLu^{1/2}$ to both sides of \eqref{eq4_7} and invoking
\eqref{eq:Luminus_half_Hminus2} together with
\eqref{eq_operator_fractional_equivalence} for $\rho=1/2$, we obtain
\ba
\|u_h^{n,*}\|_{H_h^2}
\le
C\|\mLu^{1/2}u_h^{n,*}\|_h
&\le
C(\|u_h^n\|_{H_h^2}
+
\left\|
\tau\mLu\varphi_1(\tau\mLu)
\mLu^{-1/2}\mNu^n
\right\|_h)
\\
&=
C(\|u_h^n\|_{H_h^2}
+
\left\|
\bigl(I-e^{-\tau\mLu}\bigr)
\mLu^{-1/2}\mNu^n
\right\|_h)
\\
&\le
C(\|u_h^n\|_{H_h^2}
+
\|\mNu^n\|_{H_h^{-2}})\le C(C_1,C_{\mathcal N}),
\ed
which completes the proof of \eqref{eq:u-regularity}.
\end{proof}
\begin{lemma}\label{gn_bound1}
    Under the bootstrap assumption \eqref{eq:bootstrap_discrete}, there
    exist positive constants $\widetilde G_*$ and $\widetilde G^*$,
    independent of $\tau$ and $h$, such that the scalar factor
    $g_h^{n,*}$ defined in \eqref{eq_s_ex_n} satisfies
    \begin{equation}
        0<\widetilde G_*  \le  g_h^{n,*} \le\widetilde G^*< \infty, \quad 0\le n\le m.
    \end{equation}
\end{lemma}
\begin{proof}
By Theorem~\ref{them3_1} and \eqref{eq_proof1}, we have
    \ba\label{eq_a7}
     s_h^{n,*} \le \mathcal{E}_h^{n,*}\le \mathcal{E}_h^n \le \mathcal{E}_h^0,\quad \forall n \ge 0.
    \ed
By combining  \eqref{eq:bootstrap_discrete},
estimates \eqref{eqc7}, \eqref{eqcf} and \eqref{eq:Nu-Hminus2-bound}, and
Lemma~\ref{lem_regularit}, we obtain
\begin{align}
s_h^{n,*}
&=s_h^n
-\left\langle
\mNQ^n,\bQ^{n,*}-\bQ^n
\right\rangle_h
-\left\langle
\mNu^n,u_h^{n,*}-u_h^n
\right\rangle_h
\nonumber\\
&\ge
-C_s
-\|\mNQ^n\|_h
 \|\bQ^{n,*}-\bQ^n\|_h
-\|\mNu^n\|_{H_h^{-2}}
 \|u_h^{n,*}-u_h^n\|_{H_h^2}
\nonumber\\
&\ge
-C_s
-C\left(
\|\bQ^{n,*}-\bQ^n\|_h
+\|u_h^{n,*}-u_h^n\|_{H_h^2}
\right)
\nonumber\\
&\ge
-C_s-C
=:-\widetilde C_s,
\qquad 0\le n\le m.
\label{eq_s_star_bound}
\end{align}

By Lemma~\ref{lem_regularit} and \eqref{eqc7}, there exists a constant $C_2>0$, independent of $\tau$ and $h$, such that
    \begin{equation}\label{eq_a8}
        \|\bQ^{n,*}\|_{H_h^1}^2 + \|u_h^{n,*}\|_{H_h^2}^2 \le C_2(\mathcal{E}_h^0,C_s), \quad 0\le n\le m,
    \end{equation}
    which implies that $|E_{1h}^{n,*}|\leq \widehat C^*$ for some positive constant $\widehat C^*$, independent of $\tau$ and $h$.
    The argument used in the proof of Lemma~\ref{gn_bound0}, together with \eqref{eq_a7} and \eqref{eq_a8}, gives a positive constant $\widetilde G^*$, independent of $\tau$ and $h$, such that
    \begin{equation}\label{eqcd}
       0<\widetilde G_*  := \frac{\zeta(-\widetilde C_s)}{\zeta(\widehat{C}^*)} \le g_h^{n,*} = \frac{\zeta(s_h^{n,*})}{\zeta(E_{1h}^{n,*})} \le \frac{\zeta(\mathcal{E}_h^0)}{\zeta(-\widehat{C}^*)} =: \widetilde G^*, \quad 0\le n\le m.
    \end{equation}
Finally, combining \eqref{eq_a8} with \eqref{eqcd}, we obtain
    \begin{equation}\label{eqcg}
        \|\mNQ^{n,*}\|_h  \le \widetilde C_{\mathbf{N}}(C_2, \widetilde G^*), \quad 0\le n\le m,
    \end{equation}
    where $\widetilde C_{\mathbf{N}}$ is a positive constant independent of $\tau$ and $h$.
\end{proof}
\subsection{Higher regularity estimates for the numerical solution}
We next establish a direct regularity bootstrap estimate for the GSAV-ETD2 scheme.
\begin{theorem} \label{th4_1}
    Under the regularity assumptions in \eqref{eqa_19} and the bootstrap assumption \eqref{eq:bootstrap_discrete},
     there exist positive constants $\mathcal{M}_Q$, $\mathcal{M}_u$, and $\mathcal{M}_1$, independent of $\tau$ and $h$, such that the following uniform bounds hold:
    \ba
         \|\Delta_h\bQ^n\|_{h} &\le \mathcal{M}_Q,\quad   \norm{(-\Delta_h)^{2-\epsilon} \bu^n}_{h} \le \mathcal{M}_u, \\   \norm{\mNu^n}_h +\norm{\mNu^{n,*}}_h&\le \mathcal{M}_1, \quad 0\le n\le m,
    \ed
     where $\epsilon \in (0, 1/2)$ is fixed and sufficiently small.
\end{theorem}
\begin{proof}
Let $C_0 = \max\{\norm{ \bQ^0}_{H_h^2}, \norm{ \bu^0}_{H_h^4}\}$.
Iterating \eqref{eq4_9} over $k=0,\ldots,n-1$ gives the discrete Duhamel representation
    \begin{equation} \label{eq_discrete_sum}
        \bQ^n = e^{-t_n \mLQ} \bQ^0 + \sum_{k=0}^{n-1} e^{- (t_n - t_{k+1}) \mLQ} \int_0^\tau e^{-(\tau-\varsigma)\mLQ} ((1-\frac{\varsigma}{\tau})\mNQ^k + \frac{\varsigma}{\tau}\mNQ^{k,*}) \, d\varsigma,
    \end{equation}
    where $\mNQ^k=\mNQ(\bQ^k,u_h^k,s_h^k)$ and $\mNQ^{k,*} = \mNQ(\bQ^{k,*},u_h^{k,*},s_h^{k,*})$.
    We introduce the following piecewise linear interpolant on $[0,t_n]$:
    \begin{equation} \label{eq5a}
        \widetilde{\mNQz}_h(\sigma) := (1-\frac{\sigma-t_k}{\tau})\mNQ^k + \frac{\sigma-t_k}{\tau}\mNQ^{k,*}, \quad \text{for } \sigma \in [t_k, t_{k+1}), \quad k = 0, 1, \ldots, n-1.
    \end{equation}
Setting $\sigma=t_k+\varsigma$ in each integral and using  \eqref{eq5a}, we can rewrite \eqref{eq_discrete_sum} as
    \begin{equation} \label{eq_discrete_duhamel}
        \bQ^n = {e^{-t_n \mLQ} \bQ^0} + {\int_0^{t_n} e^{- (t_n - \sigma)\mLQ} \widetilde{\mNQz}_h(\sigma) \, d\sigma},\qquad 0\le n\le m.
    \end{equation}
    Define the self-adjoint positive-definite operator $\mathcal{A} = -K\Delta_h+ \frac{1}{2}\kappa I$.
   Applying $\mathcal{A}^{1 - \frac{\epsilon}{2}}$ to both sides of \eqref{eq_discrete_duhamel} and taking the norm $\|\cdot\|_h$, we deduce from Lemmas~\ref{lem_analytic_semigroup} and~\ref{gn_bound1} that
    \begin{equation}\label{eq_triangle_combined}
        \begin{aligned}
              \norm{\mathcal{A}^{1 - \frac{\epsilon}{2}} \bQ^n}_h
              & \le \norm{e^{-t_n\mLQ} \mathcal{A}^{1 - \frac{\epsilon}{2}} \bQ^0}_h + \int_0^{t_n} \norm{\mathcal{A}^{1-\frac{\epsilon}{2}} e^{- (t_n - \sigma)(-K\Delta_h+ \kappa )}  \widetilde{\mNQz}_h(\sigma)}_h \, d\sigma                             \\
             & \le \norm{ \mathcal{A}^{1 - \frac{\epsilon}{2}} \bQ^0}_h + \int_0^{t_n} \norm{\mathcal{A}^{1-\frac{\epsilon}{2}} e^{- (t_n - \sigma)\mathcal{A}} e^{-\frac{1}{2}\kappa (t_n - \sigma)} \widetilde{\mNQz}_h(\sigma)}_h \, d\sigma                             \\
             & \le \norm{\mathcal{A}^{1 - \frac{\epsilon}{2}} \bQ^0}_h + C_\epsilon \int_0^{t_n} (t_n - \sigma)^{-(1-\frac{\epsilon}{2})} e^{-\frac{1}{2}\kappa (t_n - \sigma)} \norm{\widetilde{\mNQz}_h(\sigma)}_h \, d\sigma \\
             & \le C_0 + C_\epsilon C_{\mathbf{N}} (\frac{1}{2}\kappa)^{-\frac{\epsilon}{2}} \int_0^{\infty} (\frac{1 }{2}\kappa \varsigma)^{-1+ \frac{\epsilon}{2}} e^{-\frac{1}{2}\kappa \varsigma} \, d(\frac{1}{2}\kappa \varsigma)                                                    \\
             & \le C_0 + C_\epsilon C_{\mathbf{N}} \, (\frac{1}{2}\kappa)^{-\frac{\epsilon}{2}} \Gamma\left(\frac{\epsilon}{2}\right)=: C_7(C_0, C_{\mathbf{N}}, \epsilon),
        \end{aligned}
    \end{equation}
    where $\Gamma(\cdot)$ denotes the gamma function and $C_7$ is a positive constant independent of $\tau$ and $h$.
Combining Lemma~\ref{lem_regularit} with
 \eqref{eq_triangle_combined}, we obtain
\begin{equation}\label{eq_Qn_star_bound}
    \norm{\mathcal{A}^{1 - \frac{\epsilon}{2}}\bQ^{n,*}}_{h} \le C\left(\norm{\mathcal{A}^{1 - \frac{\epsilon}{2}}\bQ^n}_{h} + C_{\mathbf{N}}\right) \le \widetilde C_7(C_7, C_{\mathbf{N}}, \epsilon), \quad 0\le n\le m.
\end{equation}

    We now estimate the smoothed nonlinear term $\ma^{-\frac{\ep}{2}}\mNu^n$.
  The commutativity of $\ma^{-\frac{\ep}{2}}$ and $\Delta_h$,
together with estimates \eqref{eqc7}, \eqref{eq_triangle_combined},
Lemma~\ref{lem_discrete_fractional_product}, and
\eqref{eq_discrete_hessian_adjoint}, gives
\begin{equation}
    \begin{aligned}
        \norm{\ma^{-\frac{\ep}{2}} \operatorname{div}_{h,c}^{\,2}(\bM_h \bu)}_h
         &\le  C\norm{ -\Delta_h( \ma^{-\frac{\ep}{2}} (\bM_h \bu))}_h \\
         &\le  C\norm{\ma^{1-\frac{\ep}{2}} (\bM_h \bu)}_h
         \\
         &\le C\|\bM_h\|_{H_h^{2-\epsilon}}
         \|\bu\|_{H_h^{2-\epsilon}}
         \le C(C_1,C_7).
    \end{aligned}
\end{equation}
Combining Lemmas~\ref{gn_bound0} and~\ref{gn_bound1} with
\eqref{eq_triangle_combined}, \eqref{eq_Qn_star_bound}, and the preceding bounds gives a constant $C_{\mathcal N}^*>0$,
independent of $h$ and $\tau$, such that
    \begin{equation} \label{eq_Nu_bound}
        \norm{\ma^{-\frac{\ep}{2}} \mNu^n}_{h} +\norm{\ma^{-\frac{\ep}{2}} \mNu^{n,*}}_{h}\le C_\mathcal{N}^*\left(C_1, C_7, \widetilde C_7\right).
    \end{equation}

    We next apply the same argument to the equation for $\bu^n$. Iterating \eqref{eq4_10} over $k=0,\ldots,n-1$ gives the representation
    \ba
        \bu^n = {e^{- t_n \mLu} \bu^0} + {\int_0^{t_n} e^{- (t_n - \sigma)\mLu} \widetilde{\mNu}(\sigma) \, d\sigma},\quad 0\le n\le m, \label{eq_v}
    \ed
    where
     $\widetilde{\mNu}(\sigma) = (1-\frac{\sigma-t_k}{\tau})\mNu^k + \frac{\sigma-t_k}{\tau}\mNu^{k,*}$ for $\sigma \in [t_k, t_{k+1}), k = 0, 1, \ldots, n-1$.
    Define the self-adjoint positive-semidefinite operator $\mathcal{B} =2B_0 \Delta^2_h$.
    Applying $\ma^{-\frac{\ep}{2}}\mathcal{B}^{1-\frac{\epsilon}{4}}$ to both sides of \eqref{eq_v} and taking the norm $\|\cdot\|_h$, we deduce from Lemma~\ref{lem_analytic_semigroup} and \eqref{eq_Nu_bound} that
    \ba \label{eq_combined_estimate}
            \norm{\ma^{-\frac{\ep}{2}}\mathcal{B}^{1-\frac{\epsilon}{4}} \bu^n}_{h}
             & \le \norm{e^{- t_n \mLu}\ma^{-\frac{\ep}{2}}\mathcal{B}^{1-\frac{\epsilon}{4}} \bu^0}_{h} + \int_0^{t_n} \norm{ \mathcal{B}^{1-\frac{\epsilon}{4}} e^{- (t_n - \sigma)\mathcal{B}}e^{-\kappa  (t_n-\sigma)} \ma^{-\frac{\ep}{2}} \widetilde{\mNu}(\sigma) }_{h} \, d\sigma      \\
             & \le  \norm{\ma^{-\frac{\ep}{2}}\mathcal{B}^{1-\frac{\epsilon}{4}} \bu^0}_{h} + C_\epsilon \int_0^{t_n} (t_n-\sigma)^{-\left(1 - \frac{\epsilon}{4}\right)} e^{-\kappa(t_n-\sigma)} \norm{\ma^{-\frac{\ep}{2}}\widetilde{\mNu}(\sigma)}_{h} \, d\sigma \\
             & \le C_0 +  C_\epsilon C_\mathcal{N}^* \int_0^{\infty} \varsigma^{-1 + \frac{\epsilon}{4}} e^{-\kappa \varsigma} \, d\varsigma                                                                                                 \\
             & \le C_0 +  C_\epsilon C_\mathcal{N}^*(\kappa)^{-\frac{\epsilon}{4}} \Gamma\left(\frac{\epsilon}{4}\right) \le C(C_0,C_{\mathcal{N}}^*,\epsilon).
    \ed
    Estimates \eqref{eq_combined_estimate} and \eqref{eq_Nu_bound} yield a constant $\mathcal{M}_u > 0$, independent of $\tau$ and $h$, such that
    \begin{equation}\label{eq4_13}
          \norm{(-\Delta_h)^{2-\epsilon}\bu^n}_{h}\le C\norm{\ma^{-\frac{\ep}{2}}\mathcal{B}^{1-\frac{\epsilon}{4}}\bu^n}_{h} \le \mathcal{M}_u\left(C_0, C_\mathcal{N}^*, \epsilon \right),\quad 0\le n\le m.
    \end{equation}
At the predictor stage, combining Lemma~\ref{lem_analytic_semigroup}
with \eqref{eq_Nu_bound} and \eqref{eq4_13} yields the following
uniform estimate:
    \begin{align}
      \|(-\Delta_h)^{2-\epsilon}\bu^{n,*}\|_h&\le C\|\mathcal A^{-\epsilon/2}
    \mathcal B^{1-\epsilon/4}\bu^{n,*}\|_h  \nonumber\\
    &\le
    C\|\mathcal A^{-\epsilon/2}
    \mathcal B^{1-\epsilon/4}\bu^n\|_h
    +C\|\mathcal A^{-\epsilon/2}\mNu^n\|_h
    \int_0^\tau(\tau-r)^{-1+\epsilon/4}\,dr
    \le \widetilde{\mathcal M}_u,
    \label{eq_u_star_fractional_high}
    \end{align}
    where $\widetilde{\mathcal M}_u$ is a positive constant independent of $\tau$ and $h$.
Combining \eqref{eq4_13} and
\eqref{eq_u_star_fractional_high} with
\eqref{eq_triangle_combined} and \eqref{eq_Qn_star_bound}, we conclude
that there exists a constant $C_{\mathbf N,\epsilon}>0$, independent
of $\tau$ and $h$, such that
    \begin{equation}
    \|\mathcal A^{\epsilon/2}\mNQ^n\|_h
    +\|\mathcal A^{\epsilon/2}\mNQ^{n,*}\|_h
    \le C_{\mathbf N,\epsilon},
    \qquad 0\le n\le m.
    \label{eqa4_13}
    \end{equation}

    We now return to the Duhamel formula \eqref{eq_discrete_duhamel}.
    Applying $\mathcal A$ to that formula, splitting
    $\mathcal A=\mathcal A^{1-\epsilon/2}
    \mathcal A^{\epsilon/2}$ in the source term, and using
    $\mLQ=\mathcal A+\frac\kappa2I$, we obtain
    \begin{align}
    \|\mathcal A\bQ^n\|_h
    &\le \|\mathcal A\bQ^0\|_h
    +\int_0^{t_n}
    \vertiii{\mathcal A^{1-\epsilon/2}
    e^{-(t_n-\sigma)\mathcal A}}_h
    e^{-\frac\kappa2(t_n-\sigma)}
    \|\mathcal A^{\epsilon/2}
    \widetilde{\mNQz}_h(\sigma)\|_h\,d\sigma
    \nonumber\\
    &\le C_0+C C_{\mathbf N,\epsilon}
    \int_0^\infty
    r^{-1+\epsilon/2}e^{-\frac\kappa2r}\,dr
    \le \mathcal M_Q,
    \qquad 0\le n\le m.
    \label{eqa4_14}
    \end{align}
 By the spectral equivalence between $\mathcal A$ and $I-\Delta_h$
 and Lemma~\ref{lem_regularit}, we have
    \[
    \|\bQ^n\|_{H_h^2}+\|\bQ^{n,*}\|_{H_h^2}\le C(\|\mathcal A\bQ^n\|_h+C_{\mathbf{N}})
    \le C\mathcal M_Q.
    \]
Finally, we estimate the chemical potential $\mNu^n$ and its predictor $\mNu^{n,*}$ as follows:
    \begin{align}
    \|\mNu^n\|_h+\|\mNu^{n,*}\|_h
    &\le
    G^*\left(\|\mu_{\rm p}^n\|_h
    +\|\mu_{{\rm d},h}^n\|_h\right)
    +\widetilde G^*\left(\|\mu_{\rm p}^{n,*}\|_h
    +\|\mu_{{\rm d},h}^{n,*}\|_h\right)
    \nonumber\\
    &\le \mathcal M_1,
    \qquad 0\le n\le m,
    \label{eq_Nu_uniform_L2_final}
    \end{align}
    where $\mathcal M_1$ depends on
    $C_1,\mathcal M_Q,\mathcal M_u,G^*$, and $\widetilde G^*$, but is
    independent of $h$ and $\tau$.
\end{proof}
\subsection{ Lipschitz estimates and predictor-corrector error bounds}
~
\begin{lemma}
\label{lem_parabolic_regularity_bootstrap}
Under the regularity assumption \eqref{eqa_19}, the solution of the
continuous system \eqref{eq1_9} satisfies the following higher-order
spatial regularity estimates:
\begin{align*}
\mathbf Q
&\in L^\infty
 \bigl(0,T;\mathbf H^6(\Omega,\mathcal S^{(d)})\bigr),
&
\mathbf Q_t
&\in L^\infty
 \bigl(0,T;\mathbf H^4(\Omega,\mathcal S^{(d)})\bigr),
\\
u
&\in L^\infty(0,T;H^8(\Omega)),
&
u_t
&\in L^\infty(0,T;H^6(\Omega)).
\end{align*}
Moreover, the nonlinear terms satisfy
\begin{align*}
\partial_t\mNQz
&\in L^\infty
 \bigl(0,T;\mathbf H^2(\Omega,\mathcal S^{(d)})\bigr),
&
\partial_t\mNuz
&\in L^\infty(0,T;H^2(\Omega)),
\\
\partial_{tt}\mNQz
&\in L^\infty
 \bigl(0,T;\mathbf L^2(\Omega,\mathcal S^{(d)})\bigr),
&
\partial_{tt}\mNuz
&\in L^\infty(0,T;L^2(\Omega)).
\end{align*}
\end{lemma}

The proof of Lemma~\ref{lem_parabolic_regularity_bootstrap} is given in
Appendix~\ref{app:proof_parabolic_regularity}.
\begin{lemma} \label{thm_full_lipschitz}
 Combining the regularity assumptions in \eqref{eqa_19} with
Theorem~\ref{th4_1}, we obtain the following Lipschitz continuity
estimates for $\mNQ^n$, $\mNu^n$, $\mNQ^{n,*}$, and $\mNu^{n,*}$:
    \ba\label{eq_full_Lip_Q}
    \| \mNQ^n - \mNQz(t_n)\|_h \le C \left( |e_s^n| + \|\be_{\mathbf{Q}}^{n}\|_{H^1_h} + \|e_{u}^{n}\|_{H^2_h} + h^2 \right),\quad 0 \le n \le m,\\
    \|\mNu^n - \mNuz(t_n)\|_h \le C \left( |e_s^n| + \|\be_{\mathbf{Q}}^{n}\|_{H^2_h} + \|e_{u}^{n}\|_{H^2_h} + h^2 \right),\quad 0 \le n \le m,\\
    \| \mLu^{-\frac{1}{2}}(\mNu^n - \mNuz(t_n)) \|_h \le C \left( |e_s^n| + \|\be_{\mathbf{Q}}^{n}\|_{H^1_h} + \|e_{u}^{n}\|_{H^2_h} + h^2 \right),\quad 0 \le n \le m,\\
    \| \mNQ^{n,*} - \mNQz(t_{n+1})\|_h \le C \left( |e_s^{n,*}| + \|\be_{\mathbf{Q}}^{n,*}\|_{H^1_h} + \|e_{u}^{n,*}\|_{H^2_h} + h^2 \right),\quad 0 \le n \le m,\\
    \|\mNu^{n,*} - \mNuz(t_{n+1})\|_h \le C \left( |e_s^{n,*}| + \|\be_{\mathbf{Q}}^{n,*}\|_{H^2_h} + \|e_{u}^{n,*}\|_{H^2_h} + h^2 \right),\quad 0 \le n \le m,
    \ed
    where $C>0$ is independent of $\tau$ and $h$.
\end{lemma}

The proof of Lemma~\ref{thm_full_lipschitz} follows from standard local
Lipschitz and discrete Sobolev estimates and is therefore omitted.

\begin{lemma} \label{lem4_2}
Under the regularity assumptions in \eqref{eqa_19} and Lemma \ref{lem_parabolic_regularity_bootstrap},  we have the
following error estimate for the prediction step \eqref{eq4_6}--\eqref{eq4_8}:
    \begin{align}
         &\| \be_{\mathbf{Q}}^{n,*} \|_{H^2_h}   \le C \left( \|\be_{\mathbf{Q}}^{n}\|_{H^2_h}  + \|e_{u}^{n}\|_{H^2_h} +|e_s^n|+ \tau^2 + h^2 \right),\label{eq_error_q_star_bound}\\
        &\| \be_{\mathbf{Q}}^{n,*} \|_{H^1_h}  + \| e_u^{n,*} \|_{H^2_h}  + |e_s^{n,*}|    \le C \left( \|\be_{\mathbf{Q}}^{n}\|_{H^1_h} + \|e_{u}^{n}\|_{H^2_h} +|e_s^n|+ \tau^2 + h^2  \right).\label{eq_error_s_star_bound}
    \end{align}
\end{lemma}
\begin{proof}
Subtracting the corresponding exact ETD relations over
$[t_n,t_{n+1}]$ from
\eqref{eq4_6}--\eqref{eq4_7} yields the following predictor-stage error
equations:
\begin{align}
\be_{\mathbf{Q}}^{n,*}&= e^{-\tau \mLQ} \be_{\mathbf{Q}}^{n} + \tau \varphi_1(\tau \mLQ) (\mNQ^n - \mNQz(t_n)) -  \int_0^\tau e^{-(\tau - \varsigma) \mLQ}  {\mathbf{T}}_{\mathbf{Q}}^{n,*}(\varsigma) d\varsigma,\label{eq_error_Q_star}\\
e_u^{n,*}&= e^{-\tau \mLu} e_u^{n} + \tau \varphi_1(\tau \mLu) (\mNu^n - \mNuz(t_n)) -   \int_0^\tau e^{-(\tau - \varsigma) \mLu}  {T}_u^{n,*}(\varsigma) d\varsigma,\label{eq_error_u_star}
\end{align}
where ${\mathbf{T}}_{\mathbf{Q}}^{n,*}(\varsigma)$ and ${T}_u^{n,*}(\varsigma)$ are defined by
\begin{align*}
    \mathbf T_{\mathbf Q}^{n,*}(\varsigma)
&:=
\mNQz(t_n+\varsigma)-\mNQz(t_n)
+K\left(
\Delta\mathbf Q(t_n+\varsigma)
-\Delta_h\mathbf Q(t_n+\varsigma)
\right),\\
T_u^{n,*}(\varsigma)
&:=
\mNuz(t_n+\varsigma)-\mNuz(t_n)
+2B_0\left(
\Delta_h^2u(t_n+\varsigma)
-\Delta^2u(t_n+\varsigma)
\right),
\end{align*}
Using the regularity assumptions in \eqref{eqa_19} and Lemma \ref{lem_parabolic_regularity_bootstrap}, we have
\begin{align*}
\|\mLQ\mathbf T_{\mathbf Q}^{n,*}(\varsigma)\|_h
&\le
\int_0^\varsigma
\|\mLQ\partial_t\mathbf N(t_n+r)\|_h\,dr
+\|\mLQ(\mLQ-\mathcal L)
\mathbf Q(t_n+\varsigma)\|_h\\
&\le
C\tau\|\partial_t\mathbf N\|_{L^\infty(0,T;\mathbf H^2)}
+Ch^2\|\mathbf Q\|_{L^\infty(0,T;\mathbf H^6)}
\le C(\tau+h^2),\\
\|T_u^{n,*}(\varsigma)\|_h
&\le
C\tau\|\partial_t\mathcal N\|_{L^\infty(0,T;L^2)}
+Ch^2\|u\|_{L^\infty(0,T;H^6)},\\
\|\mLu^{1/2}T_u^{n,*}(\varsigma)\|_h
&\le
C\tau\|\partial_t\mathcal N\|_{L^\infty(0,T;H^2)}
+Ch^2\|u\|_{L^\infty(0,T;H^8)}.
\end{align*}

Applying the operator $\frac{1}{\tau}(\mq(\tau \mLQ)+\tau \mLQ)$ to \eqref{eq_error_Q_star} gives
\[
\psi(\tau\mLQ)\eaa+\mLQ\be_{\mathbf{Q}}^{n,*}
=\mNQ^n-\mNQz(t_n)-\widetilde{\mathbf{T}}_{\mathbf{Q}}^{n,*},
\]
 where  $\widetilde{ \mathbf{T}}_{\mathbf{Q}}^{n,*}= \int_0^\tau \mathcal K_{\mLQ}(\varsigma)  \mathbf{T}_{\mathbf{Q}}^{n,*}(\varsigma) d\varsigma$ with
\[
\mathcal K_{\Lambda_h}(\varsigma)
:=
\frac{1}{\tau}
\bigl(\mq(\tau\Lambda_h)+\tau\Lambda_h\bigr)
e^{-(\tau-\varsigma)\Lambda_h},\quad \Lambda_h\in\{\mLQ,\mLu\}.
\]
Since $\Lambda_h$ is self-adjoint and positive definite, we can calculate
\[
\int_0^\tau\mathcal K_{\Lambda_h}(\varsigma)\,d\varsigma=I,
\qquad
\vertiii{\int_r^\tau\mathcal K_{\Lambda_h}(\varsigma)\,d\varsigma}_h
\le1,\quad 0  \le r \le \tau.
\]
Using
$\mathbf{T}_{\mathbf{Q}}^{n,*}(\varsigma)=\mathbf{T}_{\mathbf{Q}}^{n,*}(0)+\int_0^\varsigma\partial_r\mathbf{T}_{\mathbf{Q}}^{n,*}(r)\,dr$, along with \eqref{eqa_19} and Lemma \ref{lem_parabolic_regularity_bootstrap}, we obtain
\ba
\norm{\mLQ^{\frac{1}{2}}\widetilde{ \mathbf{T}}_{\mathbf{Q}}^{n,*}}_h
={}&
\norm{\left(
\int_0^\tau
\mathcal K_{\mLQ}(\varsigma)\,d\varsigma
\right)\mLQ^{\frac{1}{2}}\mathbf{T}_{\mathbf{Q}}^{n,*}(0)
+
\int_0^\tau
\mathcal K_{\mLQ}(\varsigma)
\left(
\int_0^\varsigma
\partial_r(\mLQ^{\frac{1}{2}}\mathbf{T}_{\mathbf{Q}}^{n,*}(r))\,dr
\right)d\varsigma}_h
\nonumber\\
\le{}&
\norm{\mLQ^{\frac{1}{2}}\mathbf{T}_{\mathbf{Q}}^{n,*}(0)}_h
+
\norm{\int_0^\tau
\left(
\int_r^\tau
\mathcal K_{\mLQ}(\varsigma)\,d\varsigma
\right)
\partial_r(\mLQ^{\frac{1}{2}}\mathbf{T}_{\mathbf{Q}}^{n,*}(r))\,dr}_h\nonumber\\
\le{}&
\|\mLQ^{1/2}\mathbf T_{\mathbf Q}^{n,*}(0)\|_h
+\int_0^\tau
\|\mLQ^{1/2}\partial_r\mathbf T_{\mathbf Q}^{n,*}(r)\|_h\,dr
\nonumber\\
\le{}& C\|\mathbf Q(t_n)\|_{ H^5}h^2+  C\left(
\|\partial_{\varsigma}\mathbf N(t_n+\varsigma)\|_{ H^1}
+\|\mathbf Q_{\varsigma}(t_n+\varsigma)\|_{ H^4}
\right)\tau\le C(h^2+\tau).
\ed

Thus the preceding estimates can be summarized as
\ba \label{eqcv}
\|\mLQ^{\frac{1}{2}}\widetilde{\mathbf{T}}_{\mathbf{Q}}^{n,*}\|_{h}+\|\mLQ{\mathbf{T}}_{\mathbf{Q}}^{n,*}(\varsigma)\|_{h}
+\|T_u^{n,*}(\varsigma)\|_{h}
+\|\mLu^{\frac{1}{2}}T_u^{n,*}(\varsigma)\|_{h}
\le C(\tau+h^2), \quad  0\le\varsigma\le\tau,
\ed
where $C$ is independent of $h$ and $\tau$.

Taking the discrete inner product of the first prediction-step error equation with $\tau\mLQ \be_{\mathbf{Q}}^{n,*}$ gives
    \ba \label{eqa_18}
    &\left\langle  \psi(\tau \mLQ) \eaa, \tau\mLQ \be_{\mathbf{Q}}^{n,*} \right\rangle_h + \tau\left\langle \mLQ \be_{\mathbf{Q}}^{n,*}, \mLQ \be_{\mathbf{Q}}^{n,*} \right\rangle_h
    = \left\langle   \mNQ^n - \mNQz(t_n) - \widetilde{ \mathbf{T}}_{\mathbf{Q}}^{n,*}, \tau\mLQ \be_{\mathbf{Q}}^{n,*} \right\rangle_h.
    \ed
    Applying the algebraic identity and \eqref{eq_inner_Q} to the left-hand side of \eqref{eqa_18} yields
    \begin{equation} \label{eq_LHS_expanded_norm_star}
        \begin{aligned}
            \text{LHS} & = \frac{1}{2} \left( \icQ{\be_{\mathbf{Q}}^{n,*}} - \icQ{\be_{\mathbf{Q}}^{n}}+\icQ{\be_{\mathbf{Q}}^{n,*}-\be_{\mathbf{Q}}^{n}}+ \frac{1}{2}\tau\norm{\mLQ\be_{\mathbf{Q}}^{n,*}}_h^2 \right)
             + \frac{3}{4}\tau\norm{\mLQ\be_{\mathbf{Q}}^{n,*}}_h^2 \\
             & = \frac{1}{2} \left( \idQ{\be_{\mathbf{Q}}^{n,*}} - \icQ{\be_{\mathbf{Q}}^{n}}+\icQ{\be_{\mathbf{Q}}^{n,*}-\be_{\mathbf{Q}}^{n}} \right)
             + \frac{3}{4}\tau\norm{\mLQ\be_{\mathbf{Q}}^{n,*}}_h^2.
        \end{aligned}
    \end{equation}
 Applying Young's inequality and using
Lemma~\ref{thm_full_lipschitz} together with
estimate~\eqref{eqcv}, we obtain the following bound for the
right-hand side of \eqref{eqa_18}:
    \begin{equation} \label{eq_RHS_final_bound_star}
        \begin{aligned}
            \text{RHS} & \leq C\tau \ia{\mNQ^n - \mNQz(t_n)}  + \frac{1}{4} \tau\|\mLQ \be_{\mathbf{Q}}^{n,*}\|_h^2  -\left\langle    \tau\mLQ^{\frac{1}{2}}\widetilde{ \mathbf{T}}_{\mathbf{Q}}^{n,*}, \mLQ^{\frac{1}{2}} \be_{\mathbf{Q}}^{n,*} \right\rangle_h \\
                & \le C\tau \left( |e_s^n|^2 + \|\be_{\mathbf{Q}}^{n}\|_{H^1_h}^2 + \|e_{u}^{n}\|_{H^2_h}^2 + h^4 \right) + \frac{1}{4}\tau \|\mLQ \be_{\mathbf{Q}}^{n,*}\|_h^2  \\
                & \quad + C\tau^2 (\tau + h^2)^2 + \frac{1}{4}  \|\be_{\mathbf{Q}}^{n,*}\|_{\mLQ }^2.
        \end{aligned}
    \end{equation}
    Combining \eqref{eq_LHS_expanded_norm_star} and
    \eqref{eq_RHS_final_bound_star}, and invoking the norm comparisons
    \eqref{eq_norm_bounds_1} with $\tau \le1$, we obtain
    \begin{equation}
        \begin{aligned}
             \frac{1}{4}\idQ{\be_{\mathbf{Q}}^{n,*}}&
              \le\frac{1}{2}\icQ{\be_{\mathbf{Q}}^{n}}  + C\tau \left( |e_s^n|^2 + \|\be_{\mathbf{Q}}^{n}\|_{H^1_h}^2 + \|e_{u}^{n}\|_{H^2_h}^2 + h^4 \right) + C\tau^2 (\tau + h^2)^2 \\
              &\le C \left( |e_s^n|^2 + \|\be_{\mathbf{Q}}^{n}\|_{H^1_h}^2 + \|e_{u}^{n}\|_{H^2_h}^2  + \tau^4+ h^4 \right).
        \end{aligned}
    \end{equation}
Using \eqref{eq_psi_functions}, we obtain
    \begin{equation}\label{eq8a}
        \begin{aligned}
            \|\be_{\mathbf{Q}}^{n,*}\|_{H^1_h}^2 \le C \idQ{\be_{\mathbf{Q}}^{n,*}}\le C \left( |e_s^n|^2 + \|\be_{\mathbf{Q}}^{n}\|_{H^1_h}^2 + \|e_{u}^{n}\|_{H^2_h}^2  + \tau^4+ h^4 \right),
        \end{aligned}
    \end{equation}

Applying  $\mLQ$ to \eqref{eq_error_Q_star} yields
\begin{equation}
\mLQ \be_{\mathbf{Q}}^{n,*} = e^{-\tau \mLQ} \mLQ \be_{\mathbf{Q}}^{n} + \left( I - e^{-\tau \mLQ} \right) \left( \mNQ^n - \mNQz(t_n) \right) - \int_0^\tau e^{-(\tau - \varsigma) \mLQ} \left[ \mLQ {\mathbf{T}}_{\mathbf{Q}}^{n,*}(\varsigma) \right] d\varsigma.
\end{equation}
Taking the $\|\cdot\|_h$ norm on both sides and applying Lemma~\ref{thm_full_lipschitz} and \eqref{eqcv} gives
\ba
\| \mLQ \be_{\mathbf{Q}}^{n,*} \|_h \le& \| \mLQ \be_{\mathbf{Q}}^{n} \|_h + \| \mNQ^n - \mNQz(t_n) \|_h + \int_0^\tau \left\| \mLQ {\mathbf{T}}_{\mathbf{Q}}^{n,*}(\varsigma) \right\|_h d\varsigma\\
\le &\| \mLQ \be_{\mathbf{Q}}^{n} \|_h + C \left( |e_s^n| + \|\be_{\mathbf{Q}}^{n}\|_{H^1_h} + \|e_{u}^{n}\|_{H^2_h} + h^2 \right) + C\tau (\tau + h^2)\\
\le & C \left( \|\be_{\mathbf{Q}}^{n}\|_{H^2_h} + \|e_{u}^{n}\|_{H^2_h} +|e_s^n|+ \tau^2 + h^2\right),
\ed
which completes the estimate \eqref{eq_error_q_star_bound}.

Applying  $\mLu^{\frac{1}{2}}$ to \eqref{eq_error_u_star} yields
\begin{equation}\label{eq7a}
\mLu^{\frac{1}{2}} e_u^{n,*} = e^{-\tau \mLu} \mLu^{\frac{1}{2}} e_u^{n} + \left( I - e^{-\tau \mLu} \right)\mLu^{-\frac{1}{2}} \left( \mNu^n - \mNuz(t_n) \right) - \int_0^\tau e^{-(\tau - \varsigma) \mLu} \left[ \mLu^{\frac{1}{2}} T_u^{n,*}(\varsigma) \right] d\varsigma.
\end{equation}
Taking the $\|\cdot\|_h$ norm on both sides of \eqref{eq7a} and applying Lemma~\ref{thm_full_lipschitz} and \eqref{eqcv} gives
\ba
\| \mLu^{\frac{1}{2}} e_u^{n,*} \|_h \le& \| \mLu^{\frac{1}{2}} e_u^{n} \|_h + \| \mLu^{-\frac{1}{2}}(\mNu^n - \mNuz(t_n)) \|_h + \int_0^\tau \left\| \mLu^{\frac{1}{2}} T_u^{n,*}(\varsigma) \right\|_h d\varsigma\\
\le &\| \mLu^{\frac{1}{2}} e_u^{n} \|_h + C \left( |e_s^n| + \|\be_{\mathbf{Q}}^{n}\|_{H^1_h} + \|e_{u}^{n}\|_{H^2_h} + h^2 \right) + C\tau (\tau + h^2)\\
\le & C \left(  \|\be_{\mathbf{Q}}^{n}\|_{H^1_h}+\|e_{u}^{n}\|_{H^2_h} + |e_s^n| + \tau^2 + h^2 \right).
\ed
By the spectral equivalence between $\mLu^{1/2}$ and the discrete
$H_h^2$ norm, we have
\[
\| e_u^{n,*} \|_{H^2_h}\le C \| \mLu^{\frac{1}{2}} e_u^{n,*} \|_h \le C \left(  \|\be_{\mathbf{Q}}^{n}\|_{H^1_h}+\|e_{u}^{n}\|_{H^2_h} + |e_s^n| + \tau^2 + h^2 \right).
\]

For the prediction step, the scalar error equation corresponding to
\eqref{eq4_8} is
\ba\label{eqac}
\frac{1}{\tau}(e_s^{n,*}-e_s^n)
&= -\langle \mNQ^n,\eaa\rangle_h
-\langle \mNu^n,\eba\rangle_h
-\left\langle \mNQ^n-\mNQz(t_n),\da\right\rangle_h\\
&\quad
-\left\langle \mNu^n-\mNuz(t_n),\db\right\rangle_h
-T_{s}^{n,*},
\ed
where
\ba
T_s^{n,*}
:={}&
\delta_\tau s(t_{n+1})
+\left\langle
\mNQz(t_n),\delta_\tau\mathbf Q(t_{n+1})
\right\rangle_h
+\left\langle
\mNuz(t_n),\delta_\tau u(t_{n+1})
\right\rangle_h
\ed
and the temporal consistency and spatial quadrature estimates give
$|T_s^{n,*}|\le C(\tau+h^2)$.
Multiplying \eqref{eqac} by $2\tau{e}_s^{n,*}$ gives
    \begin{equation} \label{eq_es_identitya}
        \begin{aligned}
             &  |{e}_s^{n,*}|^2 - |e_s^n|^2 + |{e}_s^{n,*} - e_s^n|^2
            = {-2 \tau{e}_s^{n,*} \Big[ \langle \mNQ^n, \eaa \rangle_h + \langle \mNu^n, \eba \rangle_h \Big]}                  \\
             & \qquad- {2\tau{e}_s^{n,*} \Big[ \left\langle \mNQ^n-\mNQz(t_n), \da  \right\rangle_h + \left\langle \mNu^n-\mNuz(t_n), \db \right\rangle_h + T_{s}^{n,*}  \Big]}:= \mathcal{J}_3 + \mathcal{J}_4.
        \end{aligned}
    \end{equation}
    Applying  Young's inequality and invoking
Lemma~\ref{gn_bound0} and Theorem~\ref{th4_1}, we obtain
    \begin{equation} \label{eq_I2_bounda}
        \begin{aligned}
            \mathcal{J}_3 & \le 2 \mathcal{M}_2 |{e}_s^{n,*}| \left( \|\be_{\mathbf{Q}}^{n,*}-\be_{\mathbf{Q}}^n\|_h + \|e_{u}^{n,*}-e_{u}^n\|_h \right)                           \\
                          & \le \frac{1}{4} |{e}_s^{n,*}|^2 + C (\|\be_{\mathbf{Q}}^{n,*}-\be_{\mathbf{Q}}^n\|_h^2 +  \|e_{u}^{n,*}-e_{u}^n\|_h^2) \\
                          & \le \frac{1}{4} |{e}_s^{n,*}|^2 + C \left( \|\be_{\mathbf{Q}}^{n,*}\|_h^2 + \|\be_{\mathbf{Q}}^{n}\|_h^2 + \|e_{u}^{n,*}\|_h^2 + \|e_{u}^{n}\|_h^2 \right),
        \end{aligned}
    \end{equation}
where
$ \mathcal{M}_2=\max\{ C_{\mathbf{N}},\widetilde C_{\mathbf{N}},\mathcal{M}_1\} $.

    The Cauchy--Schwarz inequality, Young's inequality, and
    Lemma~\ref{thm_full_lipschitz} yield
    \begin{equation} \label{eq_I1_bounda}
        \begin{aligned}
            \mathcal{J}_4 & \le  2|{e}_s^{n,*}|\cdot   \left(C\|\mathbf{Q}(t_{n+1})-\mathbf{Q}(t_n)\|_h \left( |e_s^n| + \|\be_{\mathbf{Q}}^{n}\|_{H_h^1} + \|e_{u}^{n}\|_{H_h^2} + h^2 \right)\right.                                       \\
                          & \qquad \left.
            + C \|\mLu^{\frac{1}{2}}(u(t_{n+1})-u(t_n))\|_h  \| \mLu^{-\frac{1}{2}}(\mNu^n - \mNuz(t_n)) \|_h + C\tau(\tau + h^2) \right)            \\
            & \le  2|{e}_s^{n,*}|\cdot   \left(C \left( |e_s^n| + \|\be_{\mathbf{Q}}^{n}\|_{H_h^1} + \|e_{u}^{n}\|_{H_h^2} + h^2 \right)\right.                                       \\
                          & \qquad \left.
            + C \|(u(t_{n+1})-u(t_n))\|_{H_h^2} \left( |e_s^n| + \|\be_{\mathbf{Q}}^{n}\|_{H_h^1} + \|e_{u}^{n}\|_{H_h^2} + h^2 \right) + C\tau(\tau + h^2) \right)                                                                  \\
                          & \le\frac{1}{4}|{e}_s^{n,*}|^2 + C \left(   |e_s^n|^2 + \|\be_{\mathbf{Q}}^{n}\|_{H_h^1}^2 + \|e_{u}^{n}\|_{H_h^2}^2 + \tau^4 + h^4  \right).
        \end{aligned}
    \end{equation}
    Substituting the estimates for $\mathcal{J}_3$ and $\mathcal{J}_4$ into \eqref{eq_es_identitya} yields
    \begin{equation}
        \begin{aligned}
              |{e}_s^{n,*}|^2 - |e_s^n|^2 + |{e}_s^{n,*} - e_s^n|^2  &\le \frac{1}{2} |{e}_s^{n,*}|^2 + C\left(   |e_s^n|^2 + \|\be_{\mathbf{Q}}^{n}\|_{H_h^1}^2 + \|e_{u}^{n}\|_{H_h^2}^2 + \tau^4 + h^4   \right)\\
             & \quad + C \left( \|\be_{\mathbf{Q}}^{n,*}\|_h^2 + \|\be_{\mathbf{Q}}^{n}\|_h^2 + \|e_{u}^{n,*}\|_h^2 + \|e_{u}^{n}\|_h^2 \right).
        \end{aligned}
    \end{equation}
    Combining the bounds for $\be_{\mathbf{Q}}^{n,*}$, $e_u^{n,*}$, and $e_s^{n,*}$ with $\tau\le1$  yields the desired estimate
    \eqref{eq_error_s_star_bound}.
\end{proof}

\subsection{Convergence analysis of the fully discrete GSAV-ETD2 scheme}
Substitution of the exact solution into the numerical scheme \eqref{eq4_9}--\eqref{eq4_12}, followed by rearrangement, gives the following identities at time $t_{n+1}$:
\begin{subequations}\label{eqa_15}
\begin{align}
\mathbf{Q}(t_{n+1}) & = e^{-\tau \mLQ} \mathbf{Q}(t_n) + \int_0^\tau e^{-(\tau-\varsigma)\mLQ} \left( (1-\frac{\varsigma}{\tau})\mNQz(t_n) + \frac{\varsigma}{\tau}\mNQz(t_{n+1}) + \mathbf{T}_{\mathbf{Q}}^n(\varsigma) \right) d\varsigma,\label{eqa_15a}\\
u(t_{n+1}) & = e^{-\tau \mLu} u(t_n) + \int_0^\tau e^{-(\tau-\varsigma)\mLu} \left((1-\frac{\varsigma}{\tau}) \mNuz(t_n) + \frac{\varsigma}{\tau}\mNuz(t_{n+1}) + T_u^n(\varsigma) \right) d\varsigma,\label{eqa_15b}\\
\delta_\tau s(t_{n+1}) & =   -\left\langle \frac{1}{2}{(\mNQz(t_n)+\mNQz(t_{n+1}))}, \delta_\tau \mathbf{Q}(t_{n+1}) \right\rangle_h \no\\&- \left\langle  \frac{1}{2}{(\mNuz(t_n)+\mNuz(t_{n+1}))}, \delta_\tau u(t_{n+1}) \right\rangle_h  + T_s^n,\label{eqa_15c}
\end{align}
\end{subequations}
where the consistency errors $\mathbf{T}_{\mathbf Q}^n$, $T_u^n$, and $T_s^n$ are defined by
\begin{align}
\mathbf T_{\mathbf Q}^n(\varsigma)
&:=
\mathbf N(t_n+\varsigma)
-\left(1-\frac{\varsigma}{\tau}\right)\mathbf N(t_n)
-\frac{\varsigma}{\tau}\mathbf N(t_{n+1})
+\mathcal L_h\mathbf Q(t_n+\varsigma)-\mathcal L\mathbf Q(t_n+\varsigma),
\\
T_u^n(\varsigma)
&:=
\mathcal N(t_n+\varsigma)
-\left(1-\frac{\varsigma}{\tau}\right)\mathcal N(t_n)
-\frac{\varsigma}{\tau}\mathcal N(t_{n+1})
+\mathcal D_hu(t_n+\varsigma)-\mathcal Du(t_n+\varsigma),\\
T_s^n
&:=
-\frac1\tau\int_0^\tau
\left[
\bigl(\mathbf N(t_n+r),\mathbf Q_t(t_n+r)\bigr)
+
\bigl(\mathcal N(t_n+r),u_t(t_n+r)\bigr)
\right]\,dr\notag\\
&\quad
+\frac12
\left\langle
\mathbf N(t_n)+\mathbf N(t_{n+1}),
\delta_\tau\mathbf Q(t_{n+1})
\right\rangle_h
+\frac12
\left\langle
\mathcal N(t_n)+\mathcal N(t_{n+1}),
\delta_\tau u(t_{n+1})
\right\rangle_h .\label{eqrf}
\end{align}

Applying the nodal reformulation of Lemma~\ref{lem:etd2_reformulation}
to \eqref{eqa_15a}--\eqref{eqa_15b}, including their truncation terms,
and subtracting the resulting identities from
\eqref{eq:etd2_nodal_Q}--\eqref{eq:etd2_nodal_u} gives the error equations
\begin{align}
    \mq(\tau \mLQ)\ea + \mLQ \be_{\mathbf{Q}}^{n+1}
    &= \left(1-\frac{\varphi_2(\tau \mLQ)}{\varphi_1(\tau \mLQ)}\right)(\mNQ^n-\mNQz(t_n))\no\\&
     + \frac{\varphi_2(\tau \mLQ)}{\varphi_1(\tau \mLQ)} (\mNQ^{n,*}-\mNQz(t_{n+1})) - \widetilde{\mathbf{T}}_{\mathbf{Q}}^n,\label{eq_error_Q_final}\\
    \mq(\tau \mLu)\eb + \mLu e_{u}^{n+1}
    &= \left(1-\frac{\varphi_2(\tau \mLu)}{\varphi_1(\tau \mLu)}\right)(\mNu^n-\mNuz(t_n))\no\\&+ \frac{\varphi_2(\tau \mLu)}{\varphi_1(\tau \mLu)} (\mNu^{n,*}-\mNuz(t_{n+1})) - \widetilde{T}_u^n,\label{eq_error_u_final}
\end{align}
 where  $\widetilde{ \mathbf{T}}_{\mathbf{Q}}^n= \int_0^\tau \mathcal K_{\mLQ}(\varsigma)  \mathbf{T}_{\mathbf{Q}}^n(\varsigma) d\varsigma$ and $\widetilde{T}_u^n= \int_0^\tau \mathcal K_{\mLu}(\varsigma)  T_u^n(\varsigma) d\varsigma$.
Using
$\mathbf{T}_{\mathbf{Q}}^n(\varsigma)=\mathbf{T}_{\mathbf{Q}}^n(0)+\int_0^\varsigma\partial_r\mathbf{T}_{\mathbf{Q}}^n(r)\,dr$
and Fubini's theorem, we obtain
\ba
\widetilde{ \mathbf{T}}_{\mathbf{Q}}^n
={}&
\left(
\int_0^\tau
\mathcal K_{\mLQ}(\varsigma)\,d\varsigma
\right)\mathbf{T}_{\mathbf{Q}}^n(0)
+
\int_0^\tau
\mathcal K_{\mLQ}(\varsigma)
\left(
\int_0^\varsigma
\partial_r\mathbf{T}_{\mathbf{Q}}^n(r)\,dr
\right)d\varsigma
\nonumber\\
={}&
\mathbf{T}_{\mathbf{Q}}^n(0)
+
\int_0^\tau
\left(
\int_r^\tau
\mathcal K_{\mLQ}(\varsigma)\,d\varsigma
\right)
\partial_r\mathbf{T}_{\mathbf{Q}}^n(r)\,dr.
\ed
A similar identity holds for $\widetilde{T}_u^n$. To retain the
second-order temporal consistency after applying the transformation, we
estimate the derivatives of the residuals more precisely. From their
definitions, the fundamental theorem of calculus gives
\begin{align*}
\partial_\varsigma\mathbf T_{\mathbf Q}^n(\varsigma)
={}&
\partial_t\mathbf N(t_n+\varsigma)
-\frac{\mathbf N(t_{n+1})-\mathbf N(t_n)}{\tau}
+(\mathcal L_h-\mathcal L)\mathbf Q_t(t_n+\varsigma)
\\
={}&
\frac1\tau\int_0^\tau
\left[
\partial_t\mathbf N(t_n+\varsigma)
-\partial_t\mathbf N(t_n+\xi)
\right]d\xi
+(\mathcal L_h-\mathcal L)\mathbf Q_t(t_n+\varsigma),
\\
\partial_\varsigma T_u^n(\varsigma)
={}&
\partial_t\mathcal N(t_n+\varsigma)
-\frac{\mathcal N(t_{n+1})-\mathcal N(t_n)}{\tau}
+(\mathcal D_h-\mathcal D)u_t(t_n+\varsigma)
\\
={}&
\frac1\tau\int_0^\tau
\left[
\partial_t\mathcal N(t_n+\varsigma)
-\partial_t\mathcal N(t_n+\xi)
\right]d\xi
+(\mathcal D_h-\mathcal D)u_t(t_n+\varsigma).
\end{align*}
Therefore, the regularity assumptions \eqref{eqa_19},
Lemma~\ref{lem_parabolic_regularity_bootstrap}, and the second-order
consistency of the spatial difference operators imply that, for
$0\le\varsigma\le\tau$,
\begin{align*}
\|\partial_\varsigma\mathbf T_{\mathbf Q}^n(\varsigma)\|_h
&\le
\frac1\tau\int_0^\tau |\varsigma-\xi|\,d\xi\,
\|\partial_{tt}\mathbf N\|_{L^\infty(t_n,t_{n+1};\mathbf L^2)}
+Ch^2\|\mathbf Q_t\|_{L^\infty(t_n,t_{n+1};\mathbf H^4)}
\le C(\tau+h^2),
\\
\|\partial_\varsigma T_u^n(\varsigma)\|_h
&\le
\frac1\tau\int_0^\tau |\varsigma-\xi|\,d\xi\,
\|\partial_{tt}\mathcal N\|_{L^\infty(t_n,t_{n+1};L^2)}
+Ch^2\|u_t\|_{L^\infty(t_n,t_{n+1};H^6)}
\le C(\tau+h^2).
\end{align*}
Moreover, using the second-order consistency of the spatial difference operators and Lemma \ref{lem_parabolic_regularity_bootstrap}, we have
\ba
\|\mathbf T_{\mathbf Q}^n(0)\|_h+\|T_u^n(0)\|_h
&\le \left\|(\mathcal L_h-\mathcal L)\mathbf Q(t_n)\right\|_h+\left\|(\mathcal D_h-\mathcal D)u(t_n)\right\|_h
\\
&\le
Ch^2\left(
\|\mathbf Q\|_{L^\infty(t_n,t_{n+1};\mathbf H^4)}
+\|u\|_{L^\infty(t_n,t_{n+1};H^6)}
\right)
\le Ch^2.
\ed
Combining the above estimates with $\tau\le1$ gives
\ba\label{eq9a}
\|\widetilde{\mathbf T}_{\mathbf Q}^n\|_h
+\|\widetilde T_u^n\|_h
&\le
\|\mathbf T_{\mathbf Q}^n(0)\|_h+\|T_u^n(0)\|_h
+
\int_0^\tau
\left(
\|\partial_r\mathbf T_{\mathbf Q}^n(r)\|_h
+\|\partial_rT_u^n(r)\|_h
\right)dr
\\
&\le Ch^2+C\tau(\tau+h^2)
\le C(\tau^2+h^2).
\ed

Subtracting the \(s\)-equation \eqref{eqa_15c} from \eqref{eq4_12} gives the scalar error equation
\ba
   \frac{1}{\tau} ({e}_{s}^{n+1}-e_{s}^{n})
    &=- \left\langle \frac{1}{2}(\mNQ^n+\mNQ^{n,*}), \ea \right\rangle_h
    - \left\langle \frac{1}{2}( \mNu^n+ \mNu^{n,*}), \eb \right\rangle_h \notag\\
    &\quad - \frac{1}{2}\left\langle  \mNQ^n-\mNQz(t_n)+ \mNQ^{n,*}-\mNQz(t_{n+1}), \da \right\rangle_h \notag\\
    &\quad - \frac{1}{2}\left\langle  \mNu^n-\mNuz(t_n)+ \mNu^{n,*}-\mNuz(t_{n+1}), \db \right\rangle_h \notag\\&\quad
    -\frac{1}{4\tau}\langle  \mathcal{L}_h \dx \bQ^{n+1}, \dx \bQ^{n+1} \rangle_h-\frac{1}{4\tau}\langle \mathcal{D}_h \dx \bu^{n+1}, \dx \bu^{n+1} \rangle_h-
    T_s^n.\label{eq_error_s_final}
\ed
It remains to estimate the consistency error in the scalar equation \eqref{eqa_15c}.
Set
\[
\overline{\mathbf N}^{\,n}
:=\frac{\mathbf N(t_n)+\mathbf N(t_{n+1})}{2}
\]
and define the spatial quadrature residual by
\begin{align*}
R_{\mathbf Q,h}^n
:={}&
\left\langle
\overline{\mathbf N}^{\,n},
\delta_\tau\mathbf Q(t_{n+1})
\right\rangle_h
-
\left(
\overline{\mathbf N}^{\,n},
\delta_\tau\mathbf Q(t_{n+1})
\right),
\qquad
|R_{\mathbf Q,h}^n|\le Ch^2.
\end{align*}
Using
$\delta_\tau\mathbf Q(t_{n+1})
=\tau^{-1}\int_0^\tau\mathbf Q_t(t_n+r)\,dr$, we obtain
\begin{align*}
&\left|
\left\langle
\overline{\mathbf N}^{\,n},
\delta_\tau\mathbf Q(t_{n+1})
\right\rangle_h
-\frac1\tau\int_0^\tau
\bigl(\mathbf N(t_n+r),\mathbf Q_t(t_n+r)\bigr)\,dr
\right|
\\
&\quad\le |R_{\mathbf Q,h}^n|
+\left|
\left(
\frac1\tau\int_0^\tau
\bigl[\overline{\mathbf N}^{\,n}-\mathbf N(t_n+r)\bigr]\,dr,
\mathbf Q_t(t_n)
\right)
\right|
\\
&\qquad
+\frac1\tau\int_0^\tau
\left|
\left(
\overline{\mathbf N}^{\,n}-\mathbf N(t_n+r),
\mathbf Q_t(t_n+r)-\mathbf Q_t(t_n)
\right)
\right|\,dr
\\
&\quad\le Ch^2+C\tau^2
\left(
\|\partial_{tt}\mathbf N\|_{L^\infty(0,T;\mathbf L^2)}
\|\mathbf Q_t\|_{L^\infty(0,T;\mathbf L^2)}
+
\|\partial_t\mathbf N\|_{L^\infty(0,T;\mathbf L^2)}
\|\mathbf Q_{tt}\|_{L^\infty(0,T;\mathbf L^2)} \right).
\end{align*}
The same argument applied to $\mathcal N$ and $u$, together with
Lemma~\ref{lem_parabolic_regularity_bootstrap}, gives
\begin{align}
|T_s^n|
\le C\tau^2\Bigl(
&\|\partial_{tt}\mathbf N\|_{L^\infty(0,T;\mathbf L^2)}
 \|\mathbf Q_t\|_{L^\infty(0,T;\mathbf L^2)}
+\|\partial_t\mathbf N\|_{L^\infty(0,T;\mathbf L^2)}
 \|\mathbf Q_{tt}\|_{L^\infty(0,T;\mathbf L^2)}
\notag\\
&+\|\partial_{tt}\mathcal N\|_{L^\infty(0,T;L^2)}
 \|u_t\|_{L^\infty(0,T;L^2)}
+\|\partial_t\mathcal N\|_{L^\infty(0,T;L^2)}
 \|u_{tt}\|_{L^\infty(0,T;L^2)}
\Bigr)
+Ch^2
\le C(\tau^2+h^2).
\label{eq:Ts-bound}
\end{align}
Based on the estimates established in the preceding lemmas, we now provide
a detailed proof of Theorem~\ref{them4_1}.
\begin{proof}
Taking the discrete inner products of \eqref{eq_error_Q_final} and \eqref{eq_error_u_final} with $\ea$ and $\eb$, respectively, and adding the resulting identities gives
    \ba \label{eqa_17}
    &\left\langle  \psi(\tau \mLQ) \ea, \ea \right\rangle_h + \left\langle \mLQ \be_{\mathbf{Q}}^{n+1}, \ea \right\rangle_h
    + \left\langle  \psi(\tau \mLu) \eb, \eb \right\rangle_h + \left\langle \mLu e_{u}^{n+1}, \eb \right\rangle_h \\
    &= \left\langle   (1-\frac{\varphi_2(\tau \mLQ)}{\varphi_1(\tau \mLQ)})(\mNQ^n-\mNQz(t_n) ) + \frac{\varphi_2(\tau \mLQ)}{\varphi_1(\tau \mLQ)} (\mNQ^{n,*}-\mNQz(t_{n+1}) ) - \widetilde{ \mathbf{T}}_{\mathbf{Q}}^n, \ea \right\rangle_h \\
    &+ \left\langle (1-\frac{\varphi_2(\tau \mLu)}{\varphi_1(\tau \mLu)})(\mNu^n-\mNuz(t_n) ) + \frac{\varphi_2(\tau \mLu)}{\varphi_1(\tau \mLu)} (\mNu^{n,*}-\mNuz(t_{n+1}) )- \widetilde{T}_u^n, \eb \right\rangle_h.
    \ed
    Applying \eqref{eq:etd2_coercive_testing} to the successive error
    values and using \eqref{eq_inner_Q} yields
    \begin{equation} \label{eq_LHS_expanded_norm}
        \begin{aligned}
            \text{LHS} & = \iaQ{\ea} + \frac{1}{2\tau} \left[  \ilQ{\be_{\mathbf{Q}}^{n+1}} - \ilQ{\be_{\mathbf{Q}}^{n}}  + \left\langle \tau^2 \mLQ \ea, \ea \right\rangle_h \right]      \\
                       & \quad + \iau{\eb} + \frac{1}{2\tau} \left[ \ilu{e_{u}^{n+1}}  - \ilu{e_{u}^{n}}  +\left\langle \tau^2 \mLu \delta_\tau  e_{u}^{n+1}, \eb \right\rangle_h  \right] \\
                       & = \ibQ{\ea} + \frac{1}{2\tau} \left[  \ilQ{\be_{\mathbf{Q}}^{n+1}} - \ilQ{\be_{\mathbf{Q}}^{n}}   \right]
            + \ibu{\eb} + \frac{1}{2\tau} \left[ \ilu{e_{u}^{n+1}}  - \ilu{e_{u}^{n}}  \right].
        \end{aligned}
    \end{equation}
By Lemma~\ref{thm_full_lipschitz},
Lemmas~\ref{lem_phi_ineq} and~\ref{lem4_2}, and \eqref{eq9a}, the
right-hand side of \eqref{eqa_17} can be estimated as follows:
    \begin{equation} \label{eq_RHS_final_bound}
        \begin{aligned}
            \text{RHS} & \leq \frac{1}{4}\left( \ia{\ea} + \ia{\eb} \right) +  C \left( \ia{\mNu^n-\mNuz(t_n)} + \ia{\mNQ^n - \mNQz(t_n)} \right)  \\
                       & + C \left( \ia{\mNQ^{n,*}-\mNQz(t_{n+1})} + \ia{\mNu^{n,*}-\mNuz(t_{n+1})} \right)+ C (\tau^2 + h^2)^2
                           \\
                       & \leq \frac{1}{4} \left( \|\ea\|_h^2 + \|\eb\|_h^2 \right) + C (\tau^2 + h^2)^2                                                                                                                    \\
                       & \quad + C \left( |e_s^n|^2 + \|\be_{\mathbf{Q}}^{n}\|_{H_h^2}^2 + \|e_{u}^{n}\|_{H_h^2}^2 \right) + C \left( |e_s^{n,*}|^2 + \|\be_{\mathbf{Q}}^{n,*}\|_{H_h^2}^2 + \|e_{u}^{n,*}\|_{H_h^1}^2 + \|e_{u}^{n,*}\|_{H_h^2}^2 \right) \\
                       & \leq \frac{1}{4} \left( \|\ea\|_h^2 + \|\eb\|_h^2 \right)
            + C \left( |e_s^n|^2 + \|\be_{\mathbf{Q}}^{n}\|_{H_h^1}^2 + \|\be_{\mathbf{Q}}^{n}\|_{H_h^2}^2+ \|e_{u}^{n}\|_{H_h^2}^2 \right) + C (\tau^2 + h^2)^2.
        \end{aligned}
    \end{equation}
 Moreover, \eqref{eq_norm_bounds_1}
    gives $\|\ea\|_h^2+\|\eb\|_h^2\le \ibQ{\ea}+\ibu{\eb}$.

    Substituting \eqref{eq_LHS_expanded_norm} and
    \eqref{eq_RHS_final_bound} into \eqref{eqa_17} and using
    \eqref{eq_norm_bounds_1} gives the error energy inequality
    \begin{equation} \label{eq_error_energy_recursive}
        \begin{aligned}
             & \frac{3}{4} \left( \ibQ{\ea} + \ibu{\eb} \right) + \frac{1}{2\tau} \left[  \ilQ{\be_{\mathbf{Q}}^{n+1}} - \ilQ{\be_{\mathbf{Q}}^{n}}   + \ilu{e_{u}^{n+1}}  - \ilu{e_{u}^{n}}  \right] \\
             & \leq C \left( |e_s^n|^2 + \|\be_{\mathbf{Q}}^{n}\|_{H_h^2}^2 + \|e_{u}^{n}\|_{H_h^2}^2 \right) + C (\tau^2 + h^2)^2.
        \end{aligned}
    \end{equation}

    We next turn to the scalar error equation \eqref{eq_error_s_final} for $s$.
    Multiplying it by $2{e}_s^{n+1}$ gives
    \begin{equation} \label{eq_es_identity}
        \begin{aligned}
             & \frac{1}{\tau} \left( |{e}_s^{n+1}|^2 - |e_s^n|^2 + |{e}_s^{n+1} - e_s^n|^2 \right)
             = {\begin{aligned}[t]
             -{e}_s^{n+1}\Big[& \langle (\mNQ^n+ \mNQ^{n,*}), \ea \rangle_h\\
             & + \langle (  \mNu^n+\mNu^{n,*}), \eb \rangle_h \Big]
             \end{aligned}}\\
             & \qquad- {\begin{aligned}[t]
             {e}_s^{n+1}\Big[& \left\langle  \mNQ^n-\mNQz(t_n)+ \mNQ^{n,*}-\mNQz(t_{n+1}), \da \right\rangle_h\\
             & + \left\langle \mNu^n-\mNuz(t_n)+\mNu^{n,*}-\mNuz(t_{n+1}), \db \right\rangle_h + 2T_s^n \Big]
             \end{aligned}}\\
             &\qquad{-\frac{{e}_s^{n+1}}{2\tau}\Big[ \left\langle  \mathcal{L}_h \dx \bQ^{n+1}, \dx \bQ^{n+1} \right\rangle_h +\left\langle  \mathcal{D}_h \dx \bu^{n+1}, \dx \bu^{n+1} \right\rangle_h\Big]}=: \mathcal{P}_1 + \mathcal{P}_2 + \mathcal{P}_3.
        \end{aligned}
    \end{equation}
For $\mathcal{P}_1$, applying the Cauchy--Schwarz inequality and
Young's inequality and invoking
Lemma~\ref{gn_bound0} and Theorem~\ref{th4_1}, we obtain
    \begin{equation} \label{eq_I2_bound}
        \begin{aligned}
            \mathcal{P}_1 & \le 2  \mathcal{M}_2 |{e}_s^{n+1}| \left( \|\ea\|_h + \|\eb\|_h \right)                           \\
                          & \le 16   \mathcal{M}_2^2 |{e}_s^{n+1}|^2 + \frac{1}{8} \|\ea\|_h^2 + \frac{1}{8} \|\eb\|_h^2.
        \end{aligned}
    \end{equation}
  For $\mathcal{P}_2$, applying the Cauchy--Schwarz inequality and
  Young's inequality, invoking Lemma~\ref{gn_bound0}, and using
  Lemma~\ref{thm_full_lipschitz} and \eqref{eq:Ts-bound} gives
    \begin{equation} \label{eq_I1_bound}
        \begin{aligned}
            \mathcal{P}_2 & \le  2|{e}_s^{n+1}|\cdot   \left(C \left( |e_s^n| + \|\be_{\mathbf{Q}}^{n}\|_{H_h^1}+ \|\be_{\mathbf{Q}}^{n}\|_{H_h^2} + \|e_{u}^{n}\|_{H_h^2} + h^2 \right)\right.                                      \\
                          & \qquad \left.
            + C \left( |e_s^{n,*}| + \|\be_{\mathbf{Q}}^{n,*}\|_{H_h^1} + \|\be_{\mathbf{Q}}^{n,*}\|_{H_h^2} + \|e_{u}^{n,*}\|_{H_h^2} + h^2 \right) + C(\tau^2 + h^2) \right)                                                                        \\
                          & \le |{e}_s^{n+1}|^2 + C \left(   |e_s^n|^2 + \|\be_{\mathbf{Q}}^{n}\|_{H_h^1}^2 + \|\be_{\mathbf{Q}}^{n}\|_{H_h^2}^2+ \|e_{u}^{n}\|_{H_h^2}^2 + (\tau^2 + h^2)^2  \right).
        \end{aligned}
    \end{equation}

For $\mathcal{P}_3$, applying the Cauchy--Schwarz inequality and
Young's inequality gives
    \begin{equation} \label{eq_I3_bound}
        \begin{aligned}
            \mathcal{P}_3 & \le 2 |{e}_s^{n+1}| \cdot \frac{1}{2\tau} \left(  \left\langle  \mathcal{L}_h \dx \bQ^{n+1}, \dx \bQ^{n+1} \right\rangle_h +\left\langle  \mathcal{D}_h \dx \bu^{n+1}, \dx \bu^{n+1} \right\rangle_h  \right)\\
                          &\le |{e}_s^{n+1}|^2 + \frac{1}{2}\left(  \left\langle  \mathcal{L}_h \dx \bQ^{n+1}, \dxt\bQ^{n+1} \right\rangle_h^2+\left\langle  \mathcal{D}_h \dx \bu^{n+1}, \dxt\bu^{n+1} \right\rangle_h^2  \right).
        \end{aligned}
    \end{equation}
Subtracting  \eqref{eq4_6} from  \eqref{eq4_9}
gives
\begin{equation}
\bQ^{n+1}-\bQ^{n,*}
=
\tau\varphi_2(\tau\mLQ)
\bigl(\mNQ^{n,*}-\mNQ^n\bigr).
\label{eq_Q_stage_difference}
\end{equation}
Substituting \eqref{eq_Q_stage_difference} into the second term on the
right-hand side of \eqref{eq_I3_bound}, and then applying
Lemmas~\ref{th4_1}, \ref{thm_full_lipschitz} and~\ref{lem4_2}, we obtain
\ba\label{eqcn}
 &\left\langle  \mathcal{L}_h \dx \bQ^{n+1},\dxt\bQ^{n+1} \right\rangle_h
\\&=
  \left\langle  \mathcal{L}_h \tau\varphi_2(\tau\mLQ)
\bigl(\mNQ^{n,*}-\mNQ^n\bigr),\varphi_2(\tau\mLQ)
\bigl(\mNQ^{n,*}-\mNQ^n\bigr) \right\rangle_h\\
&\le \left\|
 \mNQ^{n,*}-\mNQ^n
\right\|_h^2\\
&\le
  \left\|
 (
\mNQ^{n,*}-\mNQz(t_{n+1})
+
\bigl(
\mNQz(t_{n+1})-\mNQz(t_n)
\bigr)
+
\bigl(
\mNQz(t_n)-\mNQ^n))
\right\|_h^2
\\
&\le
 C \left(\left\|
\mNQ^{n,*}-\mNQz(t_{n+1})\right\|_h^2
+  \left(
\int_{t_n}^{t_{n+1}}
\left\|
\partial_t\mNQz(t)
\right\|_h\,dt
\right)^2
+
 \left\|
\mNQz(t_n)-\mNQ^n
\right\|_h^2\right)
\\&\le
 C \left(\left\|\mNQ^{n,*}-\mNQz(t_{n+1})\right\|_h
+  \tau^2
+
 \left\|\mNQz(t_n)-\mNQ^n\right\|_h\right)
 \\&\le
 C\Bigl(
\|\be_{\mathbf Q}^{n}\|_{H_h^1}
+
\|e_u^{n}\|_{H_h^2}
+
|e_s^{n}|+\tau^2+h^2
\Bigr)
 .
\ed
Applying the analogous argument to the $u$-component of
\eqref{eq_I3_bound} and combining it with \eqref{eqcn} gives
\begin{align}
&\left\langle\mathcal L_h\dx\bQ^{n+1},
\dxt\bQ^{n+1}
\right\rangle_h^2+\left\langle\mathcal D_h\dx\bu^{n+1},
\dxt\bu^{n+1}
\right\rangle_h^2
\nonumber\\
  &\le
C\Bigl(
\|\be_{\mathbf{Q}}^n\|_{H_h^1}^2
+
\|\be_{\mathbf{Q}}^n\|_{H_h^2}^2
+
\|e_u^n\|_{H_h^2}^2
+
|e_s^n|^2+(\tau^2 + h^2)^2
\Bigr).
\label{eq_u_stage_energy_bound}
\end{align}
Substituting \eqref{eq_u_stage_energy_bound} into \eqref{eq_I3_bound} and applying Lemma~\ref{thm_full_lipschitz} yields
\begin{equation} \label{eq_I3_final_bound}
    \begin{aligned}
        \mathcal{P}_3 & \le |{e}_s^{n+1}|^2 + C \left( |e_s^n|^2 + \|\be_{\mathbf{Q}}^{n}\|_{H_h^1}^2 + \|\be_{\mathbf{Q}}^{n}\|_{H_h^2}^2 + \|e_{u}^{n}\|_{H_h^2}^2 + (\tau^2 + h^2)^2 \right).
    \end{aligned}
\end{equation}

    Substituting \eqref{eq_I2_bound}, \eqref{eq_I1_bound}, and \eqref{eq_I3_final_bound} into \eqref{eq_es_identity} yields
    \begin{equation} \label{eq_es_energy_recursive}
        \begin{aligned}
            \frac{1}{\tau} \left( |{e}_s^{n+1}|^2 - |e_s^n|^2 \right)
            &\le \left( 16    \mathcal{M}_2^2 + 1 \right) |{e}_s^{n+1}|^2 + \frac{1}{8} \|\ea\|_h^2 + \frac{1}{8} \|\eb\|_h^2\\
            &\quad + C \left( |e_s^n|^2 + \|\be_{\mathbf{Q}}^{n}\|_{H_h^1}^2 + \|\be_{\mathbf{Q}}^{n}\|_{H_h^2}^2 \right.\left.  + \|e_{u}^{n}\|_{H_h^2}^2 + (\tau^2 + h^2)^2 \right).
        \end{aligned}
    \end{equation}
    Combining \eqref{eq_es_energy_recursive} with \eqref{eq_error_energy_recursive} gives
    \begin{equation} \label{eq_unified_single}
        \begin{aligned}
             & \frac{1}{\tau} \left( \ilQ{\be_{\mathbf{Q}}^{n+1}} - \ilQ{\be_{\mathbf{Q}}^{n}}   + \ilu{e_{u}^{n+1}}  - \ilu{e_{u}^{n}}  + |e_s^{n+1}|^2 - |e_s^n|^2 \right) \\
             & \le  \left( 16    \mathcal{M}_2^2 + 1 \right) |{e}_s^{n+1}|^2+C \left( |e_s^n|^2 + \|\be_{\mathbf{Q}}^{n}\|_{H_h^1}^2 + \|\be_{\mathbf{Q}}^{n}\|_{H_h^2}^2 + \|e_{u}^{n}\|_{H_h^2}^2 + (\tau^2 + h^2)^2  \right).
        \end{aligned}
    \end{equation}
    Provided that
$\tau(16\mathcal M_2^2+1)\le \frac12$, we multiply
\eqref{eq_unified_single} by $\tau$ and use
$(1-x)^{-1}\le1+2x$ for $0\le x\le\frac12$. We then obtain a constant $C_E>0$ such that
    \begin{equation} \label{eq_unified_single1}
        \begin{aligned}
             & \frac{1}{\tau} \left( \ilQ{\be_{\mathbf{Q}}^{n+1}} - \ilQ{\be_{\mathbf{Q}}^{n}}   + \ilu{e_{u}^{n+1}}  - \ilu{e_{u}^{n}}  + |e_s^{n+1}|^2 - |e_s^n|^2 \right) \\
             & \le C_E \left( |e_s^n|^2 + \|\be_{\mathbf{Q}}^{n}\|_{H_h^1}^2 + \|\be_{\mathbf{Q}}^{n}\|_{H_h^2}^2 + \|e_{u}^{n}\|_{H_h^2}^2 + (\tau^2 + h^2)^2  \right).
        \end{aligned}
    \end{equation}

    We next turn to the $H_h^2$-norm of the error $\be_{\mathbf{Q}}^{n}$.
    Taking the discrete inner product of the time-shifted equation
    \eqref{eq_error_Q_final} with $\mLQ \be_{\mathbf{Q}}^{n}$ gives
    \begin{equation} \label{eq_H2_bound_shifted_detailed}
        \begin{aligned}
              \left\langle  \mQ \delta_\tau \be_{\mathbf{Q}}^{n}, \mLQ \be_{\mathbf{Q}}^{n} \right\rangle_h + \|\mLQ \be_{\mathbf{Q}}^{n}\|_h^2
             &= \left\langle
    \left(1-\frac{\varphi_2(\tau \mLQ)}{\varphi_1(\tau \mLQ)}\right)
    \bigl(\mNQ^{n-1}-\mNQz(t_{n-1})\bigr)\right.\\
             &\left.
             +\frac{\varphi_2(\tau \mLQ)}{\varphi_1(\tau \mLQ)}
    \bigl(\mNQ^{n-1,*}-\mNQz(t_n)\bigr)
    -\widetilde{\mathbf T}_{\mathbf Q}^{n-1},
             \mLQ \be_{\mathbf{Q}}^{n} \right\rangle_h
             =:\mathcal{P}_4.
        \end{aligned}
    \end{equation}
    The self-adjointness of $\mQ$ and $\mLQ$ and identity \eqref{eq_inner_Q} give the following expression for the first term on the left-hand side of \eqref{eq_H2_bound_shifted_detailed}:
    \begin{equation}\label{eq_source_prev_step}
        \begin{aligned}
            \left\langle  \mQ\delta_\tau \be_{\mathbf{Q}}^{n}, \mLQ\be_{\mathbf{Q}}^{n} \right\rangle_h & = \left\langle  \mQ\mLQ\delta_\tau \be_{\mathbf{Q}}^{n}, \be_{\mathbf{Q}}^{n} \right\rangle_h
            \\&=\frac{1}{2\tau} \left( \icQ{\be_{\mathbf{Q}}^{n}} - \icQ{\be_{\mathbf{Q}}^{n-1}}+\icQ{\delta\be_{\mathbf{Q}}^{n}} \right).
        \end{aligned}
    \end{equation}
 Applying Young's inequality, Lemma~\ref{thm_full_lipschitz}, and
 estimate~\eqref{eq9a} to $\mathcal{P}_4$ gives
    \begin{align}
        \mathcal{P}_4
        &\le \frac{1}{2}(\|\mNQ^{n-1}-\mNQz(t_{n-1})\|_h^2+\|\mNQ^{n-1,*}-\mNQz(t_n)\|_h^2+\|\widetilde{\mathbf T}_{\mathbf Q}^{n-1}\|_h^2)
        +\frac{1}{2}\|\mLQ \be_{\mathbf{Q}}^{n}\|_h^2\no\\
        &\le C \left( |e_s^{n-1}|^2
        + \|\be_{\mathbf{Q}}^{n-1}\|_{H_h^1}^2
        + \|e_{u}^{n-1}\|_{H_h^2}^2
        +|e_s^{n-1,*}|^2
        + \|\be_{\mathbf{Q}}^{n-1,*}\|_{H_h^1}^2
        + \|e_{u}^{n-1,*}\|_{H_h^2}^2\right)\no\\
        &\quad + C (\tau^2 + h^2)^2
        + \frac{1}{2} \|\mLQ \be_{\mathbf{Q}}^{n}\|_h^2.
        \label{eq_source_prev_step2}
    \end{align}
  Substituting \eqref{eq_source_prev_step} and
\eqref{eq_source_prev_step2} into
\eqref{eq_H2_bound_shifted_detailed}, and using
Lemma~\ref{lem4_2} together with the norm-equivalence estimate
\[
\|{\be_{\mathbf Q}^{n}}\|_{H_h^2}
\le
C_L\|\mathcal L_h \be_{\mathbf{Q}}^{n}\|_h,
\]
we obtain
    \begin{align}
       & \frac{C_L^2}{\tau} \left( \icQ{\be_{\mathbf{Q}}^{n}} - \icQ{\be_{\mathbf{Q}}^{n-1}} \right)+\|\be_{\mathbf{Q}}^{n}\|_{H_h^2}^2
       \le C \left( |e_s^{n-1}|^2 + \|\be_{\mathbf{Q}}^{n-1}\|_{H_h^1}^2 + \|e_{u}^{n-1}\|_{H_h^2}^2 + (\tau^2 + h^2)^2 \right).
       \label{eq_H2_full_bound_equiv}
    \end{align}
        Define $\mathcal Y^n
    :=\|\be_{\mathbf Q}^n\|_{\mathcal L_h}^2
    +\|e_u^n\|_{\mathcal D_h}^2+|e_s^n|^2$.
    Multiplying
    \eqref{eq_H2_full_bound_equiv} by $C_E $ and adding the result to
    \eqref{eq_unified_single1} gives, for $n\ge1$,
    \begin{align}
    \frac{1}{\tau}\left[
    \mathcal Y^{n+1}-\mathcal Y^n
    +C_E  C_L^2\left(
    \icQ{\be_{\mathbf Q}^n}
    -\icQ{\be_{\mathbf Q}^{n-1}}
    \right)\right]
  \le
    C\left(
    \mathcal Y^n+\mathcal Y^{n-1}
    +(\tau^2+h^2)^2\right).
    \label{eq_augmented_energy_recursive}
    \end{align}

By the initial error definition \eqref{eqcv1}, together with
\eqref{eq_initial_error_bound}, and taking $n=0$ in \eqref{eq_unified_single}, we have
\begin{equation}\label{eqcf1}
    \mathcal Y^0
    =|e_s^0|^2
    \le Ch^4,
    \qquad
    \|\be_{\mathbf Q}^0\|_{\psi_{\mathcal L}^2}^2=0,\quad \mathcal Y^1
    \le C\left(\tau(\tau^2+h^2)^2+h^4\right).
\end{equation}
    Summing \eqref{eq_augmented_energy_recursive} from
    $n=1$ to $m$ therefore gives
    \ba\label{eqa_21}
    \mathcal Y^{m+1}+C_E  C_L^2
    \icQ{\be_{\mathbf Q}^m}
    &\le \mathcal Y^1
    +CT(\tau^2+h^2)^2
    +C\tau\mathcal Y^0
    +2C\tau\sum_{k=1}^{m}\mathcal Y^k.
    \ed
    Applying the discrete Gr\"onwall inequality to \eqref{eqa_21} yields
            \ba \label{eqa_24}
    \mathcal{Y}^{m+1}
    &\le \left( \mathcal{Y}^1+C\tau\mathcal Y^0 + C T (\tau^2 + h^2)^2 \right) \exp\left( \sum_{k=0}^{m} 2C\tau \right) \\
    &\le \left( C(\tau+T) (\tau^2 + h^2)^2+C(1+\tau)h^4 \right) e^{2C(T+\tau)}\\
    &\le C(T,C_L,\mathcal{M}_2) (\tau^4 + h^4).
    \ed

By virtue of the norm equivalence, we can obtain the following error estimate:
\[
\|\be_{\mathbf Q}^{n}\|_{H_h^1}
+\|e_u^n\|_{H_h^2}
+|e_s^n|
\le C_{err}(T,\mathcal{M}_2,C_L)(\tau^2+h^2),
\quad 0\le n\le m+1.
\]
Taking $h$ and $\tau$ sufficiently small, we have
\[
s_h^n=s(t_n)+e_s^n
\ge -M_s-C_{err}(\tau^2+h^2)\ge -M_s-\frac12>-M_s-1=-C_s,
\quad 0\le n\le m+1.
\]
Therefore, the bootstrap assumption is improved and can be continued to the
next time level. By induction, it holds for all
$0\le n\le N_T$.
\end{proof}

\section{Numerical experiments}\label{section6}
We now assess the accuracy, energy behavior, and pattern-resolving capability
of the GSAV--ETD2 scheme.  Section~\ref{subsec:convergence-tests} verifies the
predicted convergence rates.  Section~\ref{subsec:periodic-self-assembly}
examines wavelength selection and dislocation-mediated layer fusion in
periodic cells,
Section~\ref{subsec:confined-2d} considers several confined geometries, and
Section~\ref{subsec:three-dimensional-test} resolves a three-dimensional
screw-dislocation annihilation event.  Unless explicitly stated otherwise,
the reported energy curves are obtained directly from the same numerical
trajectory as the displayed fields; no post-processing relaxation is
applied.

\subsection{Convergence tests}\label{subsec:convergence-tests}
Following \cite{shi2025modified}, we consider
$\vT<\vT_2^*<\vT_1^*$, corresponding to a deep quench into the stable SmA
regime.  We set $a_1=a_2=1$, $\vT=-1$, $\vT_1^*=1$, and $\vT_2^*=0$, so that
$\A=-2$ and $a=-1$.  The remaining parameters are
\begin{equation*}
L_d=1,\quad K=0.1,\quad \B=0,\quad \C=4,\quad b=0,\quad c=3,
\quad q=4,\quad B_0=10^{-3},\quad s_+=\sqrt{-2\A/\C}.
\end{equation*}
We take $\zeta(x)=e^x$ and $\kappa=3$.

We evaluate the convergence rates of the proposed scheme by computing the discrete errors in $\mathbf{Q}$, $u$, and $s$ in several norms at the final time $T=1$.
For the temporal test, we fix $h=1/128$ and use
$\tau=2^{-k}\tau_{1}$, where $k=0,1,\ldots,7$ and
$\tau_{1}=2^{-4}$.  The numerical solution computed with
$\tau=2^{-10}\tau_{1}$ serves as the benchmark solution.  For the spatial
test, the temporal step is chosen sufficiently small that its contribution is
negligible relative to the spatial error, and the benchmark is computed on a
further refined grid.  We prescribe the initial conditions
\ba\no
\mathbf{Q}^0(x,y) = S \left( \mathbf{n}\mathbf{n}^T - \frac{1}{2}\mathbf{I}_2 \right),\quad \mathbf{n} = (\cos 2\pi q(x+y), \sin 2\pi q(x+y))^T,
\ed
where $S$ is set to $0.25$ and $u$ is initialized by $u^0(x,y) = \cos(2\pi q x)$.
\begin{table}[htbp]
    \centering
    \setlength{\tabcolsep}{3pt}
    \begin{tabular}{l cc cc c}
        \toprule
        \multirow{2}{*}{$\tau$} & \multicolumn{2}{c}{$\mathbf{Q}$} & \multicolumn{2}{c}{$u$} & \multicolumn{1}{c}{$s$} \\
        \cmidrule(lr){2-3} \cmidrule(lr){4-5} \cmidrule(lr){6-6}
                                & $\ell^2$-norm & $H_h^1$-norm & $\ell^2$-norm & $H_h^2$-norm & Absolute error \\
        \midrule
        $2^{-4}$  & 7.13e-4 (--) & 8.73e-3 (--) & 3.06e-2 (--) & 7.82e-1 (--) & 1.37e-6 (--) \\
        $2^{-5}$  & 2.13e-4 (1.74) & 3.17e-3 (1.46) & 1.59e-2 (0.94) & 4.07e-1 (0.94) & 8.49e-7 (0.70) \\
        $2^{-6}$ & 5.96e-5 (1.84) & 1.09e-3 (1.54) & 6.65e-3 (1.26) & 1.74e-1 (1.22) & 3.34e-7 (1.34) \\
        $2^{-7}$ & 1.59e-5 (1.90) & 3.40e-4 (1.68) & 2.29e-3 (1.54) & 6.18e-2 (1.50) & 1.07e-7 (1.64) \\
        $2^{-8}$ & 4.14e-6 (1.95) & 9.52e-5 (1.83) & 6.73e-4 (1.76) & 1.85e-2 (1.74) & 3.08e-8 (1.80) \\
        $2^{-9}$ & 1.05e-6 (1.97) & 2.49e-5 (1.94) & 1.80e-4 (1.90) & 4.98e-3 (1.89) & 8.30e-9 (1.89) \\
        $2^{-10}$ & 2.66e-7 (1.99) & 6.32e-6 (1.98) & 4.64e-5 (1.96) & 1.28e-3 (1.95) & 2.16e-9 (1.94) \\
        $2^{-11}$ & 6.69e-8 (1.99) & 1.59e-6 (1.99) & 1.18e-5 (1.98) & 3.26e-4 (1.98) & 5.50e-10 (1.97) \\
        \bottomrule
    \end{tabular}
    \caption{Errors and convergence rates for $\mathbf{Q}$, $u$, and $s$ in the indicated norms.}
    \label{tab_error_all_variables_combined1}
\end{table}

\begin{table}[htbp]
    \centering
    \setlength{\tabcolsep}{5pt}
    \begin{tabular}{lccccc}
        \toprule
        \multirow{2}{*}{$h$}
        & \multicolumn{2}{c}{$\mathbf{Q}$}
        & \multicolumn{2}{c}{$u$}
        & $s$ \\
        \cmidrule(lr){2-3}
        \cmidrule(lr){4-5}
        \cmidrule(lr){6-6}
        & $\ell^2$-norm
        & $H_h^1$-norm
        & $\ell^2$-norm
        & $H_h^2$-norm
        & absolute error \\
        \midrule

        $2^{-6}$
        & $3.93\mathrm{e}{-2}$ (--)
        & $4.02\mathrm{e}{-1}$ (--)
        & $1.42\mathrm{e}{-2}$ (--)
        & $1.69\mathrm{e}{0}$ (--)
        & $1.63\mathrm{e}{-3}$ (--) \\

        $2^{-7}$
        & $1.05\mathrm{e}{-2}$ (1.90)
        & $1.17\mathrm{e}{-1}$ (1.79)
        & $4.01\mathrm{e}{-3}$ (1.83)
        & $6.17\mathrm{e}{-1}$ (1.45)
        & $6.62\mathrm{e}{-4}$ (1.30) \\

        $2^{-8}$
        & $2.68\mathrm{e}{-3}$ (1.98)
        & $3.02\mathrm{e}{-2}$ (1.95)
        & $1.03\mathrm{e}{-3}$ (1.97)
        & $1.68\mathrm{e}{-1}$ (1.88)
        & $2.31\mathrm{e}{-4}$ (1.52) \\

        $2^{-9}$
        & $6.72\mathrm{e}{-4}$ (1.99)
        & $7.63\mathrm{e}{-3}$ (1.99)
        & $2.58\mathrm{e}{-4}$ (1.99)
        & $4.28\mathrm{e}{-2}$ (1.97)
        & $6.70\mathrm{e}{-5}$ (1.78) \\

        $2^{-10}$
        & $1.68\mathrm{e}{-4}$ (2.00)
        & $1.91\mathrm{e}{-3}$ (2.00)
        & $6.45\mathrm{e}{-5}$ (2.00)
        & $1.08\mathrm{e}{-2}$ (1.99)
        & $1.74\mathrm{e}{-5}$ (1.95) \\

        \bottomrule
    \end{tabular}
    \caption{Spatial errors and convergence rates for $\mathbf{Q}$, $u$,
    and $s$ in various norms.}
    \label{tab_spatial_error_all_variables}
\end{table}

Tables~\ref{tab_error_all_variables_combined1} and
\ref{tab_spatial_error_all_variables} show that the rates approach two for
$\mathbf Q$ and $u$ in all the indicated norms.  The auxiliary variable also
exhibits asymptotically second-order behavior.  These results agree with the
fully discrete estimate of Theorem~\ref{them4_1}.

\subsection{Energy dissipation, wavelength selection, and
dislocation-mediated layer fusion}
\label{subsec:periodic-self-assembly}

We first test whether the scheme reproduces the preferred layer spacing of
the mLdG energy.  On the periodic unit square, we prescribe
\begin{equation}
\mathbf Q_h^0(x_i,y_j)
=
\begin{pmatrix}
\xi_{1,ij}&\xi_{2,ij}\\
\xi_{2,ij}&-\xi_{1,ij}
\end{pmatrix},
\qquad
u^0(x_i,y_j)=\cos(2\pi qx_i),
\label{eq:numerical-random-layered-initial-data}
\end{equation}
where $\xi_{1,ij}$ and $\xi_{2,ij}$ are independent samples from the uniform
distribution on $[-0.5,0.5]$.  We use the coefficients of
Section~\ref{subsec:convergence-tests}, $h=1/128$, $\tau=2^{-5}$, and
$T=50$, and repeat the computation for $q=4,6$, and $10$.

\begin{figure}[htbp]
  \centering
  \includegraphics[width=0.96\textwidth]{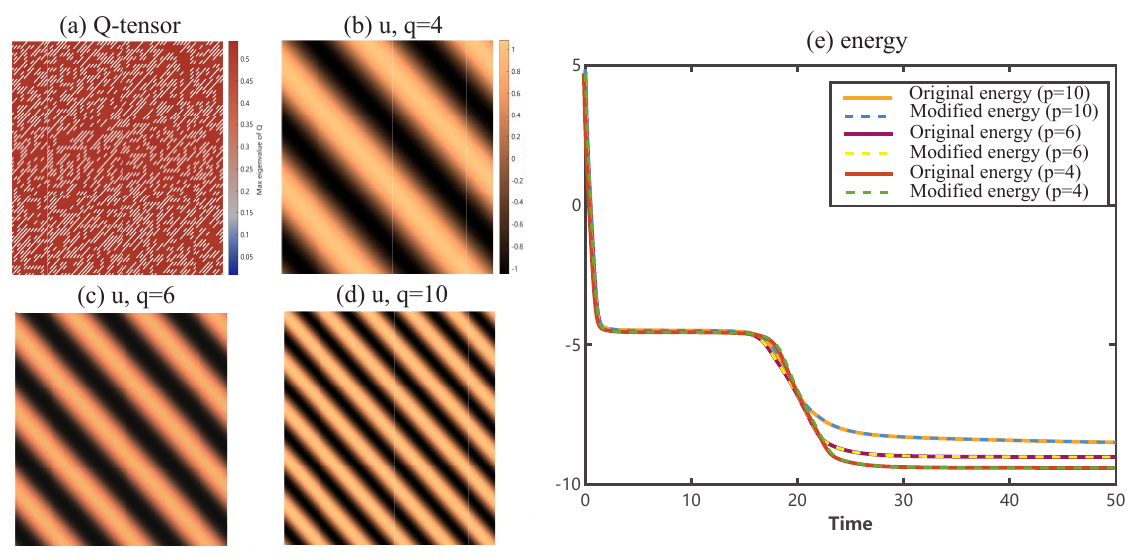}
  \caption{Wavelength selection from random orientational order and a
  layered initial density.  Panel (a) shows the randomly oriented initial
  $\mathbf Q$-tensor field, panels (b)--(d) show the final density field for
  $q=4,6$, and $10$, respectively, and panel (e) compares the corresponding
  original and modified energies on $0\le t\le50$.  The parameter denoted by
  $p$ in the embedded legend is the wave number $q$.}
  \label{fig:periodic-wavenumber-selection}
\end{figure}

As $q$ increases, the number of density interfaces increases and their
spacing decreases, in agreement with the preferred wavelength $2\pi/q$.
Despite the common disordered orientational initialization, each run
produces a regular lamellar density field with the spacing selected by its
value of $q$.  For all three values of $q$, the original and modified
energies decay throughout the computation and remain visually close.  The
rapid initial decrease followed by a slower relaxation separates the
formation of local layers from their subsequent orientational reorganization.

Layer fusion has a direct physical interpretation in smectics rather than
being merely a numerical change of morphology.  Because the layered phase
breaks translational symmetry, the termination or merger of a layer is the
characteristic signature of an edge dislocation.  In particular, Xia and
Han~\cite{xia2024simple} recovered a single SmA edge-dislocation profile in
which two layers merge into one by prescribing nine layers in one half of
the cell and ten in the other.  Their construction is
supported by nonperiodic boundary constraints and cannot be copied directly
to a doubly periodic domain because it does not match across the periodic
boundary.  Moreover, the present smooth periodic construction introduces no
boundary-supported net layer-count mismatch.  We therefore do not attempt to
reproduce their isolated equilibrium dislocation.  Instead, we use a smooth,
exactly periodic initial state to examine a periodically compatible,
net-neutral realization of the corresponding dynamical mechanism of layer
reconnection and removal.

More precisely, we use the periodic cell $\Omega=(0,2\pi)^2$ with $N=64$,
$\tau=10^{-2}$, and $T=100$.  The model parameters are
\begin{align*}
&K=0.1,\quad \A=-1,\quad \B=0,\quad \C=2,\quad
a=-5,\quad b=0,\quad c=5,\\
&q=5,\quad B_0=10^{-3},\quad
\kappa_{\mathbf Q}=\kappa_u=1.
\end{align*}
Writing $\mathbf n^0=(\cos(q(x+y)),\sin(q(x+y)))^T$, we prescribe the
orientationally incompatible layered state
\begin{equation*}
\mathbf Q^0=0.8\left(\mathbf n^0(\mathbf n^0)^T-\tfrac12\mathbf I_2\right),
\qquad
u^0=0.8\cos(qx)\sin(qy)
=0.4\bigl[\sin(q(x+y))+\sin(q(y-x))\bigr].
\end{equation*}
Thus $u^0$ is a superposition of two oblique layer families with the same
preferred wave number, whereas $\mathbf Q^0$ does not select a common global
layer normal.  This periodic mode competition provides a boundary-free
driving mechanism for a localized reconnection event without prescribing an
isolated defect.

\begin{figure}[htbp]
  \centering
  \includegraphics[width=0.7\textwidth]{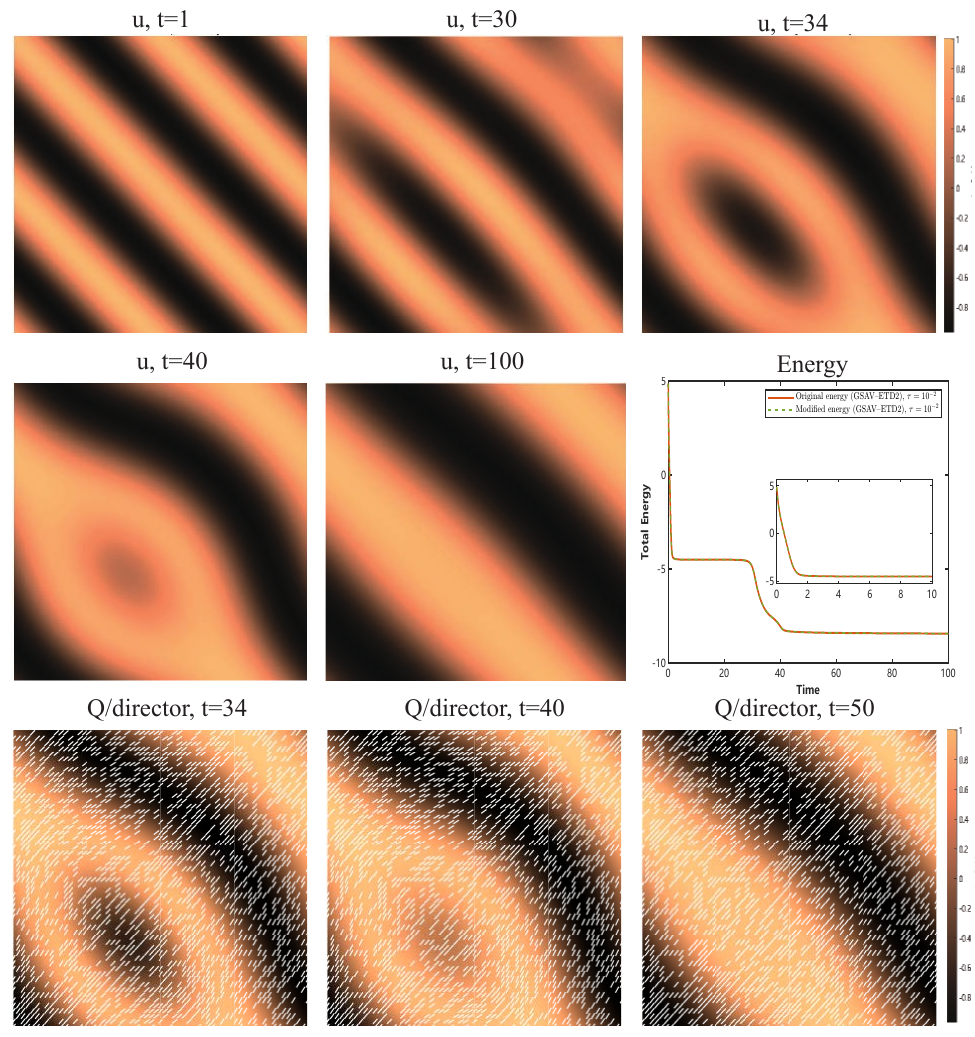}
  \caption{Periodic layer fusion and orientational reorganization through a
  dislocation-mediated reconnection event.  The density field $u$ is shown
  at $t=1,30,34,40$, and $100$; the energy histories are shown in the
  middle-right panel; and the bottom row shows the
  $\mathbf Q$-tensor/director field at $t=34,40$, and $50$.  In contrast to
  the boundary-supported single dislocation in~\cite{xia2024simple}, the
  present doubly periodic configuration has no imposed net layer-count
  mismatch.  The computation uses $\tau=10^{-2}$.}
  \label{fig:periodic-layer-fusion}
\end{figure}

Figure~\ref{fig:periodic-layer-fusion} resolves the loss of stability of the
initial layer arrangement.  Between $t=30$ and $t=40$, strongly curved
interfaces approach and reconnect, leaving a short closed layer that is
subsequently eliminated.  This two-to-one layer-reduction mechanism is the
dynamical analogue of the edge-dislocation structure reported by Xia and
Han~\cite{xia2024simple}, where two smectic layers merge into one.  Here,
however, the event occurs in a fully periodic cell and should be interpreted
as a periodically compatible, dislocation-mediated reconnection rather than
as a boundary-supported isolated dislocation.  The director field reorganizes
most visibly in the strongly curved region and then becomes coherent with
the emerging lamellae.  The pronounced decrease of both the original and
modified energies over the same time interval identifies the morphological
transition as a dissipative relaxation of the frustrated layered state,
rather than a numerical oscillation.  The experiment therefore demonstrates
that the proposed scheme resolves a physically important defect-mediated
pathway for changing the number and connectivity of smectic layers, together
with the coupled reorganization of orientational and positional order.

\subsection{Two-dimensional confined regions}\label{subsec:confined-2d}

We next embed a fixed confined region in the periodic box
$[0,L]^2$, $L=2\pi$.  If $d(\mathbf x)$ is the signed distance to the
boundary (negative inside), the diffuse indicator is
$M(\mathbf x)=\frac12[1-\tanh(d(\mathbf x)/\epsilon)]$, with
$\epsilon=2h$.  The exterior wall is imposed variationally through quadratic
penalties of strength $\alpha_{\mathbf Q}=\alpha_u=100$, and a narrow
anchoring layer is used on the boundary.  This fixed-mask construction avoids
changing the computational domain during the evolution and includes the
matching wall contribution in the discrete energy.

Unless stated otherwise, the computations in this subsection use
\begin{equation*}
N=128,\quad h=L/N,\quad \tau=5\times10^{-3},\quad T=15,
\quad K=0.1,\quad \A=-1,\quad \B=0,\quad \C=5,
\end{equation*}
\begin{equation*}
a=-5,\quad b=0,\quad c=5,\quad B_0=5\times10^{-4},
\quad \kappa_{\mathbf Q}=\kappa_u=4.
\end{equation*}
The amplitudes of the initial tensor and density fields are both $0.45$.
In Figs.~\ref{fig:confined-circle}--\ref{fig:confined-hexagon}, the color map
shows $|\mathbf Q|_F$, black curves are the zero level sets of $u$ (the layer
interfaces), white segments show the principal director, and magenta curves
mark low-order regions.  The latter are interpreted only as candidate defect
cores; they are not, by themselves, a computation of topological charge.

For the circular test, we impose homeotropic anchoring with strength $W=60$
and use $q=10$.  The boundary-adapted initial state has its director along
the wall normal and has density
\begin{equation*}
u^0=0.45\cos(q|d(\mathbf x)|).
\end{equation*}
We compare radii $R=0.42L$ and $R=0.26L$ while keeping all other parameters
fixed.

\begin{figure}[htbp]
  \centering
  \includegraphics[width=0.7\textwidth]{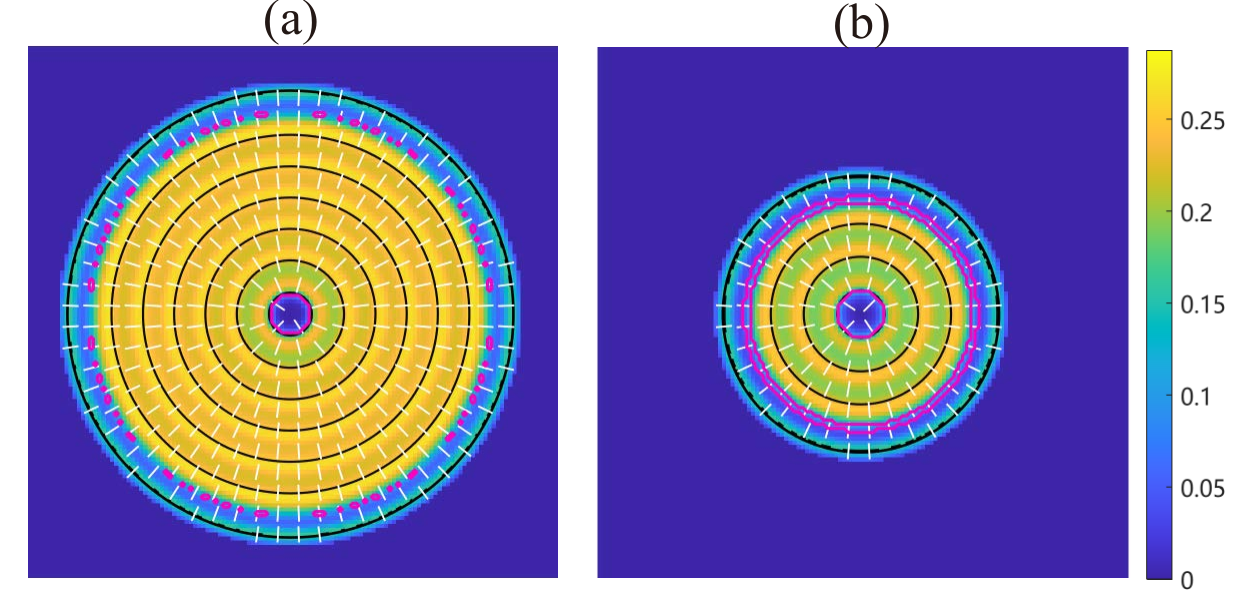}
  \caption{Final confined states at $T=15$ in circular regions with
  $q=10$ and homeotropic anchoring: (a) $R=0.42L$ and (b) $R=0.26L$.
  The common color scale represents $|\mathbf Q|_F$; black, white, and
  magenta curves represent layer interfaces, directors, and candidate
  low-order cores, respectively.}
  \label{fig:confined-circle}
\end{figure}

Both radii produce concentric layers and radial director alignment.  The
smaller cavity accommodates fewer preferred wavelengths, however, and the
low-order annulus occupies a visibly larger fraction of the interior.  This
comparison isolates the competition between the preferred layer spacing and
the finite size of the cavity.

The elliptical test uses semiaxes $0.42L$ and $0.27L$, $q=10$, and
homeotropic anchoring with $W=30$.  Two runs have identical physical and
numerical parameters but different initial states.  The tangential initial
state uses $\mathbf n^0=(-\sin\theta,\cos\theta)^T$ and a radially modulated
density, whereas the lamellar initial state uses a uniform director and
$u^0=0.45\cos(q(x-L/2))$.

\begin{figure}[htbp]
  \centering
  \includegraphics[width=0.7\textwidth]{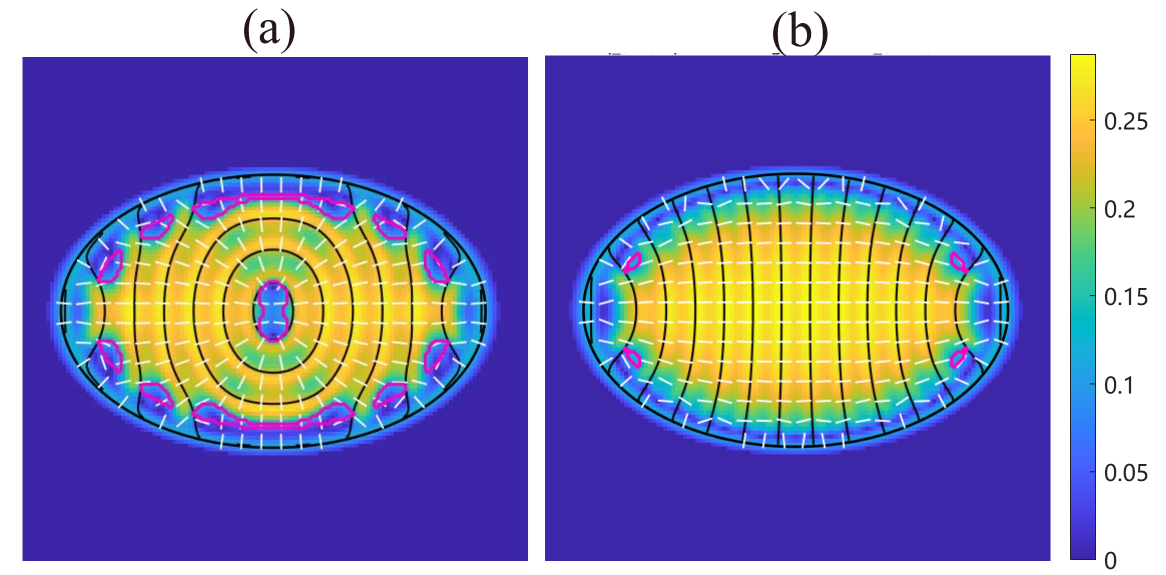}
  \caption{Initial-condition dependence in the same elliptical region at
  $T=15$: (a) tangential initial state and (b) lamellar initial state.  Here
  $q=10$, the semiaxes are $0.42L$ and $0.27L$, and the homeotropic anchoring
  strength is $W=30$.}
  \label{fig:confined-ellipse}
\end{figure}

The tangential initialization relaxes to nested curved layers with several
boundary-associated low-order regions, whereas the lamellar initialization
retains a nearly parallel family of layers and contains substantially fewer
low-order regions.  The two long-lived states therefore demonstrate
initial-condition-dependent metastability in the confined energy landscape;
the calculation does not assert uniqueness of the final state.

For the third geometry, the outer boundary is a regular hexagon of
circumradius $0.42L$.  We set $q=8$ and $W=60$ and insert at its center either
a circular hole of radius $0.14L$ or a regular pentagonal hole of the same
circumradius.  The director and density are again initialized from the
boundary-adapted state.

\begin{figure}[htbp]
  \centering
  \includegraphics[width=0.7\textwidth]{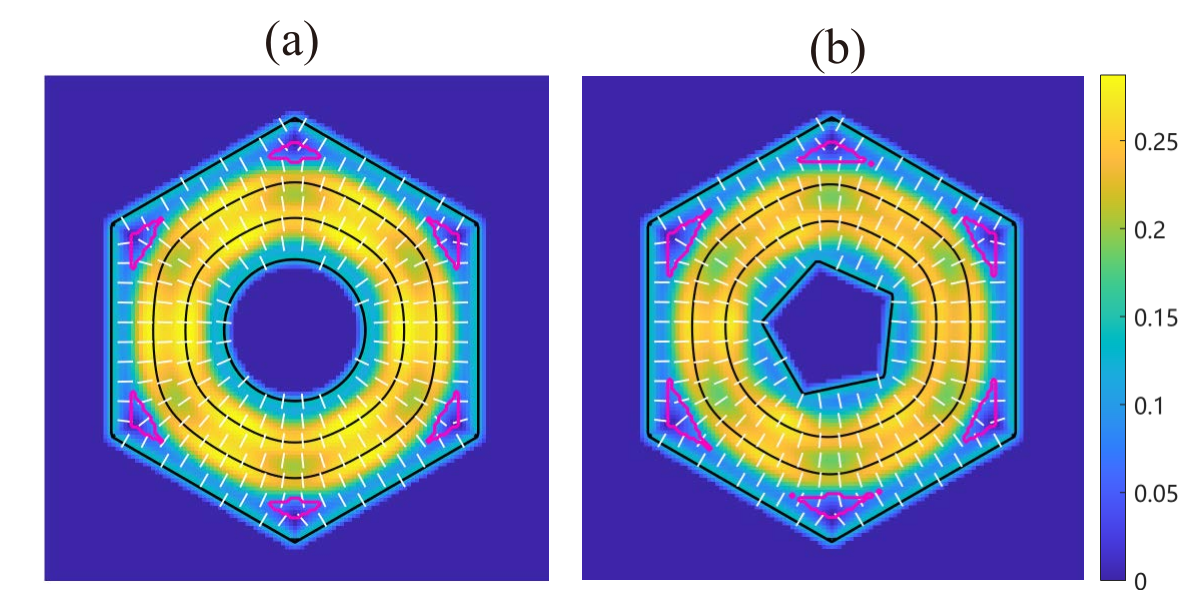}
  \caption{Final states at $T=15$ in a hexagonal outer region with (a) a
  circular inclusion of radius $0.14L$ and (b) a regular pentagonal
  inclusion of circumradius $0.14L$.  In both cases $q=8$ and the anchoring
  is homeotropic with $W=60$.}
  \label{fig:confined-hexagon}
\end{figure}

The outer corners pin localized distortions in both computations.  Replacing
the circular inclusion by a pentagon introduces additional preferred
directions at the inner boundary and produces a less uniform layer spacing
near the inclusion.  Away from the corners, the director remains closely
aligned with the layer normal, showing that the scheme resolves both the
bulk smectic order and the geometrically localized frustration.


\begin{figure}[htbp]
  \centering
  \includegraphics[width=0.8\textwidth]{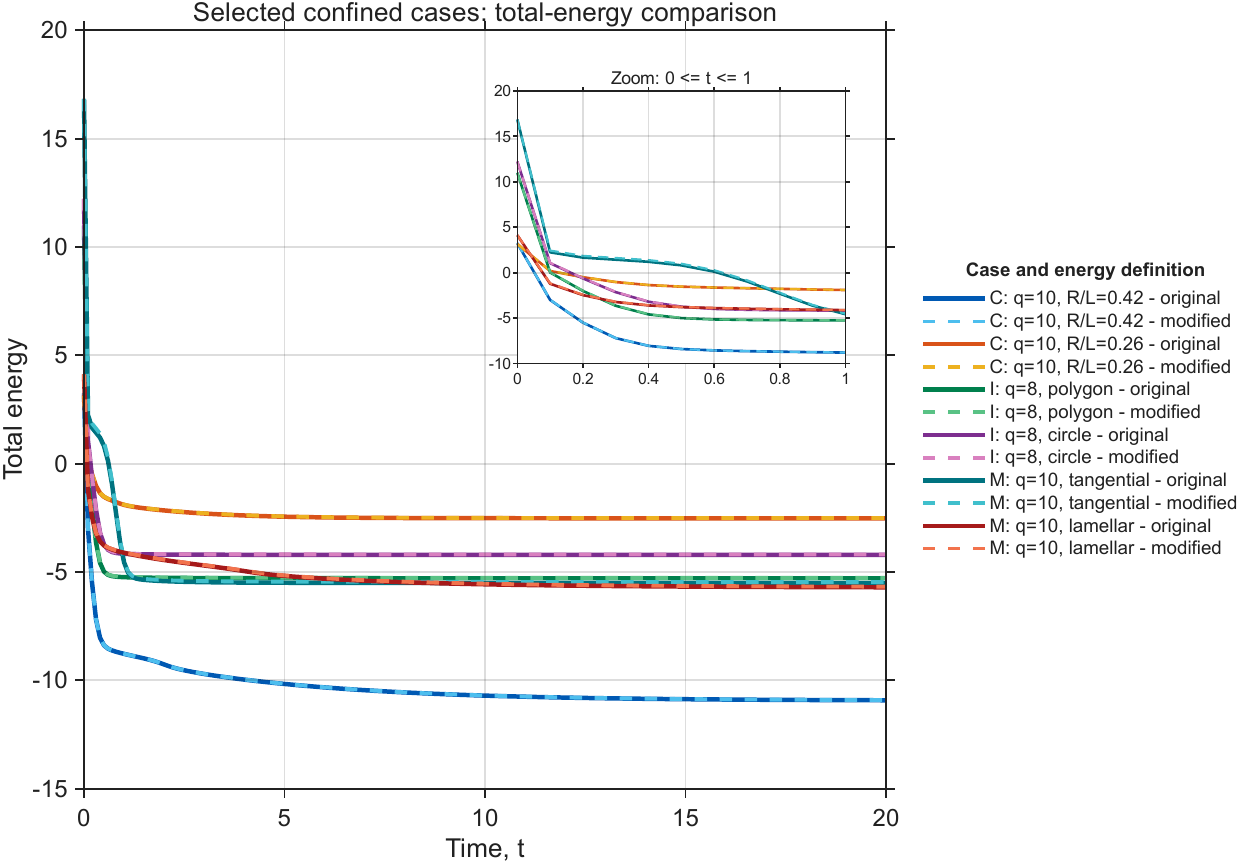}
  \caption{Original (solid) and modified (dashed) total energies for the
  confined experiments in Figs.~\ref{fig:confined-circle}--%
  \ref{fig:confined-hexagon}.  The curves labeled C correspond to the two
  circular radii in Fig.~\ref{fig:confined-circle}; M denotes the tangential
  and lamellar initial states in Fig.~\ref{fig:confined-ellipse}; and I
  denotes the circular and pentagonal inclusions in
  Fig.~\ref{fig:confined-hexagon}.  The inset enlarges the initial interval
  $0\le t\le1$.}
  \label{fig:confined-energy-comparison}
\end{figure}

Figure~\ref{fig:confined-energy-comparison} shows continuations to $T=20$ of
the six confined trajectories used in Figs.~\ref{fig:confined-circle}--%
\ref{fig:confined-hexagon}.  Every original and modified energy decreases
monotonically for the chosen time step.  The inset resolves the rapid initial
relaxation, after which the decay becomes substantially slower.  Within each
case, the original and modified curves remain nearly indistinguishable,
showing that the auxiliary energy closely tracks the physical energy
throughout the geometric reorganization.  Since these experiments have
different accessible areas, boundaries, and wall contributions, the vertical
ordering of their total energies is not interpreted as a thermodynamic
preference between different geometries.


\subsection{Three-dimensional screw-dislocation annihilation}
\label{subsec:three-dimensional-test}

Finally, we consider a neutral pair of oppositely charged screw
dislocations in the periodic cube $\Omega=(0,2\pi)^3$.  We use
\begin{equation*}
N=64,\quad h=2\pi/N,\quad \tau=10^{-2},\quad T=20,\quad
K=0.1,\quad \A=-1,\quad \B=1,\quad \C=2,
\end{equation*}
\begin{equation*}
a=-5,\quad b=0,\quad c=5,\quad q=3,\quad
B_0=2\times10^{-3},\quad \kappa_{\mathbf Q}=\kappa_u=4.
\end{equation*}
Let $(x_1,y_1)=(0.35L,0.5L)$ and
$(x_2,y_2)=(0.65L,0.5L)$, let
$\theta_j=\operatorname{atan2}(y-y_j,x-x_j)$ be evaluated with periodic
coordinate differences, and set
\begin{equation*}
\Theta=qz+\theta_1-\theta_2,
\qquad
\chi=\prod_{j=1}^2
\tanh\left(\frac{[(x-x_j)^2+(y-y_j)^2]^{1/2}}{r_c}\right),
\qquad r_c=0.22.
\end{equation*}
With $\mathbf n$ the normalized regularized phase normal, the initial data are
\begin{equation*}
u^0=0.45\chi\cos\Theta,
\qquad
\mathbf Q^0=s_+\chi
\left(\mathbf n\mathbf n^T-\tfrac13\mathbf I_3\right).
\end{equation*}
The factor $\chi$ regularizes the two cores and suppresses the order
parameter at their centers.

\begin{figure}[htbp]
  \centering
  \includegraphics[width=0.98\textwidth]{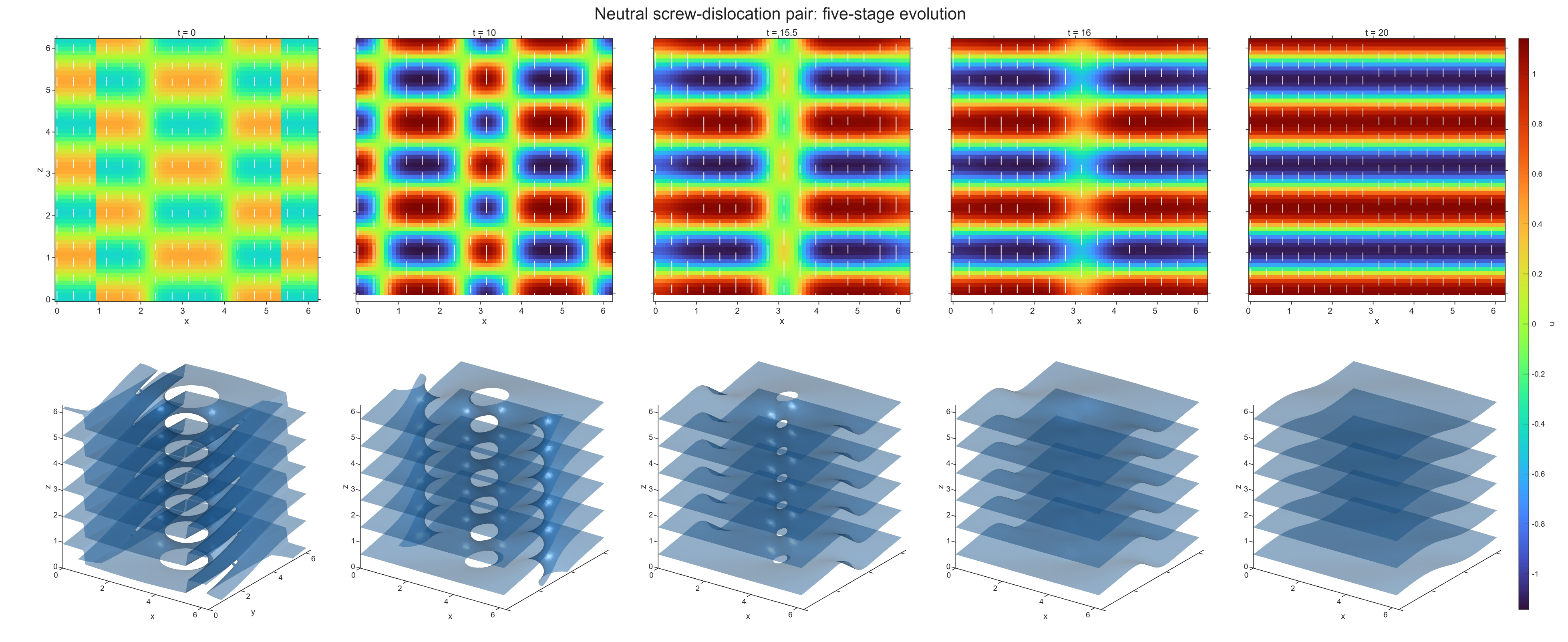}
  \caption{Five-stage evolution of a neutral screw-dislocation pair at
  $t=0,10,15.5,16$, and $20$.  The top row shows $x$--$z$ mid-plane slices
  of $u$ with the projected principal director, and the bottom row shows the
  three-dimensional zero-level surfaces of $u$.}
  \label{fig:three-dimensional-screw-pair}
\end{figure}

At early times, the two cores terminate and twist the otherwise layered
surfaces.  They approach slowly up to approximately $t=15.5$, followed by a
rapid annihilation and reconnection between $t=15.5$ and $t=16$.  By $t=20$
the defect cores have disappeared from the displayed slices and the
isosurfaces have recovered a nearly regular lamellar arrangement.  The slice
and isosurface views give mutually consistent descriptions of this
three-dimensional event.


\begin{figure}[htbp]
  \centering
  \includegraphics[width=0.8\textwidth]{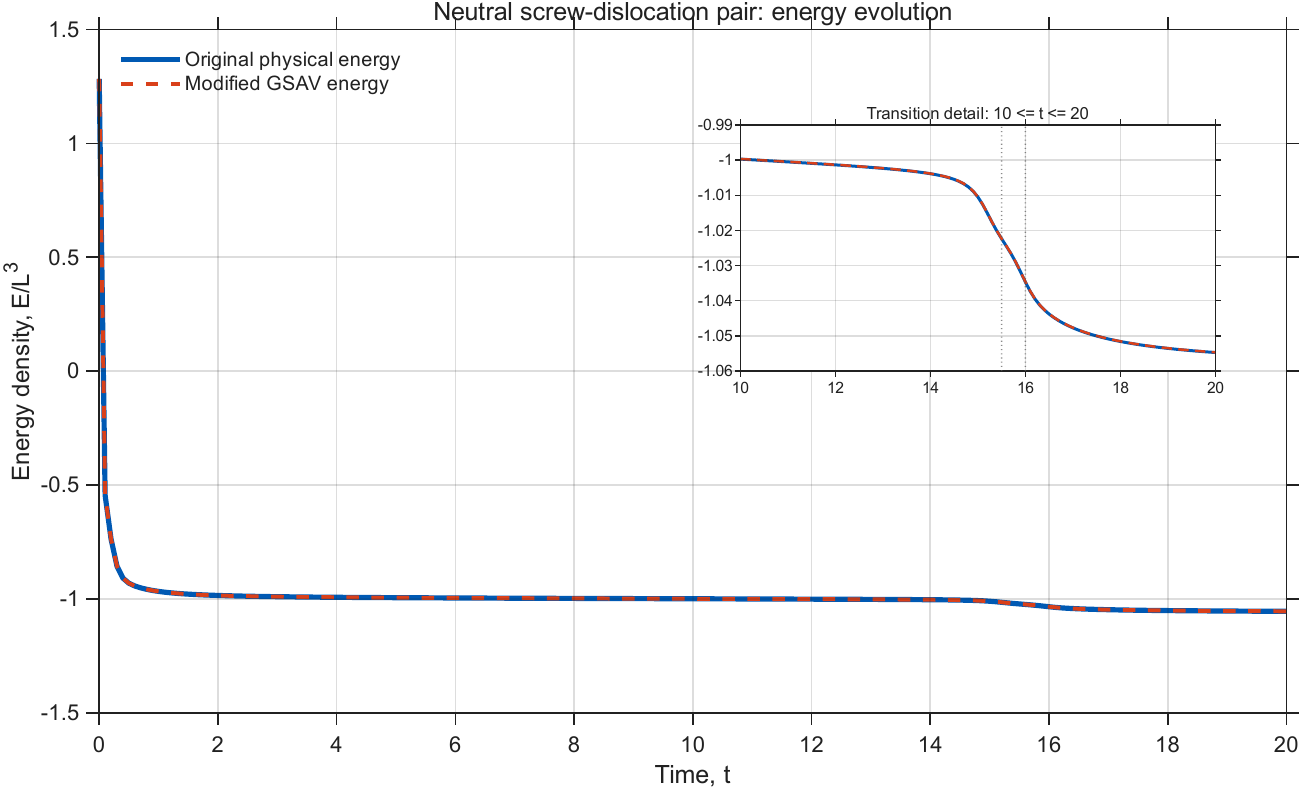}
  \caption{Original physical energy density $E_h/L^3$ and modified GSAV
  energy density during the three-dimensional screw-pair simulation.  The
  inset enlarges the transition interval $10\le t\le20$.}
  \label{fig:three-dimensional-energy-density}
\end{figure}

Both energy densities in
Fig.~\ref{fig:three-dimensional-energy-density} decrease for the chosen time
step and remain nearly indistinguishable on the scale of the plot.  The
additional energy release near $t=15.5$--$16$ coincides with the core
interaction and layer reconnection in
Fig.~\ref{fig:three-dimensional-screw-pair}.  This correlation provides a
dynamical consistency check: the most rapid morphological change is resolved
at the same time as the principal dissipative event.  Together, the
two- and three-dimensional experiments show that the GSAV--ETD2 method can
resolve wavelength selection, boundary-induced frustration, metastable
branch selection, and defect-driven layer reorganization while retaining the
predicted dissipative structure.

\section{Conclusion}\label{section7}
We have developed a second-order, structure-preserving scheme for the SmA
system. In particular, we recast the discrete-in-time exponential integrator
as a quasi-implicit scheme of backward Euler type. This formulation permits
a rigorous fully discrete error analysis and clarifies the connection between
ETD methods and backward-Euler-type formulations.

The resulting framework provides a practical numerical tool for
simulating the SmA phase. Future work will include the development of
higher-order GSAV-ETD2 schemes, the construction of methods that preserve the
physical eigenvalue constraints of the $\mathbf{Q}$-tensor, and extensions to more general
liquid-crystal hydrodynamic models.
\medskip

\noindent\textbf{Acknowledgments.} G. Ji is partially supported by the National Natural Science Foundation of China (Grant No. 12471363).

\appendix
\section{Proofs of auxiliary results}
\label{app:auxiliary_proofs}

\subsection{Proof of Lemma~\ref{lem_phi_ineq}}
\label{app:proof_phi_ineq}
\begin{proof}
By \eqref{eq3a},
 since $1-\varsigma \le 1$ for $\varsigma \in (0,1)$ and $e^{-\varsigma z} > 0$, it immediately follows that $\varphi_2(z) \le \varphi_1(z)$.
Moreover,
\[
0\le z\varphi_2(z)
=\frac{e^{-z}-1+z}{z}
=1-\varphi_1(z)\le 1,
\]
where the identity at $z=0$ is understood by continuity.

To prove $\varphi_1(z) \le 2\varphi_2(z)$, we use their explicit algebraic forms:
\[
\frac{1-e^{-z}}{z} \le 2 \frac{e^{-z}-1+z}{z^2}.
\]
For $z>0$, multiplying by $z^2$ and rearranging yields the equivalent condition $y(z) \ge 0$, where
\[
y(z) = z + (z+2)e^{-z} - 2.
\]
The case $z=0$ follows from the continuous extensions of $\varphi_1$ and $\varphi_2$.
Differentiation gives
\[
y'(z) = 1 - (z+1)e^{-z}, \quad y''(z) = z e^{-z}.
\]
For $z \ge 0$, $y''(z) \ge 0$, so $y'$ is nondecreasing. Since $y'(0) = 0$, we have $y'(z) \ge 0$ for all $z \ge 0$. Thus, $y$ is nondecreasing, and $y(0)=0$ implies that $y(z) \ge 0$ for all $z \ge 0$.

Define the auxiliary function $F(z)$ by
\[
F(z) = \frac{1}{2}\bar{\varphi}_2(z) - \bar{\varphi}_1(z) + \frac{1}{2} z.
\]
Substituting the definitions into $F(z)$ and factoring out $\frac{z}{2}$, we obtain
\[
\begin{aligned}
F(z) &= \frac{z}{2} \left( \frac{z}{e^{-z}-1+z} - \frac{2}{1-e^{-z}} + 1 \right) \\
&= \frac{z}{2} \frac{z(1-e^{-z}) - 2(e^{-z}-1+z) + (e^{-z}-1+z)(1-e^{-z})}{(e^{-z}-1+z)(1-e^{-z})}.
\end{aligned}
\]
Let $N(z)$ denote the numerator of the fraction. Expanding and collecting terms gives
\[
\begin{aligned}
N(z) &= (z - ze^{-z}) - (2e^{-z} - 2 + 2z) + (2e^{-z} - e^{-2z} - 1 + z - ze^{-z}) \\
&= 1 - e^{-2z} - 2ze^{-z}.
\end{aligned}
\]
To determine the sign of $N(z)$ for $z \ge 0$, we compute
\[
N'(z) = 2e^{-2z} - 2(e^{-z} - ze^{-z}) = 2e^{-z}(e^{-z} + z - 1).
\]
The standard exponential inequality $e^{-z} \ge 1-z$ gives $e^{-z}+z-1 \ge 0$. Hence, $N'(z) \ge 0$ for all $z \ge 0$. Since $N(0)=0$ and $N$ is nondecreasing, $N(z) \ge 0$ for $z \ge 0$.

For $z > 0$, the denominator $D(z) = (e^{-z}-1+z)(1-e^{-z})$ is strictly positive because $e^{-z} > 1-z$ and $1 > e^{-z}$.
Consequently, the fraction is nonnegative. Since $z/2\ge 0$, we obtain
\[
\frac{1}{2}\bar{\varphi}_2(z) - \bar{\varphi}_1(z) + \frac{1}{2} z \ge 0, \quad \forall z \ge 0.
\]
\end{proof}

\subsection{Proof of Lemma~\ref{lem_parabolic_regularity_bootstrap}}
\label{app:proof_parabolic_regularity}
\begin{proof}
    Differentiating \eqref{eq1_9_continuous_1} and \eqref{eq1_9_continuous_2} with respect to time and rearranging gives
\begin{equation}
    \mathcal L\mathbf Q_t
=
\partial_t\mNQz-\mathbf Q_{tt}, \quad
\mathcal D u_t
=
\partial_t\mNuz-u_{tt}.
\label{eq:ut_bootstrap_equation}
\end{equation}
Using the Sobolev embedding $H^2\hookrightarrow L^\infty$ for $d\le3$ and \eqref{eqa_19}, we obtain
\begin{align}
\|\partial_t\mNuz\|_{L^2}
&\le
C\left(
1
+\left\| \partial_t\bigl(\mathbf M:D^2u\bigr) \right\|_{L^2} + \left\| \partial_t\left[ \nabla\cdot\nabla\cdot \bigl(\mathbf Mu\bigr) \right] \right\|_{L^2}
\right)
\nonumber\\
&\le
C\left(
1
+\left\|
\mathbf M_t:D^2u+\mathbf M:D^2u_t
\right\|_{L^2}
+\left\|
\mathbf M_tu+\mathbf M u_t
\right\|_{H^2}
\right)
\nonumber\\
&\le
C\left(
1+\|\mathbf Q_t\|_{H^2}
+\|u_t\|_{H^2}
\right)
\le C.
\label{eq:Nu_time_bound}
\end{align}
By \eqref{eq:ut_bootstrap_equation} and the assumption
$u_{tt}\in L^\infty(0,T;H^2)$ in \eqref{eqa_19}, we have
\ba \label{eq:ut_H4_bound}
\|u_t\|_{H^4}
\le
C\|\mathcal D u_t\|_{L^2}
\le
C\bigl(\|\partial_t\mNuz\|_{L^2}
+
\|u_{tt}\|_{L^2}\bigr)
\le C.
\ed

Using the algebra property of $H^2$ for $d\le3$, the bound
$u\in L^\infty(0,T;H^6)$, and \eqref{eq:ut_H4_bound}, we obtain
\begin{align}
\|\partial_t\mNQz\|_{H^2}
&\le
C\left(
1+\|\mathbf Q_t\|_{H^2}
+\|u_tD^2u\|_{H^2}
+\|uD^2u_t\|_{H^2}
\right)
\nonumber\\
&\le
C\left(
1+\|\mathbf Q_t\|_{H^2}
+\|u_t\|_{H^4}
\right)
\le C.
\label{eq:NQ_time_bound}
\end{align}
It then follows from \eqref{eq:NQ_time_bound} and the
assumption
$\mathbf Q_{tt}\in L^\infty(0,T;\mathbf H^2)$ that
\ba\label{eq:Qt_H4_bound}
\|\mathbf Q_t\|_{H^4}
\le
C\|\mathcal L\mathbf Q_t\|_{H^2}
\le
C\bigl(\|\partial_t\mNQz\|_{H^2}
+
\|\mathbf Q_{tt}\|_{H^2}\bigr)
\le C.
\ed

Using the algebra properties of $H^2$ and $H^4$ for $d\le3$, and
combining \eqref{eqa_19} with \eqref{eq:ut_H4_bound} and
\eqref{eq:Qt_H4_bound}, we obtain
\begin{align}
\|\partial_t\mNuz\|_{H^2}
&\le
C\left(
1
+\left\|
\mathbf M_t:D^2u
+\mathbf M(\mathbf Q):D^2u_t
\right\|_{H^2}
+\left\|
\mathbf M_tu+\mathbf M(\mathbf Q)u_t
\right\|_{H^4}
\right)
\nonumber\\
&\le
C\left(
1
+\|\mathbf M_t\|_{H^4}\|u\|_{H^4}
+\|\mathbf M(\mathbf Q)\|_{H^4}\|u_t\|_{H^4}
\right)
\nonumber\\
&\le
C\left(
1+\|\mathbf Q_t\|_{H^4}
+\|u_t\|_{H^4}
\right)
\le C.
\label{eqacc}
\end{align}
It follows from \eqref{eqacc} and the assumption
$u_{tt}\in L^\infty(0,T;H^2)$ that
\ba\label{eq:ut_H6_bound}
\|u_t\|_{H^6}
\le
C
\|\mathcal D u_t\|_{H^2}
\le
C\bigl(\|\partial_t\mNuz\|_{H^2}
+
\|u_{tt}\|_{H^2}\bigr)
\le C.
\ed

Using the algebra property of $H^4$ for $d\le3$ and the regularity
bounds in \eqref{eqa_19}, we obtain
\begin{align}
\|\mNQz\|_{H^4}
&\le
C\left(
1+\|uD^2u\|_{H^4}
\right)
\nonumber\\
&\le
C\left(
1+\|u\|_{H^4}\|D^2u\|_{H^4}
\right)
\nonumber\\
&\le
C\left(
1+\|u\|_{H^4}\|u\|_{H^6}
\right)
\le C.
\label{eq:NQ_H4_bound}
\end{align}
Combining this estimate with \eqref{eq:Qt_H4_bound}, we find
\begin{align}
    \|\mathbf Q\|_{H^6}
\le
C\|\mathcal L\mathbf Q\|_{H^4}
\le
C\bigl(\|\mNQz\|_{H^4}
+
\|\mathbf Q_t\|_{H^4}\bigr)
\le C.\label{eq1a}
\end{align}

Using the algebra properties of $H^4$ and $H^6$ for $d\le3$, together
with the regularity bounds in \eqref{eqa_19}, we obtain
\begin{align}
\|\mNuz\|_{H^4}
&\le
C\left(
1
+\|\mathbf M(\mathbf Q):D^2u\|_{H^4}
+\|\mathbf M(\mathbf Q)u\|_{H^6}
\right)
\nonumber\\
&\le
C\left(
1
+\|\mathbf M(\mathbf Q)\|_{H^4}\|u\|_{H^6}
+\|\mathbf M(\mathbf Q)\|_{H^6}\|u\|_{H^6}
\right)
\nonumber\\
&\le
C\left(
1+\|\mathbf Q\|_{H^6}\|u\|_{H^6}
\right)
\le C.
\label{eq:Nu_H4_bound}
\end{align}
Using this estimate and \eqref{eq:ut_H6_bound}, we obtain
\begin{align}
\|u\|_{H^8}
\le
C\|\mathcal D u\|_{H^4}
\le
C\bigl(\|\mNuz\|_{H^4}
+
\|u_t\|_{H^4}\bigr)
\le C.\label{eq2a}
\end{align}

Finally, we estimate the second time derivatives of the nonlinear
terms $\mNQz$ and $\mNuz$.
The second time derivatives of the leading terms are
\ba \label{eq:NQtt_leading}
\partial_{tt}(uD^2u)
&=
u_{tt}D^2u
+
2u_tD^2u_t
+
uD^2u_{tt},
\\
\partial_{tt}\bigl(\mathbf M(\mathbf Q):D^2u\bigr)
&=
\mathbf M_{tt}:D^2u
+
2\mathbf M_t:D^2u_t
+
\mathbf M(\mathbf Q):D^2u_{tt},
\\
\partial_{tt}\left[
\nabla\cdot\nabla\cdot
\bigl(\mathbf M(\mathbf Q)u\bigr)
\right]
&=
\nabla\cdot\nabla\cdot
\Bigl(
\mathbf M_{tt}u
+
2\mathbf M_tu_t
+
\mathbf M(\mathbf Q)u_{tt}
\Bigr).
\ed
Using the regularity assumptions \eqref{eqa_19},
together with \eqref{eq1a} and \eqref{eq2a},
we deduce from
\eqref{eq:NQtt_leading} that
\begin{align}
\left\|\partial_{tt}(uD^2u)\right\|_{L^2}
+\left\|
\partial_{tt}\bigl(\mathbf M(\mathbf Q):D^2u\bigr)
\right\|_{L^2}
+\left\|
\partial_{tt}\left[
\nabla\cdot\nabla\cdot
\bigl(\mathbf M(\mathbf Q)u\bigr)
\right]
\right\|_{L^2}
\le C.
\end{align}
Estimating the remaining lower-order terms in the same manner, we conclude that
\[
\|\partial_{tt}\mNQz\|_{L^2}
+
\|\partial_{tt}\mNuz\|_{L^2}
\le C,
\]
which completes the proof.
\end{proof}

\bibliographystyle{siam}
\bibliography{S0362546X14002934}

\begin{thebibliography}{10}

\bibitem{Ball03052015}
{\sc J.~M. Ball and S.~J. Bedford}, {\em Discontinuous order parameters in liquid crystal theories}, Molecular Crystals and Liquid Crystals, 612 (2015), pp.~1--23.

\bibitem{ball2011orientability}
{\sc J.~M. Ball and A.~Zarnescu}, {\em Orientability and energy minimization in liquid crystal models}, Archive for Rational Mechanics and Analysis, 202 (2011), pp.~493--535.

\bibitem{baskaran2013convergence}
{\sc A.~Baskaran, J.~S. Lowengrub, C.~Wang, and S.~M. Wise}, {\em Convergence analysis of a second order convex splitting scheme for the modified phase field crystal equation}, SIAM Journal on Numerical Analysis, 51 (2013), pp.~2851--2873.

\bibitem{biscari2007landau}
{\sc P.~Biscari, M.~C. Calderer, and E.~M. Terentjev}, {\em Landau--de gennes theory of isotropic-nematic-smectic liquid crystal transitions}, Physical Review E—Statistical, Nonlinear, and Soft Matter Physics, 75 (2007), p.~051707.

\bibitem{chen1976landau}
{\sc J.-h. Chen and T.~Lubensky}, {\em {Landau-Ginzburg} mean-field theory for the nematic to {Smectic-C} and nematic to {Smectic-A} phase transitions}, Physical Review A, 14 (1976), p.~1202.

\bibitem{deGennes1974}
{\sc P.-G. de~Gennes}, {\em The Physics of Liquid Crystals}, Oxford University Press, Oxford, 1974.

\bibitem{du2018stabilized}
{\sc Q.~Du, L.~Ju, X.~Li, and Z.~Qiao}, {\em Stabilized linear semi-implicit schemes for the nonlocal cahn--hilliard equation}, Journal of Computational Physics, 363 (2018), pp.~39--54.

\bibitem{du2019}
\leavevmode\vrule height 2pt depth -1.6pt width 23pt, {\em Maximum-principle-preserving exponential time differencing schemes for the nonlocal {Allen--Cahn} equation}, SIAM Journal on Numerical Analysis, 57 (2019), pp.~875--898.

\bibitem{du2021}
\leavevmode\vrule height 2pt depth -1.6pt width 23pt, {\em Maximum bound principles for a class of semilinear parabolic equations and exponential time-differencing schemes}, SIAM Review, 63 (2021), pp.~317--359.

\bibitem{du1991numerical}
{\sc Q.~Du and R.~A. Nicolaides}, {\em Numerical analysis of a continuum model of phase transition}, SIAM Journal on Numerical Analysis, 28 (1991), pp.~1310--1322.

\bibitem{E1997modeling}
{\sc W.~E}, {\em Nonlinear continuum theory of smectic-a liquid crystals}, Archive for Rational Mechanics and Analysis, 137 (1997), pp.~159--175.

\bibitem{eyre1998unconditionally}
{\sc D.~J. Eyre}, {\em Unconditionally gradient-stable time marching for the {Cahn--Hilliard} equation}, MRS Online Proceedings Library (OPL), 529 (1998), p.~39.

\bibitem{fei2018isotropic}
{\sc M.~Fei, W.~Wang, P.~Zhang, and Z.~Zhang}, {\em On the isotropic--nematic phase transition for the liquid crystal}, Peking Mathematical Journal, 1 (2018), pp.~141--219.

\bibitem{feng2013stabilized}
{\sc X.~Feng, T.~Tang, and J.~Yang}, {\em Stabilized {Crank--Nicolson}/{Adams--Bashforth} schemes for phase-field models}, East Asian Journal on Applied Mathematics, 3 (2013), pp.~59--80.

\bibitem{furihata2001stable}
{\sc D.~Furihata}, {\em A stable and conservative finite difference scheme for the {Cahn--Hilliard} equation}, Numerische Mathematik, 87 (2001), pp.~675--699.

\bibitem{han2015microscopic}
{\sc J.~Han, Y.~Luo, W.~Wang, P.~Zhang, and Z.~Zhang}, {\em From microscopic theory to macroscopic theory: a systematic study on modeling for liquid crystals}, Archive for Rational Mechanics and Analysis, 215 (2015), pp.~741--809.

\bibitem{hu2026structurepreservinggsavexponentialintegrator}
{\sc W.~Hu, G.~Ji, and X.~Li}, {\em A structure-preserving gsav exponential integrator for smectic-a liquid crystals}, 2026.

\bibitem{huangGlobalWellposednessDynamical2015}
{\sc J.~Huang and S.~Ding}, {\em Global well-posedness for the dynamical {{Q}-tensor} model of liquid crystals}, Science China Mathematics, 58 (2015), pp.~1349--1366.

\bibitem{huang2024errorestimateorderenergy}
{\sc J.~Huang, X.~Li, and G.~Ji}, {\em Error estimate for the first-order energy-stable scheme of the {Q}-tensor nematic model}, 2024.

\bibitem{izzo2020landau}
{\sc D.~Izzo and M.~J. De~Oliveira}, {\em {Landau} theory for isotropic, nematic, {Smectic-A}, and {Smectic-C} phases}, Liquid Crystals, 47 (2020), pp.~99--105.

\bibitem{Ji2020BDF2}
{\sc G.~Ji}, {\em A {BDF2} energy-stable scheme for a general tensor-based model of liquid crystals}, East Asian Journal on Applied Mathematics, 10 (2020), pp.~57--71.

\bibitem{jiang2022improving}
{\sc M.~Jiang, Z.~Zhang, and J.~Zhao}, {\em Improving the accuracy and consistency of the scalar auxiliary variable (sav) method with relaxation}, Journal of Computational Physics, 456 (2022), p.~110954.

\bibitem{ju2022generalized}
{\sc L.~Ju, X.~Li, and Z.~Qiao}, {\em Generalized {SAV}-exponential integrator schemes for {Allen--Cahn}-type gradient flows}, SIAM Journal on Numerical Analysis, 60 (2022), pp.~1905--1931.

\bibitem{liu2007dynamics}
{\sc C.~Liu, J.~Shen, and X.~Yang}, {\em Dynamics of defect motion in nematic liquid crystal flow: modeling and numerical simulation}, Communications in Computational Physics, 2 (2007), pp.~1184--1198.

\bibitem{liu2025maximum}
{\sc Y.~Liu, C.~Quan, and D.~Wang}, {\em On the maximum bound principle and energy dissipation of exponential time differencing methods for the matrix-valued allen--cahn equation}, IMA Journal of Numerical Analysis, 45 (2025), pp.~3342--3377.

\bibitem{majumdar2010equilibrium}
{\sc A.~Majumdar}, {\em Equilibrium order parameters of nematic liquid crystals in the {Landau-de Gennes} theory}, European Journal of Applied Mathematics, 21 (2010), pp.~181--203.

\bibitem{pazy2012semigroups}
{\sc A.~Pazy}, {\em Semigroups of linear operators and applications to partial differential equations}, Springer Science \& Business Media, 2012.

\bibitem{pevnyi2014modeling}
{\sc M.~Y. Pevnyi, J.~V. Selinger, and T.~J. Sluckin}, {\em Modeling smectic layers in confined geometries: Order parameter and defects}, Physical Review E, 90 (2014), p.~032507.

\bibitem{shen2018scalar}
{\sc J.~Shen, J.~Xu, and J.~Yang}, {\em The scalar auxiliary variable (sav) approach for gradient flows}, Journal of Computational Physics, 353 (2018), pp.~407--416.

\bibitem{shen2019new}
\leavevmode\vrule height 2pt depth -1.6pt width 23pt, {\em A new class of efficient and robust energy stable schemes for gradient flows}, SIAM Review, 61 (2019), pp.~474--506.

\bibitem{shen2010numerical}
{\sc J.~Shen and X.~Yang}, {\em Numerical approximations of {Allen--Cahn} and {Cahn--Hilliard} equations}, Discrete Contin. Dyn. Syst., 28 (2010), pp.~1669--1691.

\bibitem{shi2025modified}
{\sc B.~Shi, Y.~Han, C.~Ma, A.~Majumdar, and L.~Zhang}, {\em A modified {Landau--de Gennes} theory for smectic liquid crystals: phase transitions and structural transitions}, SIAM Journal on Applied Mathematics, 85 (2025), pp.~821--847.

\bibitem{wang2021modelling}
{\sc W.~Wang, L.~Zhang, and P.~Zhang}, {\em Modelling and computation of liquid crystals}, Acta Numerica, 30 (2021), pp.~765--851.

\bibitem{xia2023variational}
{\sc J.~Xia and P.~E. Farrell}, {\em Variational and numerical analysis of a {Q}-tensor model for {Smectic-A} liquid crystals}, ESAIM: Mathematical Modelling and Numerical Analysis, 57 (2023), pp.~693--716.

\bibitem{xia2024simple}
{\sc J.~Xia and Y.~Han}, {\em Simple tensorial theory of smectic c liquid crystals}, Physical Review Research, 6 (2024), p.~033232.

\bibitem{xia2021structural}
{\sc J.~Xia, S.~MacLachlan, T.~J. Atherton, and P.~E. Farrell}, {\em Structural landscapes in geometrically frustrated smectics}, Physical Review Letters, 126 (2021), p.~177801.

\bibitem{xu2006stability}
{\sc C.~Xu and T.~Tang}, {\em Stability analysis of large time-stepping methods for epitaxial growth models}, SIAM Journal on Numerical Analysis, 44 (2006), pp.~1759--1779.

\bibitem{xu2019efficient}
{\sc Z.~Xu, X.~Yang, H.~Zhang, and Z.~Xie}, {\em Efficient and linear schemes for anisotropic {Cahn--Hilliard} model using the stabilized-invariant energy quadratization ({S-IEQ}) approach}, Computer Physics Communications, 238 (2019), pp.~36--49.

\bibitem{yang2017linearly}
{\sc X.~Yang and D.~Han}, {\em Linearly first- and second-order, unconditionally energy-stable schemes for the phase-field crystal model}, Journal of Computational Physics, 330 (2017), pp.~1116--1134.

\bibitem{yang2020convergence}
{\sc X.~Yang and G.-D. Zhang}, {\em Convergence analysis for the invariant energy quadratization ({IEQ}) schemes for solving the {Cahn--Hilliard} and {Allen--Cahn} equations with general nonlinear potential}, Journal of Scientific Computing, 82 (2020), p.~55.

\end{thebibliography}

\end{document}